\documentclass[11pt]{amsart}
\usepackage[dvipsnames,usenames]{color}
\usepackage{hyperref}
\usepackage{graphicx, import}
\usepackage{epsfig}
\usepackage[latin1]{inputenc}
\usepackage{amsmath}
\usepackage{amsfonts}
\usepackage{amssymb}
\usepackage{amsthm}
\usepackage{amscd}
\usepackage{verbatim}
\usepackage{subfigure}
\usepackage{caption}
\usepackage{pinlabel}
\usepackage{stmaryrd}
\usepackage{enumerate, enumitem}
\usepackage{todonotes}
\usepackage{bm}
\usepackage{thmtools}
\usepackage{thm-restate}
\usepackage{lipsum}
\usepackage{setspace}
\usepackage{mathtools}
\usepackage[all]{xypic}
\usepackage[abs]{overpic}

\allowdisplaybreaks

\usepackage{mathdots}
\usepackage{tikz}
\usetikzlibrary{arrows}
\usetikzlibrary{decorations.pathreplacing}
\usepackage{verbatim}
\usetikzlibrary{cd}
\tikzset{taar/.style={double, double equal sign distance, -implies}}
\tikzset{amar/.style={->, dotted}}
\tikzset{dmar/.style={->, dashed}}
\tikzset{aar/.style={->, very thick}}

\usepackage[section]{placeins}

\newtheorem{theorem}{Theorem}[section]

\newtheorem{lemma}[theorem]{Lemma}
\newtheorem{proposition}[theorem]{Proposition}

\theoremstyle{definition}

\theoremstyle{remark}
\newtheorem{remark}[theorem]{Remark}
\newtheorem{example}[theorem]{Example}

\newcommand{\alphas}{\boldsymbol{\alpha}}
\newcommand{\betas}{\boldsymbol{\beta}}

\def\parmod{\pmod}

\def\F{\mathbb{F}}

\def\Z{\mathbb{Z}}

\def\del{\partial}

\def\Sym{\mathrm{Sym}}

\def\CFKi{\CFK^{\infty}}

\def\Im{\operatorname{Im}}

\def \gr {\operatorname{gr}}

\def\Cone{\operatorname{Cone}}

\def\co{\colon}

\def\Cone{\operatorname{Cone}}

\newcommand{\sunderline}[1]{\underline{#1\mkern-3mu}\mkern3mu }

\def\Vu {\overline{V}}
\def\Vl {\sunderline{V}}

\def\HF {\mathit{HF}}
\def\HFK {\mathit{HFK}}
\newcommand\HFKhat{\widehat{\HFK}}

\newcommand \HFm {\HF^-}

\def\CFK{\mathit{CFK}}

\newcommand{\real}{\tilde{\mathfrak{I}}} 

\newcommand{\Id}{\operatorname{Id}}

\author[K. Hendricks]{Kristen Hendricks}
\email{kristen.hendricks@rutgers.edu}
\thanks{KH was partially supported by NSF CAREER Grant DMS-2019396 and a Sloan Research Fellowship.}

\author[M. Issac]{Matthew Issac}
\email{mi213@scarletmail.rutgers.edu}
\thanks{MI was partially supported by NSF CAREER Grant DMS-2019396.}

\author[N. McConnell]{Nicholas McConnell}
\email{ncm78@scarletmail.rutgers.edu}
\thanks{NM was partially supported by NSF CAREER Grant DMS-2019396.}

\numberwithin{equation}{section}

\title{A note on the involutive invariants of certain pretzel knots}

\begin{document}
\maketitle
\begin{abstract} We compute the involutive knot invariants for pretzel knots of the form $P(-2,m,n)$ for $m\geq n \geq 3$. \end{abstract}
\tableofcontents
\section{Introduction}

Heegaard Floer homology is a collection of invariants of three-manifolds and knots and links within them introduced by Ozsv{\'a}th and Szab{\'o} \cite{OS3manifolds1,OS3manifolds2,OSknots} and in the knot case independently by Rasmussen \cite{RasmussenThesis} in the early 2000s. The knot version associates to a knot $K \subseteq S^3$ a $(\Z\oplus\Z)$-filtered, $\Z$-graded chain complex over $\F[U,U^{-1}]$ called $\CFK^{\infty}(K)$. This chain complex recovers the data of the classical Alexander polynomial \cite{OSknots} and detects the knot genus \cite{OS:four} and whether the knot is fibred \cite{Ghigginifibred, Nifibred}. Furthermore, a plethora of interesting invariants of knot concordance have been extracted from it \cite{OS:four, OSinteger, Rasmussenh, Homconcordance, OSS}.

In 2015, Hendricks and Manolescu \cite{HM:involutive} introduced a refinement of Heegaard Floer homology called involutive Heegaard Floer homology, which incorporates the data of a conjugation symmetry on the Heegaard Floer chain complexes. In the knot case, this takes the form of a skew-filtered automorphism
\[ \iota_K \co \CFK^{\infty}(K) \rightarrow \CFK^{\infty}(K)\]
\noindent which is order four up to filtered chain homotopy. From this additional data, they construct two new concordance invariants $\Vl_0(K)$ and $\Vu_0(K)$, which are analogs of a concordance invariant $V_0(K)$ from the non-involutive setting \cite{OSinteger, Rasmussenh,Peters}. These invariants are particularly interesting in that, unlike many concordance invariants from Heegaard Floer homology and the related knot homology theory Khovanov homology, they can take nonzero values on knots of finite concordance order. For example, $\Vl_0(4_1)=1$; the invariants therefore detect the nonsliceness of the figure eight knot. The authors give combinatorial computations of the involutive concordance invariants for $L$-space knots (which include the torus knots) and thin knots (which include alternating and quasi-alternating knots).

Following a similar strategy to Hendricks and Manolescu's computation for thin knots \cite[Section 8]{HM:involutive}, in this note we compute the involutive concordance invariants of $P(-2,m,n)$ pretzel knots for $m$ and $n$ odd, and their mirrors. The complexes $\CFK^{\infty}(P(-2,m,n))$ associated to these knots were computed by Goda, Matsuda, and Morifuji \cite{GMM}. The reader should compare the statement of our result below with \cite[Proposition 8.2]{HM:involutive}. We include the values of the ordinary concordance invariant $V_0(K)$ in the statement for ease of comparison.

\begin{theorem} \label{thm:main} Let $m,n$ be odd numbers such that $m\geq n\geq 3$. The involutive knot concordance invariants of the pretzel knots $K = P(-2,m,n)$  are as follows.

\begin{itemize}
\item If $m \not\equiv n \parmod{4}$, then $V_0(K) = \Vl_0(K) = 0$ and $\Vu_0(K) = -\frac{m+n}{4}$.
\item If $m \equiv n \parmod{4}$, then $V_0(K) = \Vl_0(K) = 0$ and $\Vu_0(K) = -\frac{m+n-2}{4}$.
\end{itemize}

\noindent Moreover, the involutive knot concordance invariants of the mirrors $\overline{K}=P(2,-m,-n)$ are as follows.

\begin{itemize}

\item If $m \not\equiv n \parmod{4}$, then $V_0(\overline{K}) = \Vl_0(\overline{K}) = \Vu_0(\overline{K}) = \frac{m+n}{4}$.
\item If $m \equiv n \equiv 3 \parmod{4}$, then $V_0(\overline{K}) = \Vl_0(\overline{K}) = \Vu_0(\overline{K}) = \frac{m+n-2}{4}$.
\item If $m \equiv n \equiv 1 \parmod{4}$, then $ \Vl_0(\overline{K}) = \frac{m+n+2}{4}$ and $V_0(\overline{K}) = \Vu_0(\overline{K}) = \frac{m+n-2}{4}$.
\end{itemize}
\end{theorem}

This computation comes from analyzing four essentially distinct cases for the structure of the chain complex $\CFKi(K)$, corresponding to the values of $m$ and $n$ modulo four, as we explain further in Section \ref{subsec:model}.

\begin{remark} In the case that $n=3$, the involutive concordance invariants were already known. The knots $P(-2,m,3)$ are mirrors of $L$-space knots \cite{OS05:lens}. Hendricks and Manolescu computed the involutive concordance invariants of $L$-space knots and their mirrors \cite[Section 7]{HM:involutive}; we review the results of their computation in Section 2. We include the case $n=3$ above for completeness.\end{remark}

\begin{remark} The computation of knot Floer homology for the pretzel knots discussed in this note is particularly simple because they are $(1,1)$ knots \cite{GMM}. In general, $(1,1)$ knots admit Heegaard diagrams depending only on a set of four integer parameters, whose knot Floer homology may be computed combinatorially \cite{Rasmussenexpository, Doyle, Racz}. A possible future research direction is to attempt to give a general, hopefully combinatorial, formula for the skew-filtered chain homotopy equivalence class of $\iota_K$ for knots admitting such diagrams. \end{remark}

\subsection*{Organization} This paper is organized as follows. In Section \ref{sec:background} we review some necessary background; in particular, in Section 2.1 we review some properties of Heegaard Floer homology for knots, and in Section 2.2 we recall involutive Heegaard Floer homology for knots and the construction of the involutive concordance invariants. In Section 3.1 we recall the Heegaard Floer complexes associated to the pretzel knots $P(-2,m,n)$ and carry out a convenient change of basis. In Section 3.2 we present the computation of the involutive invariants associated to these knots and conclude the proof of Theorem \ref{thm:main}.

\subsection*{Acknowledgments} This project was carried out during the Rutgers DIMACS REU in Summer 2020; the second and third authors thank the organizers of the REU for their support. We are also grateful to Jennifer Hom for comments on a draft of this paper. Finally, we thank the referee for many helpful comments and corrections.

\section{Background on Heegaard Floer homology and involutive Heegaard Floer homology} \label{sec:background}

In this section we recall some background on Heegaard Floer homology for knots and involutive Heegaard Floer homology for knots.

\subsection{Heegaard Floer homology for knots}

We begin by briefly reviewing the construction of knot Floer homology, after which we will give a more focused description of some of its algebraic aspects; for a more detailed look, see \cite{OSknots}, or \cite{Homsurvey} for an expository view.

Let $\F$ be the field of two elements. Recall that a \emph{doubly-pointed Heegaard diagram} is a tuple $\mathcal H = (\Sigma, \alphas, \betas, z, w)$ such that
\begin{itemize}
\item $\Sigma$ is a closed oriented surface of genus $g$;
\item $\alphas$ (respectively $\betas$) is a tuple $\{\alpha_1, \dots, \alpha_g\}$ of pairwise disjoint circles (respectively $\{\beta_1, \dots, \beta_g\}$) in $\Sigma$ which span a $g$-dimensional subspace of $H_1(\Sigma; \mathbb F)$.
\item The curves $\alpha_i$ and $\beta_j$ intersect transversely for all $i,j$
\item $w$ and $z$ are points in the complement of $\alphas$ and $\betas$
\end{itemize}

Momentarily ignoring the basepoint $z$, the tuple $(\Sigma, \alphas, \betas, w)$ specifies a 3-manifold $Y$ via thickening $\Sigma$ to $\Sigma \times [0,1]$, attaching thickened disks along each $\alpha_i \times \{0\}$ and $\beta_j \times \{1\}$, and capping off each of the two remaining $S^3$ boundary components with three-balls. A knot $K$ inside of $Y$ is determined by connecting $w$ to $z$ in the complement of the $\beta$-disks and $z$ to $w$ in the complement of the $\alpha$-disks. An example of a Heegaard diagram for the right-handed trefoil appears in Figure~\ref{fig:heegaard}.

\begin{figure}
\fontsize{8pt}{11pt}\selectfont
\begin{center}
\begingroup%
  \makeatletter%
  \providecommand\color[2][]{%
    \errmessage{(Inkscape) Color is used for the text in Inkscape, but the package 'color.sty' is not loaded}%
    \renewcommand\color[2][]{}%
  }%
  \providecommand\transparent[1]{%
    \errmessage{(Inkscape) Transparency is used (non-zero) for the text in Inkscape, but the package 'transparent.sty' is not loaded}%
    \renewcommand\transparent[1]{}%
  }%
  \providecommand\rotatebox[2]{#2}%
  \newcommand*\fsize{\dimexpr\f@size pt\relax}%
  \newcommand*\lineheight[1]{\fontsize{\fsize}{#1\fsize}\selectfont}%
  \ifx\svgwidth\undefined%
    \setlength{\unitlength}{124.20294327bp}%
    \ifx\svgscale\undefined%
      \relax%
    \else%
      \setlength{\unitlength}{\unitlength * \real{\svgscale}}%
    \fi%
  \else%
    \setlength{\unitlength}{\svgwidth}%
  \fi%
  \global\let\svgwidth\undefined%
  \global\let\svgscale\undefined%
  \makeatother%
  \begin{picture}(1,0.7373809)%
    \lineheight{1}%
    \setlength\tabcolsep{0pt}%
    \put(0.40120999,0.19445729){\color[rgb]{0,0,0}\makebox(0,0)[lt]{\lineheight{1.25}\smash{\begin{tabular}[t]{l}$a$\end{tabular}}}}%
    \put(0.42088235,0.29204368){\color[rgb]{0,0,0}\makebox(0,0)[lt]{\lineheight{1.25}\smash{\begin{tabular}[t]{l}$b$\end{tabular}}}}%
    \put(0.52677206,0.0289378){\color[rgb]{0,0,0}\makebox(0,0)[lt]{\lineheight{1.25}\smash{\begin{tabular}[t]{l}$c$\end{tabular}}}}%
    \put(0.38961193,0.09471422){\color[rgb]{0,0,0}\makebox(0,0)[lt]{\lineheight{1.25}\smash{\begin{tabular}[t]{l}$z$\end{tabular}}}}%
    \put(0.49420699,0.19715299){\color[rgb]{0,0,0}\makebox(0,0)[lt]{\lineheight{1.25}\smash{\begin{tabular}[t]{l}$w$\end{tabular}}}}%
    \put(0,0){\includegraphics[width=\unitlength,page=1]{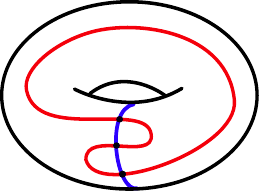}}%
  \end{picture}%
\endgroup%

\end{center}
\caption{A Heegaard diagram $\mathcal H$ for the right-handed trefoil. Here $g=1$ and the single $\alpha$ curve is drawn in red while the single $\beta$ curve is drawn in blue. The complex $\CFKi(\mathcal H)$ is generated over $\F[U,U^{-1}]$ by the three intersection points $a$, $b$, and $c$.}
\label{fig:heegaard}
\end{figure}
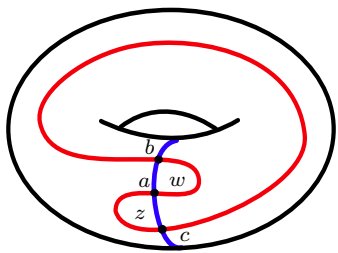

Given a Heegaard diagram $\mathcal H$ for $K \subset S^3$ as above, knot Floer homology associates to $\mathcal H$ a free finitely-generated $(\Z \oplus \Z)$-filtered chain complex $\CFKi(\mathcal H)$ over $\F[U,U^{-1}]$ \cite{OSknots, RasmussenThesis}. The construction of this complex uses the $g$-fold symmetric product $\Sym^g(\Sigma) = \Sigma^g/S_g$, where here $\Sigma^g$ denotes the ordinary product of $g$ copies of $\Sigma$ and the quotient is by the action of the symmetric group $S_g$. This points of this symmetric product are unordered $g$-tuples of points on the surface $\Sigma$; moreover, $\Sym^g(\Sigma)$ has the structure of a $g$-dimensional complex manifold. Inside this symmetric product one may consider the tori $\mathbb T_{\alphas} = \alpha_1 \times \dots \times \alpha_g$ and $\mathbb T_{\betas} = \beta_1 \times \dots \times \beta_g$. The generators of $\CFKi(\mathcal H)$ as an $\mathbb F[U,U^{-1}]$-complex are the finitely-many intersection points $x \in \mathbb T_{\alphas} \pitchfork \mathbb T_{\betas}$; concretely, this means that the generators consist of unordered $g$-tuples of intersection points between the curves $\alpha_i$ and $\beta_j$, such that each curve is used exactly once. The differential is defined by counting pseudoholomorphic curves in the symmetric product; for more detail, see \cite{OSknots}.

To the set of generators $\mathbb T_{\alphas} \pitchfork \mathbb T_{\betas}$, Ozsv{\'a}th and Szab{\'o} define maps $A, M: \mathbb T_{\alphas} \pitchfork \mathbb T_{\betas} \rightarrow \mathbb Z$, called the \emph{Alexander} and \emph{Maslov} (or \emph{homological}) gradings respectively. With these maps in hand, the generators of $\CFKi(\mathcal H)$ as an $\mathbb F$-vector space may be written
\[ U^{-i}x = [x, i, j] \text{ such that } x \in \mathbb T_{\alphas}\pitchfork \mathbb T_{\betas}, (i,j) \in \mathbb Z \oplus \mathbb Z, A(x)=j-i.\]
These generators are conventionally drawn on a plane; the element $[x;i,j]$ lies at $(i,j)$ and is said to have \emph{planar grading} $(i,j)$. One may then extend the functions $A$ and $M$ to $\CFKi(\mathcal H)$ via 
\[\gr([x;i,j])=M([x;i,j])=M(x)+2i \qquad \qquad A([x;i,j])=j\] 
The action of the $U$-variable is now given by
\[ U[x; i,j] = [x; i-1, j-1] \]
\noindent and the effect on the gradings of $U$-multiplication is 
\[M(U[x;i,j])=M([x;i-1,j-1])=M([x;i,j])-2\] \[A(U[x;i,j])=A([x;i-1,j-1])=j-1.\]

\noindent The $(i,j)$ level of the $\Z \oplus \Z$ filtration is $\mathcal F_{(i,j)} = \{[x;i',j']\in \CFKi(K): (i',j')\leq (i,j)\}$, where $\Z \oplus \Z$ is given the dictionary order. The differential $\partial$ on $\CFKi(K)$ respects the filtration and is $U$-equivariant; moreover if
\[
\partial([x;i,j]) = \sum [y; i',j']
\]
where each $[y;i'j']$ appears at most once (that is, there are no cancelling pairs in the expression) then for each $[y; i',j']$ we have \[M([y;i',j'])=M([x;i,j])-1.\]

While the construction of $\CFKi(\mathcal H)$ requires a choice of Heegaard diagram, Ozsv{\'a}th and Szab{\'o} show that all such choices produce chain homotopy equivalent chain complexes \cite{OSknots}; indeed, work of Juh{\'a}sz, Thurston, and Zemke \cite{JTNaturality} shows that these chain homotopies are themselves canonical up to homotopy, from which it follows that there is a well-defined filtered chain homotopy equivalence class of complexes $\CFKi(K)$. Throughout the paper, we will generally take some representative for the filtered chain homotopy equivalence class of $\CFKi(K)$; in some cases, such as Example~\ref{ex:fig8} below, this representative will not be the chain complex associated to any Heegaard diagram for the knot.

The homology $H_*(\CFKi(K))$ is always isomorphic to $\mathbb F[U,U^{-1}]$ \cite{OS3manifolds1, OSknots, RasmussenThesis}.

\begin{example} A representative for the filtered chain homotopy equivalence class of the knot Floer complex associated to the right-handed trefoil is shown in Figure \ref{fig:rightTrefoil}; in fact, it is exactly $\CFKi(\mathcal H)$ for the Heegaard diagram $\mathcal H$ of Figure~\ref{fig:heegaard}. As an $\F[U,U^{-1}]$-module, it has three generators $a=[a;0,0]$ in homological grading $-1$, $b=[b; 0,1]$ in homological grading $0$, and $c=[c;0,-1]$ in homological grading $-2$, with differential given by 

\begin{align*}
\partial a = Ub+c \qquad \qquad \partial b=\partial c = 0.
\end{align*}

\noindent The homology of the chain complex is generated over $\F[U,U^{-1}]$ by $[b]=[U^{-1}c]$.

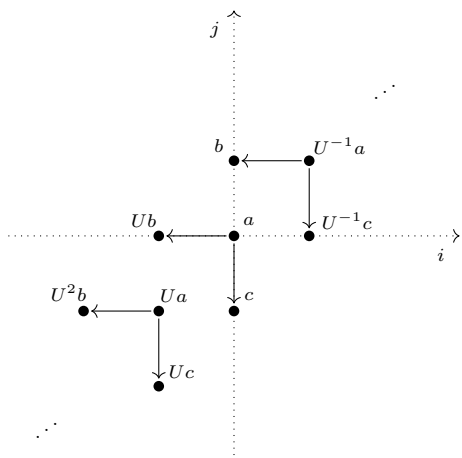
\begin{figure}
\begin{tikzpicture}\tikzstyle{every node}=[font=\tiny] 
\path[->][dotted](0,-3)edge(0,3);
\path[->][dotted](-3,0)edge(3,0);
\node at (-.25,2.75){$j$};
\node at (2.75,-.25){$i$};

\fill(0,1)circle [radius=2pt];
\fill(1,1)circle [radius=2pt];
\fill(1,0)circle [radius=2pt];

\node(y1)at (-.2,1.2){$b$};
\node(y1)at (1.4,1.2){$U^{-1}a$};
\node(y1)at (1.5,.2){$U^{-1}c$};

\path[->](.9,1)edge(.1,1);
\path[->](1,.9)edge(1,.1);

\fill(-1,0)circle [radius=2pt];
\fill(0,0)circle [radius=2pt];
\fill(0,-1)circle [radius=2pt];

\node(y1)at (-1.2,0.2){$Ub$};
\node(y1)at (0.2,0.2){$a$};
\node(y1)at (0.2,-0.8){$c$};

\path[->](-0.1,0)edge(-0.9,0);
\path[->](0,-0.1)edge(0,-0.9);

\fill(-2,-1)circle [radius=2pt];
\fill(-1,-1)circle [radius=2pt];
\fill(-1,-2)circle [radius=2pt];

\node(y1)at (-2.2,-0.75){$U^2b$};
\node(y1)at (-0.8,-0.8){$Ua$};
\node(y1)at (-0.7,-1.8){$Uc$};

\path[->](-1.1,-1)edge(-1.9,-1);
\path[->](-1,-1.1)edge(-1,-1.9);

\node(y1)at (2,2){$\iddots$};
\node(y1)at (-2.5,-2.5){$\iddots$};

\end{tikzpicture}
\caption{A representative for the filtered chain homotopy equivalence class of the knot Floer complex associated to the right-handed trefoil.}
\label{fig:rightTrefoil}
\end{figure}
\end{example}

\begin{example} \label{ex:fig8} A representative for the filtered chain homotopy equivalence class of the knot Floer complex associated to the figure-eight knot is shown in Figure \ref{fig:figureEight}. As an $\F[U,U^{-1}]$-module, it has five generators $a=[a;0,0]$ in homological grading $0$, $b=[b; -1,0]$ in homological grading $-1$, $c=[c;0,-1]$ in homological grading $-1$, $e=[e;0,0]$ in homological grading $0$, and $x=[x;0,0]$ in homological grading $0$ with nonzero differentials given by 

\begin{align*}
\partial a = b+c \qquad \qquad \partial b = \partial c = Ue \qquad \qquad \partial e = \partial x = 0.
\end{align*}

\noindent The homology of the chain complex is generated over $\mathbb F[U,U^{-1}]$ by $[x]$.

\begin{figure} 
\begin{tikzpicture}\tikzstyle{every node}=[font=\tiny] 
\path[->][dotted](0,-3)edge(0,3);
\path[->][dotted](-3,0)edge(3,0);

\node at (-.25,2.75){$j$};
\node at (2.75,-.25){$i$};

\fill(2,2)circle [radius=2pt];
\fill(2.15,2.15)circle [radius=2pt];

\path[->](2.15,2.65)edge(2.15,2.25);
\path[->](2.65,2.15)edge(2.25,2.15);

\fill(1,1)circle [radius=2pt];
\fill(1.15,1.15)circle [radius=2pt];
\fill(1.15,1.85)circle [radius=2pt];
\fill(1.85,1.15)circle [radius=2pt];
\fill(1.85,1.85)circle [radius=2pt];

\node(y1)at (1.4,0.8){$U^{-1}x$};
\node(y1)at (.7,1.3){$U^{-1}e$};

\path[->](1.15,1.75)edge(1.15,1.25);
\path[->](1.85,1.75)edge(1.85,1.25);
\path[->](1.75,1.15)edge(1.25,1.15);
\path[->](1.75,1.85)edge(1.25,1.85);

\fill(0,0)circle [radius=2pt];
\fill(.15,.15)circle [radius=2pt];
\fill(.15,.85)circle [radius=2pt];
\fill(.85,.15)circle [radius=2pt];
\fill(.85,.85)circle [radius=2pt];

\node(y1)at (.2,-.2){$x$};
\node(y1)at (.3,.3){$e$};
\node(y1)at (-.05,1.05){$U^{-1}b$};
\node(y1)at (1.05,-.05){$U^{-1}c$};
\node(y1)at (-.3,-.3){$a$};

\path[->](.15,.75)edge(.15,.25);
\path[->](.85,.75)edge(.85,.25);
\path[->](.75,.15)edge(.25,.15);
\path[->](.75,.85)edge(.25,.85);

\fill(-1,-1)circle [radius=2pt];
\fill(-0.85,-0.85)circle [radius=2pt];
\fill(-0.85,-0.15)circle [radius=2pt];
\fill(-0.15,-0.85)circle [radius=2pt];
\fill(-0.15,-0.15)circle [radius=2pt];

\node(y1)at (-0.8,-1.2){$Ux$};
\node(y1)at (-1.05,0.05){$b$};
\node(y1)at (0.05,-1.05){$c$};

\path[->](-0.85,-0.25)edge(-0.85,-0.75);
\path[->](-0.15,-0.25)edge(-0.15,-0.75);
\path[->](-0.25,-0.85)edge(-0.75,-0.85);
\path[->](-0.25,-0.15)edge(-0.75,-0.15);

\fill(-1.15,-1.15)circle [radius=2pt];

\path[->](-1.15,-1.25)edge(-1.15,-1.75);
\path[->](-1.25,-1.15)edge(-1.75,-1.15);

\node(y1)at (2.75,2.75){$\iddots$};
\node(y1)at (-2,-2){$\iddots$};

\end{tikzpicture}
\caption{A representative for the filtered chain homotopy equivalence class of the knot Floer complex associated to the figure-eight knot.}
\label{fig:figureEight}
\end{figure}
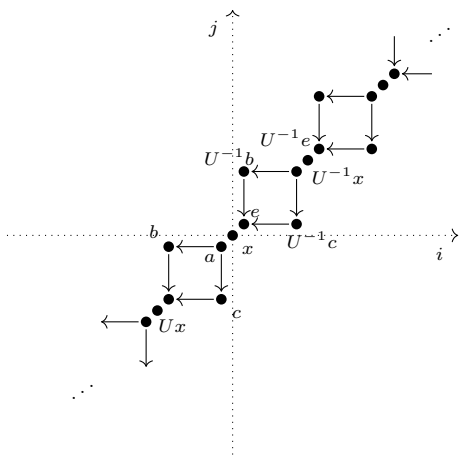
\end{example}

If $K$ is a knot and $\overline K$ is its mirror, then $\CFK^\infty(\overline K)$ is the dual complex $\CFK^\infty(K)^*$ over $\F [U,U^{-1}]$.

\begin{example} One may obtain a representative for the filtered chain homotopy equivalence class of the knot Floer complex associated to the left-handed trefoil by dualizing the complex for the right-handed trefoil shown in Figure~\ref{fig:rightTrefoil}; the result is shown in Figure \ref{fig:leftTrefoil}. As an $\F[U,U^{-1}]$-module, it has three generators $a=[a;0,0]$ in homological grading $1$, $b=[b; 0,-1]$ in homological grading $0$, and $c=[c;0,1]$ in homological grading $2$, with differential specified by

\begin{align*}
\partial b = Ua \qquad \qquad \partial c = a \qquad \qquad \partial a = 0.
\end{align*}

\noindent The homology of the complex is generated over $\F[U,U^{-1}]$ by $[b+Uc]$.

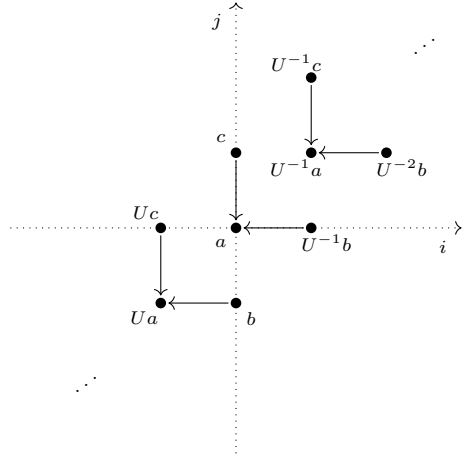
\begin{figure}
\begin{tikzpicture}\tikzstyle{every node}=[font=\tiny] 
\path[->][dotted](0,-3)edge(0,3);
\path[->][dotted](-3,0)edge(3,0);

\node at (-.25,2.75){$j$};
\node at (2.75,-.25){$i$};

\fill(1,2)circle [radius=2pt];
\fill(1,1)circle [radius=2pt];
\fill(2,1)circle [radius=2pt];

\node(y1)at (0.8,2.2){$U^{-1}c$};
\node(y1)at (0.8,0.8){$U^{-1}a$};
\node(y1)at (2.2,0.8){$U^{-2}b$};

\path[->](1,1.9)edge(1,1.1);
\path[->](1.9,1)edge(1.1,1);

\fill(0,1)circle [radius=2pt];
\fill(0,0)circle [radius=2pt];
\fill(1,0)circle [radius=2pt];

\node(y1)at (-.2,1.2){$c$};
\node(y1)at (-.2,-.2){$a$};
\node(y1)at (1.2,-.2){$U^{-1}b$};

\path[->](0,.9)edge(0,.1);
\path[->](.9,0)edge(.1,0);

\fill(-1,0)circle [radius=2pt];
\fill(-1,-1)circle [radius=2pt];
\fill(0,-1)circle [radius=2pt];

\node(y1)at (-1.2,0.2){$Uc$};
\node(y1)at (-1.22,-1.2){$Ua$};
\node(y1)at (0.2,-1.2){$b$};

\path[->](-1,-0.1)edge(-1,-0.9);
\path[->](-0.1,-1)edge(-0.9,-1);

\node(y1)at (2.5,2.5){$\iddots$};
\node(y1)at (-2,-2){$\iddots$};

\end{tikzpicture}
\caption{A representative for the filtered chain homotopy equivalence class of the knot Floer complex associated to the left-handed trefoil.}\label{fig:leftTrefoil}
\end{figure}
\end{example}

We will have occasion to consider several special subsets of $\CFKi(K)$ which are chain complexes over either the ring $\F[U]$ or over the vector space $\F$. Given a subset $X \in \Z \oplus \Z$, let $CX$ denote the $\mathbb F$-vector space with basis consisting of elements with planar grading $(i,j) \in X$. Some examples of particular importance to us include

\begin{center}
\begin{itemize}
\item The complex $A_0^- = C\{i,j \leq 0\}$, consisting of the portion of $\CFKi(K)$ lying in the third quadrant of the plane, which has the structure of a chain complex over $\mathbb F[U]$.
\item $B^-_0 = C\{i \leq 0\}$, consisting of the portion of $\CFKi(K)$ lying in the second and third quadrants of the plane, which has the structure of a chain complex over $\mathbb F[U]$.
\item $C\{i=0\}= C\{i \leq 0\}/ C\{i < 0\}$, consisting of the portion of $\CFKi(K)$ lying on the $j$-axis, which has the structure of a chain complex over $\mathbb F$.
\end{itemize}
\end{center}

Let us begin by considering the first two complexes. The homology $H_*(A_0^-)$ always admits a (noncanonical) decomposition as a direct sum
\[\F[U]\oplus \left(\bigoplus_{i=1}^k \F[U]/U^{n_i}\right)\] 
\noindent for some natural numbers $1 \leq n_1 \leq \dots \leq n_k$; this follows from the fact that the homology of $\CFK^{\infty}(K)$ is the Heegaard Floer three-manifold invariant $\HF^{\infty}(S^3) \simeq \mathbb F[U,U^{-1}]$ \cite{OS3manifolds1,OSknots, RasmussenThesis}. Moreover, the homology $H_*(B_0^-)$ is isomorphic to a copy of $\F[U]$ with the property that $\gr(1)=0$. (In particular, $H_*(B_0^-)$ is the three-manifold Heegaard Floer invariant $\HFm(S^3)$.) There is a chain map $v_0 \co A_0^- \rightarrow B_0^-$ given by inclusion. For sufficiently large $n$, the induced map on homology $U^nH_*(A_0^-) \to U^nH_*(B^-_0)$ is a nonzero map $\mathbb F[U] \rightarrow \mathbb F[U]$ which must therefore be given by multiplication by some $U^{V_0(K)}$, where $V_0(K)$ is a nonnegative integer. Since the map $v_0$ is grading-preserving, the integer $V_0(K)$ may be computed from the homological degree of the element $1$ in the ``tower'' summand $\mathbb F[U]$ in any decomposition of $H_*(A_0^-)$ into a direct sum of an $\mathbb F[U]$ summand and $U$-torsion summands. In other words, we see that
\[
V_0(K) = -\frac{1}{2}\max \{r : \exists \ x \in H_r(A_0^-), \forall \ n, \ U^nx\neq 0\} \]
\noindent where $H_r(A_0^-)$ denotes the homology in homological grading $r$. Peters \cite[Proposition 2.1]{Peters} and Rasmussen \cite[Theorem 2.3]{Rasmussenh} showed that $V_0(K)$ is an invariant of knot concordance, in Peters's case under the name $d(S^3_{+1}(K))=-2V_0(K)$ and in Rasmussen's case under the name $h_0(\overline{K})$. For more on the context of the concordance invariant $V_0(K)$, see eg \cite[Section 3.2.2]{Homsurvey}.

Now consider the complex $C\{i=0\}$, whose chain homotopy equivalence class is called $\widehat{\CFK}(K)$, which is naturally $\Z$-filtered by the Alexander filtration.  The homology of the associated graded object of $C\{i=0\}$ is 
\begin{equation}
\begin{split}
\HFKhat(K) & =  \bigoplus_{w \in \Z} \HFKhat(K,w) \\
 & = \bigoplus_{w \in \Z} H_*(C\{i = 0,j = w\})
\end{split}
\end{equation}
\noindent and is often referred to as the \emph{knot Floer homology}. If we include the homological grading $s$, we get a further decomposition
\[
\HFKhat(K) = \bigoplus_{w \in \Z, s \in \Z} \HFKhat_s(K,w).
\]
This theory is symmetric in the sense that $\HFKhat_s(K,w) \simeq \HFKhat_{s-2w}(K,-w)$ and furthermore it detects the knot genus via
\[
g(K)= \mathrm{max}\{w: \HFKhat(K,w) \neq 0\}.
\]
Finally, the graded Euler characteristic of the knot Floer homology is the Alexander polynomial of the knot, that is,
\[
\Delta_K(t) = \sum_{w} \chi(\HFKhat(K,w))t^w.
\]
We also consider the vertical and horizontal homologies associated to $\CFKi(K)$, as follows. Let $\partial = \sum_{i,j} \partial_{ij}$ where $\partial_{ij}$ is the term in the differential which decreases the two planar gradings by $i$ and $j$ respectively. Then the vertical differential is $\partial_{\mathrm{vert}}= \sum_{j} \partial_{0,j}$. The vertical homology is the $\F[U,U^{-1}]$-module \[H_*(\CFKi(K), \partial_{\mathrm{vert}})=\bigoplus_{i'} H_*(C\{i=i'\}).\] Likewise the horizontal differential is $\partial_{\mathrm{horz}}= \sum_{i} \partial_{i,0}$ and the horizontal homology is \[H^*(\CFKi(K), \partial_{\mathrm{horz}})=\bigoplus_{j'} H_*(C\{j=j'\}).\] For any knot $K$, the ranks of the vertical and horizontal homologies are one as $\mathbb F[U,U^{-1}]$-modules; alternately, if we ignore the action of $U$, the dimensions of each $H_*(C\{j=j'\})$ and $H_*(C\{i=i'\})$ are one as an $\mathbb F$-vector space.

The action of a Dehn twist around the knot $K$ in $S^3$ induces a filtered chain homotopy equivalence $\sigma$ on $\CFKi(K)$ with the property that $\sigma^2 \sim \Id$, where $\sim$ denotes filtered chain homotopy \cite{Sarkarmoving}. Zemke \cite{Zemkemoving} shows that $\sigma$ admits a simple computation, as follows. Let 

\[ \Phi = \sum_{\substack{i,j\geq 0 \\ i \text{ odd}}} \partial_{ij} \qquad \qquad \Psi = \sum_{\substack{i,j\geq 0 \\ j \text{ odd}}} \partial_{ij}.\]

\noindent We then have 

\[\sigma = \Id + U^{-1}\Phi \circ \Psi.\]

We briefly mention two special types of knots. A knot $K$ is said to be \emph{$L$-space} if it admits a positive surgery which is a three-manifold with Heegaard Floer homology of minimal rank (called a \emph{Heegaard Floer $L$-space}). Ozsv{\'a}th and Szab{\'o} showed \cite{OS05:lens} that if $K$ is an $L$-space knot, the filtered chain homotopy equivalence class of the complex $\CFKi(K)$ has a particularly simple representative. To describe it, we begin by introducing the following notation.

A \emph{positive staircase complex} is a $(\Z\oplus \Z)$-filtered $\F$-chain complex generated by elements $z_0,z_r^1,z_r^2$ where $r$ ranges from $1$ to some integer $v\geq 1$. The element $z_0$ has planar grading $(0,0)$; moreover, the planar gradings of $z^1_{v-2w}$ and $z^1_{v-(2w+1)}$ differ only in the $i$ grading and the planar gradings of $z^1_{v-(2w+1)}$ and $z^1_{v-(2w+2)}$ differ only in the $j$ grading. The planar gradings have the symmetry property that if $z_r^1$ lies at planar grading $(i,j)$ then $z_r^2$ lies at planar grading $(j,i)$.  If $v$ is even, the nonzero differentials in this complex are
\begin{align*}
\partial(z_r^s)=z_{r-1}^s+z_{r+1}^s \text{ for } 1<r<v, r \text{ odd } \qquad \qquad \qquad \partial(z_1^s) = z_0+ z_2^s
\end{align*}
whereas if $v$ is odd they are
\begin{align*}
\partial(z_r^s)=z_{r-1}^s+z_{r+1}^s \text{ for } r>0, r \text{ even} \qquad \qquad \qquad \partial(z_0)=z_1^1+z_1^2.
\end{align*}
\noindent Examples of the two possible forms of a positive staircase are shown in Figure~\ref{fig:Stair}.
The dual of a positive staircase is a \emph{negative staircase complex}. Again the generators are elements $z_0,z_r^1,z_r^2$ where $r$ ranges from $1$ to some integer $v\geq 1$. The element $z_0$ has planar grading $(0,0)$; moreover, the planar gradings of $z^1_{v-2w}$ and $z^1_{v-(2w+1)}$ differ only in the $j$ grading and the planar gradings of $z^1_{v-(2w+1)}$ and $z^1_{v-(2w+2)}$ differ only in the $i$ grading. The planar gradings have the symmetry property that if $z_r^1$ lies at planar grading $(i,j)$ then $z_r^2$ is at $(j,i)$.  If $v$ is even, the nonzero differentials of this complex are
\begin{align*}
&\partial(z_0)=z_1^1 + z_1^2 \\
&\partial(z_r^s)=z_{r-1}^s+z_{r+1}^s\text{ for } 0<r<v, r \text{ even } \\
&\partial(z_v^s)=z_{v-1}^s
\end{align*}
whereas if $v$ is odd they are
\begin{align*}
&\partial(z_1^s) = z_0 + z_2^s \\
&\partial(z_r^s)=z_{r-1}^s+z_{r+1}^s \text{ for } 1 < r < v, r \text{ odd} \\
&\partial(z_v^s)=z_{v-1}^s.
\end{align*}
\noindent Examples of the two possible forms of negative staircase complexes are shown in Figure~\ref{fig:Stair2}.

We now turn to the specific case of an $L$-space knot, which by work of  Ozsv{\'a}th and Szab{\'o} \cite{OS05:lens} has Alexander polynomial
\[
\Delta_K(t) = (-1)^v + \sum_{i=1}^v (-1)^{v-i}(t^{w_i} + t^{-w_i})
\]
for some sequence of positive integers $0<w_1<w_2<\dots <w_v = g(K)$. Let $\ell_i = w_i - w_{i-1}$ and let $n(K)$ denote the sum
\[
n(K) = w_v - w_{v-1} + \cdots + (-1)^{v-2}w_2 + (-1)^{v-1}w_1.
\]
Then the filtered chain homotopy equivalence class of $\CFKi(K)$ has a representative given by $C\otimes \F[U,U^{-1}]$ where $C$ is a positive staircase complex with generators $z_0,z_1^s,\dots, z_v^s$ such that the planar grading of $z_v^1$ is $(n(K), g(K)-n(K))$, and in general the gradings of $z^1_{v-2w}$ and $z^1_{v-(2w+1)}$ differ in the $i$ grading by $\ell_{v-2w}$ and the planar gradings of $z^1_{v-(2w+1)}$ and $z^1_{v-(2w-2)}$ differ in the $j$ grading by $\ell_{v-(2w+1)}$. We call the $\F$-complex $C$ a \emph{model complex}, or in this case a \emph{model complex for $\CFKi(K)$.} The complex associated to the mirror of an $L$-space knot is generated by the elements of the negative staircase produced by dualizing.

In the construction above, we see that the sum $n(K)$ appears in the positive staircase complex associated to the $L$-space knot as the sum of the lengths of the horizontal arrows in the top half of the staircase. For an arbitrary knot, if a representative of the filtered chain homotopy equivalence class of $\CFKi(K)$ contains a direct summand generated over $\mathbb F[U,U^{-1}]$ by a positive staircase complex, we let this fact be a definition of the quantity $n(K)$. If $\CFKi(K)$ contains a direct summand generated by a negative staircase complex, we let the sum of the lengths of the horizontal arrows in the top half of the dual of the complex be $n(K)$; this is equivalent to the sum of the lengths of the vertical arrows in the top half of the complex. Note that any staircase summand contributes a rank one summand $\mathbb F[U,U^{-1}]$ to the vertical homology and to the horizontal homology of the chain complex, and therefore there can be at most one staircase summand in any representation of the chain complex.

\begin{figure}[htb]
\begin{center}
\scalebox{.75}{
\begin{tikzpicture}\tikzstyle{every node}=[font=\tiny] 
\fill(3,-4)circle [radius=2pt];
\fill(3,-2)circle [radius=2pt];
\fill(2,-2)circle [radius=2pt];
\fill(2,-1)circle [radius=2pt];
\fill(1,0)circle [radius=2pt];
\fill(0,0)circle [radius=2pt];
\fill(0,1)circle [radius=2pt];
\fill(-1,2)circle [radius=2pt];
\fill(-2,2)circle [radius=2pt];
\fill(-2,3)circle [radius=2pt];
\fill(-4,3)circle [radius=2pt];

\node(xa)at (-2.17,1.75){$z_{v-2}^1$};
\node(xa)at (-1.83,3.25){$z_{v-1}^1$};
\node(xa)at (-3.83,3.25){$z_v^1$};

\node(xa)at (-.7,1.3){$\ddots$};

\node(xa)at (.25,.17){$z_0$};
\node(xa)at (1.25,.17){$z_1^2$};
\node(xa)at (.25,1.17){$z_1^1$};

\node(xa)at (1.3,-.7){$\ddots$};

\node(xa)at (1.75,-2.17){$z_{v-2}^2$};
\node(xa)at (3.25,-1.83){$z_{v-1}^2$};
\node(xa)at (3.25,-3.83){$z_v^2$};

\path[->](3,-2.1)edge(3,-3.9);
\path[->](-1.1,2)edge(-1.9,2);
\path[->](-2,2.9)edge(-2,2.1);
\path[->](-.1,1)edge(-.4,1);
\path[->](-1,1.9)edge(-1,1.6);
\path[->](.9,0)edge(.1,0);
\path[->](0,.9)edge(0,.1);
\path[->](1,-.1)edge(1,-.4);
\path[->](1.9,-1)edge(1.6,-1);
\path[->](2,-1.1)edge(2,-1.9);
\path[->](2.9,-2)edge(2.1,-2);
\path[->](-2.1,3)edge(-3.9,3);
\end{tikzpicture}
\begin{tikzpicture}\tikzstyle{every node}=[font=\tiny] 
\fill(-4,2)circle [radius=2pt];
\fill(-2,2)circle [radius=2pt];
\fill(-2,1)circle [radius=2pt];
\fill(-1,0)circle [radius=2pt];
\fill(0,0)circle [radius=2pt];
\fill(0,-1)circle [radius=2pt];
\fill(1,-2)circle [radius=2pt];
\fill(2,-2)circle [radius=2pt];
\fill(2,-4)circle [radius=2pt];

\node(xa)at (2.17,-1.75){$z_{v-1}^2$};
\node(xa)at (1.83,-4.25){$z_v^2$};

\node(xa)at (0.7,-1.3){$\ddots$};

\node(xa)at (-0.25,-0.17){$z_0$};
\node(xa)at (-1.25,-0.17){$z_1^1$};
\node(xa)at (-0.25,-1.17){$z_1^2$};

\node(xa)at (-1.3,0.7){$\ddots$};

\node(xa)at (-1.75,2.17){$z_{v-1}^1$};
\node(xa)at (-4.25,1.83){$z_v^1$};

\path[->](1.9,-2)edge(1.1,-2);
\path[->](2,-2.1)edge(2,-3.9);
\path[->](0.4,-1)edge(0.1,-1);
\path[->](1,-1.6)edge(1,-1.9);
\path[->](-0.1,0)edge(-0.9,0);
\path[->](0,-0.1)edge(0,-0.9);
\path[->](-1,0.4)edge(-1,0.1);
\path[->](-1.6,1)edge(-1.9,1);
\path[->](-2,1.9)edge(-2,1.1);
\path[->](-2.1,2)edge(-3.9,2);
\end{tikzpicture}
}
\end{center}
\caption{Positive staircase complexes for the case that $v$ is even (on the left) and odd (on the right).}
\label{fig:Stair}
\end{figure}
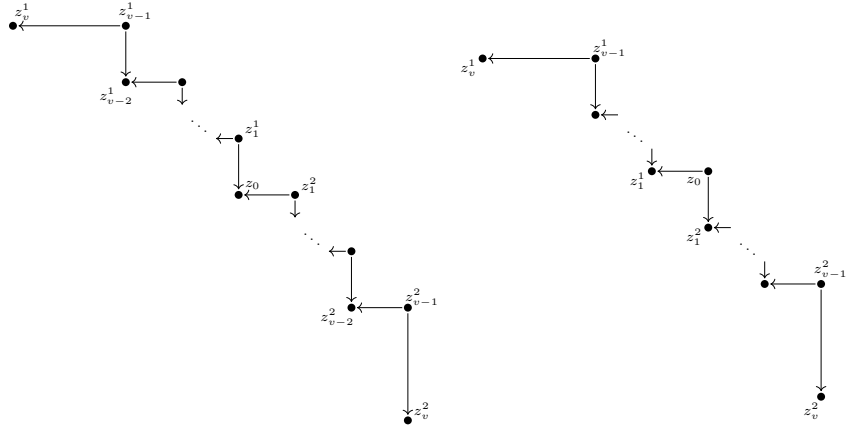
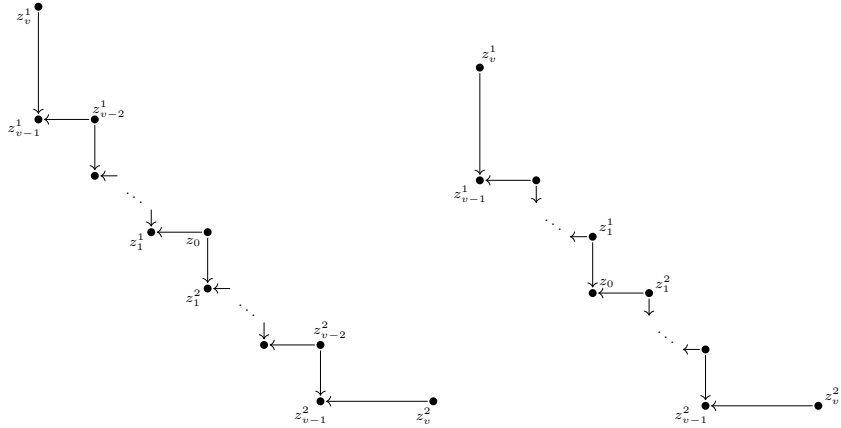
\begin{figure}[htb]
\begin{center}
\scalebox{.75}{
\begin{tikzpicture}\tikzstyle{every node}=[font=\tiny] 
\fill(-3,4)circle [radius=2pt];
\fill(-3,2)circle [radius=2pt];
\fill(-2,2)circle [radius=2pt];
\fill(-2,1)circle [radius=2pt];
\fill(-1,0)circle [radius=2pt];
\fill(0,0)circle [radius=2pt];
\fill(0,-1)circle [radius=2pt];
\fill(1,-2)circle [radius=2pt];
\fill(2,-2)circle [radius=2pt];
\fill(2,-3)circle [radius=2pt];
\fill(4,-3)circle [radius=2pt];

\node(xa)at (2.17,-1.75){$z_{v-2}^2$};
\node(xa)at (1.83,-3.25){$z_{v-1}^2$};
\node(xa)at (3.83,-3.25){$z_v^2$};

\node(xa)at (0.7,-1.3){$\ddots$};

\node(xa)at (-0.25,-0.17){$z_0$};
\node(xa)at (-1.25,-0.17){$z_1^1$};
\node(xa)at (-0.25,-1.17){$z_1^2$};

\node(xa)at (-1.3,0.7){$\ddots$};

\node(xa)at (-1.75,2.17){$z_{v-2}^1$};
\node(xa)at (-3.25,1.83){$z_{v-1}^1$};
\node(xa)at (-3.25,3.83){$z_v^1$};

\path[->](-3,3.9)edge(-3,2.1);
\path[->](1.9,-2)edge(1.1,-2);
\path[->](2,-2.1)edge(2,-2.9);
\path[->](0.4,-1)edge(0.1,-1);
\path[->](1,-1.6)edge(1,-1.9);
\path[->](-0.1,0)edge(-0.9,0);
\path[->](0,-0.1)edge(0,-0.9);
\path[->](-1,0.4)edge(-1,0.1);
\path[->](-1.6,1)edge(-1.9,1);
\path[->](-2,1.9)edge(-2,1.1);
\path[->](-2.1,2)edge(-2.9,2);
\path[->](3.9,-3)edge(2.1,-3);
\end{tikzpicture}
\begin{tikzpicture}\tikzstyle{every node}=[font=\tiny] 
\fill(4,-2)circle [radius=2pt];
\fill(2,-2)circle [radius=2pt];
\fill(2,-1)circle [radius=2pt];
\fill(1,0)circle [radius=2pt];
\fill(0,0)circle [radius=2pt];
\fill(0,1)circle [radius=2pt];
\fill(-1,2)circle [radius=2pt];
\fill(-2,2)circle [radius=2pt];
\fill(-2,4)circle [radius=2pt];

\node(xa)at (-2.17,1.75){$z_{v-1}^1$};
\node(xa)at (-1.83,4.25){$z_v^1$};

\node(xa)at (-.7,1.3){$\ddots$};

\node(xa)at (.25,.17){$z_0$};
\node(xa)at (1.25,.17){$z_1^2$};
\node(xa)at (.25,1.17){$z_1^1$};

\node(xa)at (1.3,-.7){$\ddots$};

\node(xa)at (1.75,-2.17){$z_{v-1}^2$};
\node(xa)at (4.25,-1.83){$z_v^2$};

\path[->](-1.1,2)edge(-1.9,2);
\path[->](-2,3.9)edge(-2,2.1);
\path[->](-.1,1)edge(-.4,1);
\path[->](-1,1.9)edge(-1,1.6);
\path[->](.9,0)edge(.1,0);
\path[->](0,.9)edge(0,.1);
\path[->](1,-.1)edge(1,-.4);
\path[->](1.9,-1)edge(1.6,-1);
\path[->](2,-1.1)edge(2,-1.9);
\path[->](3.9,-2)edge(2.1,-2);
\end{tikzpicture}
}
\end{center}
\caption{Negative staircase complexes for the case that $v$ is even (on the left) and odd (on the right).}
\label{fig:Stair2}
\end{figure}

We now turn to our second special type of knot. A knot is said to be \emph{thin} if its knot Floer homology $\HFKhat(K) = \oplus \HFKhat_s(K,w)$ has the property that $w-s$ is a constant $k$ for all $\HFKhat_s(K,w) \neq \{0\}$. The terminology here is because if the knot Floer homology $\HFKhat$ were graphed with the Maslov grading on one axis and the Alexander on the other, the support of $\HFKhat$ would lie in a single diagonal.  More generally, a \emph{thin complex} is a graded $\mathbb Z\oplus \mathbb Z$-filtered $\mathbb F[U,U^{-1}]$ chain complex with the property that there is a constant $k$ such that generators at planar grading $(i,j)$ always have homological grading $s=i+j-k$; this implies that the generators $[x;0,j]$ on the $j$-axis have constant difference between their Alexander and homological gradings. A thin knot $K$ admits a representative of $\CFKi(K)$ which is a thin complex \cite[Lemma 5]{Petkovacables}.

Petkova \cite{Petkovacables} showed that complexes associated to thin knots have an especially simple form up to chain homotopy equivalence, which we now review. Let the \emph{square complex} $C_s$, also known as a \emph{one-by-one box complex}, refer to an $\F$-chain complex with generators $a,b,c,Ue$, in filtration levels $(i+1,j+1),(i,j+1),(j+1,i),(i,j)$ respectively, and homological gradings $k+2,k+1,k+1,k$ respectively, with differential
$$\partial(a)=b+c,~~~~\partial(b)=Ue,~~~~\partial(c)=Ue,~~~~\partial(Ue)=0.$$
One readily checks that $\ker\partial$ and $\operatorname{im}\partial$ are both the vector space span of $b+c$ and $Ue$; hence, $C_s$ is acyclic. If $i=j$ the square complex is said to be \emph{on the main diagonal}. See Figure \ref{fig:Cs} for a picture of $C_s$.
\begin{figure}[htb]
\begin{center}
\begin{tikzpicture}\tikzstyle{every node}=[font=\tiny] 
\fill(0,0)circle [radius=2pt];
\fill(0,1)circle [radius=2pt];
\fill(1,0)circle [radius=2pt];
\fill(1,1)circle [radius=2pt];

\node(xa)at (.25,.17){$Ue$};
\node(xb)at (1.15,.15){$c$};
\node(xc)at (.15,1.15){$b$};
\node(xe)at (1.15,1.15){$a$};

\path[->](.9,1)edge(.1,1);
\path[->](.9,0)edge(.1,0);
\path[->](1,.9)edge(1,.1);
\path[->](0,.9)edge(0,.1);
\end{tikzpicture}
\end{center}
\caption{A copy of the square complex $C_s$.}
\label{fig:Cs}
\end{figure}
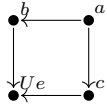
\begin{lemma} \cite[Lemma 7]{Petkovacables} \label{lemma:Petkova}
Suppose $C$ is a thin complex with horizontal and vertical homologies of rank at most one. Then $C$ is filtered chain homotopy equivalent to the tensor product of $\F[U,U^{-1}]$ with a direct sum of square complexes and at most one staircase complex all of whose steps are length one.
\end{lemma}

Indeed, Petkova shows this result for any $\mathbb Z\oplus \mathbb Z$-filtered $\mathbb F[U,U^{-1}]$-complex with the property that all differentials lower one of the horizontal and vertical filtrations by exactly one, that is, such that all arrows in the complex are either horizontal of length one or vertical of length one (which must be true of a thin complex). The case analysis involved in the proof also determines the length of the staircase (that is, the integer $v$) and whether the staircase is positive or negative.

\subsection{Involutive Heegaard Floer homology for knots}  Recall that if $C$ is a $(\Z \oplus \Z)$-filtered complex, a chain map $\alpha \co C\rightarrow C$ is said to be skew-filtered if $\alpha(\mathcal{F}_{j,i}) \subset \mathcal{F}_{i,j}$. In \cite[Section 6]{HM:involutive}, Hendricks and Manolescu use a conjugation operation on knot Floer homology to define a skew-filtered automorphism 

\[ \iota_{K,\mathcal H} \co \CFKi(\mathcal H) \rightarrow \CFKi(\mathcal H).\]

The the pair $(\CFKi(\mathcal H),\iota_{K, \mathcal H})$ is an invariant of the knot up to equivariant chain homotopy equivalence, sometimes called \emph{strong equivalence}. Specifically, two pairs $(C_i, \iota_i)$ for $i=1,2$ for which $C_i$ is a finitely-generated $\mathbb Z\oplus \mathbb Z$-filtered $\mathbb F[U,U^{-1}]$-complex and $\iota_i$ is a skew-filtered automorphism are said to be strongly equivalent if there are filtered chain maps $f: C_1 \rightarrow C_2$ and $g:C_2 \rightarrow C_1$ with the property that \[gf \sim \Id_{C_1} \qquad \qquad gf \sim \Id_{C_2} \qquad \qquad f\iota_1 \sim \iota_2f \qquad \qquad \iota_1g \sim g\iota_2\] where the first two equivalences are filtered chain homotopy equivalence and the second two are skew-filtered chain homotopy equivalence.

In general we take some representative for the strong equivalence class $(\CFKi(K), \iota_K)$, not always one arising from a Heegaard diagram. The map $\iota_K$ is in principle difficult to compute. However, it has the following useful property.

\begin{lemma} \cite{HM:involutive} The square of the map $\iota_K$ is filtered chain homotopic to the Sarkar involution; that is, we have
\[ {\iota_K}^2 \sim \sigma. \]
\end{lemma}

Note that this implies that ${\iota_K}^4$ is filtered chain homotopic to the identity map. In \cite[Sections 7 and 8]{HM:involutive}, Hendricks and Manolescu show that this is enough to uniquely determine $\iota_K$ up to filtered chain homotopy for $L$-space knots and thin knots. 

There are two knot concordance invariants analogous to $V_0(K)$ arising from this complex, as follows. First, choosing some representative for $(\CFKi(K),\iota_K)$, we consider the complex
\[
CI^{\infty} = \left((\CFKi(K) \otimes \mathbb F[Q]/(Q^2))[-1], \del + Q(1+\iota_K) \right)
\]
or equivalently the mapping cone
\[ CI^{\infty} = (\Cone(\CFKi(K) \xrightarrow{Q(1+\iota_K)} Q\CFKi(K)[-1]))\]
where multiplication by $Q$ lowers the homological grading by $1$ and the term $[-1]$ denotes an upward shift on the homological grading by $1$. Notice that this specifically implies that if $x$ is a generator in our representative for $\CFKi(K)$ having homological grading $\gr(x)=s$, then in the complex $CI^{\infty}$ the element $x$ has homological grading $s+1$ and the element $Qx$ has homological grading $s$.  To distinguish the involutive differential from the ordinary differential, we let $\partial^{\iota} = \del + Q(1+\iota_K)$ denote the involutive differential throughout.

Observe that there is an exact triangle
\[ \dots \rightarrow H_{r+1}(CI^{\infty}) \rightarrow H_{r+1}(\CFKi(K)) \xrightarrow{Q(1+\iota_K)} H_{r}(Q\CFKi(K)) \rightarrow H_r(CI^{\infty}))\rightarrow \dots\]
Since $H_*(C) \simeq \F[U,U^{-1}]$ is at most one-dimensional in any given homological grading and $(\iota_K)_*$ is an isomorphism, we see that $1+(\iota_K)_*$ is a zero map, and \[H_*(CI^{\infty}) \simeq \F[U,U^{-1},Q]/(Q^2).\]

Now consider the subcomplex $A_0^-$ of our representative for $\CFKi$, which is preserved by $\iota_K$. As previously, denote its boundary map by $\del$. Then we consider the $\mathbb F[U,Q]/(Q^2)$-complex
\[ AI_0^- = \left((A_0^- \otimes \mathbb F[Q]/(Q^2))[-1], \del + Q(1+\iota_K) \right).\]
This may also be expressed as the mapping cone
\[
AI_0^-=\mathrm{Cone}(A_0^- \xrightarrow{Q(1+\iota_K)} QA_0^-[-1])
\]
\noindent  A similar argument using the exact triangle associated to the mapping cone shows that the homology $H_*(AI_0^-)$ always admits a (noncanonical) decomposition as an $\mathbb F[U]$-module into a direct sum of two copies of $\mathbb F[U]$ along with some $U$-torsion summands. Of the two $\mathbb F[U]$ summands, one has a generator $[x_1]$ lying in an odd homological grading with the property that $U^n[x_1]$ is never in the image of $Q$ for any $n\geq 0$, and one has a generator $[x_2]$ lying in an even homological grading with the property that $U^n[x_2] \subset \Im(Q)$ for $n\gg 0$. Indeed, $[x_1]$ and $[x_2]$ may be chosen such that $Q[x_1] = U^m[x_2]$ for some sufficiently large $m$. As in the non-involutive case, the top gradings of these two summands, which is to say the gradings of the generators $[x_1]$ and $[x_2]$, are concordance invariants associated to the knot. The involutive concordance invariants are then
\[
\Vl_0(K) = -\frac{1}{2}\left(\max \{r : \exists \ x \in H_r(AI_0^-), \forall \ n, \ U^nx\neq 0 \ \text{and} \ U^nx \notin \operatorname{Im}(Q)\}-1\right) \]
and
\[ 
\Vu_0(K) = -\frac{1}{2}\max \{r : \exists \ x \in H_r(AI_0^-), \forall \ n, U^nx\neq 0; \exists \ m\geq 0 \ \operatorname{s.t.} \ U^m x \in \operatorname{Im}(Q)\}. \]

This is not quite Hendricks and Manolescu's original definition, which is given in terms of correction terms of surgeries on knots \cite[Theorem 1.6]{HM:involutive} and rephrased in terms of the gradings of $AI_0^+ = C\{(i,j): i\geq 0 \text{ or } j\geq 0\}$ \cite[pg. 45]{HM:involutive}; our definition is equivalent via the duality of the minus and plus variants of Heegaard Floer homology (cf for example the discussion in \cite[Section 3.8]{HHSZ}).

We briefly recall two special cases. The \emph{standard staircase map} on a staircase complex $C$ is the reflection across the line $i=j$.

\begin{proposition}\cite[Section 7]{HM:involutive} Let $K$ be an $L$-space knot, so that $\CFKi(K) \simeq C\otimes \F[U,U^{-1}]$ for $C$ a positive staircase complex. The involution on $\CFKi(K)$ is generated by the standard staircase map, and the involutive concordance invariants are
\[V_0(K) = \Vl_0(K) = \Vu_0(K) = n(K)\]
and
\[V_0(\overline{K})=\Vl_0(\overline{K})=0, \ \Vu_0(\overline{K}) = -n(K).\]
\end{proposition}

Computations are also accessible for thin knots. We first define a standard map on a pair of square complexes. Consider two square complexes $C_s$ generated by $a,b,c,Ue$ and $C_s'$ generated by $a',b',c',Ue'$, with such that $a$ lies in planar grading $(i+1,j+1)$ and $a'$ lies in planar grading $(j+1,i+1)$, as in Figure \ref{fig:Interchange}. The \emph{standard square map} between $C_s \otimes \F[U,U^{-1}]$ and $C_s' \otimes \F[U,U^{-1}]$ is
$$\iota_K(a)=a'+e',~~~~\iota_K(b)=c',~~~~\iota_K(c)=b',~~~~\iota_K(Ue)=Ue'$$
$$\iota_K(a')=a,~~~~\iota_K(b')=c,~~~~\iota_K(c')=b,~~~~\iota_K(Ue')=Ue.$$

\noindent Observe that this is a chain map with the property that ${\iota_K}^2=\sigma$, as $\sigma(a)=a+e$ and $\sigma(a')=a'+e'$.
\begin{figure}[htb]
\begin{center}
\begin{tikzpicture}\tikzstyle{every node}=[font=\tiny] 
\fill(-1,1)circle [radius=2pt];
\fill(-1,2)circle [radius=2pt];
\fill(0,1)circle [radius=2pt];
\fill(0,2)circle [radius=2pt];

\node(xa)at (-.75,1.17){$Ue$};
\node(xb)at (.15,1.15){$c$};
\node(xc)at (-.85,2.15){$b$};
\node(xe)at (.15,2.15){$a$};

\path[->](-.1,2)edge(-.9,2);
\path[->](-.1,1)edge(-.9,1);
\path[->](0,1.9)edge(0,1.1);
\path[->](-1,1.9)edge(-1,1.1);

\fill(1,-1)circle [radius=2pt];
\fill(1,0)circle [radius=2pt];
\fill(2,-1)circle [radius=2pt];
\fill(2,0)circle [radius=2pt];

\node(xa)at (1.25,-.83){$Ue'$};
\node(xb)at (2.15,-.85){$c'$};
\node(xc)at (1.15,.15){$b'$};
\node(xe)at (2.15,.15){$a'$};

\path[->](1.9,0)edge(1.1,0);
\path[->](1.9,-1)edge(1.1,-1);
\path[->](2,-.1)edge(2,-.9);
\path[->](1,-.1)edge(1,-.9);
\end{tikzpicture}
\end{center}
\caption{A pair of square complexes.}
\label{fig:Interchange}
\end{figure}
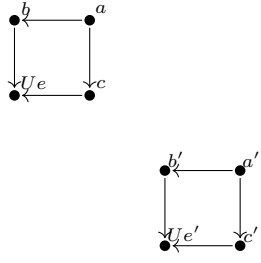

\begin{proposition}\cite[Proposition 8.1]{HM:involutive}\label{prop:thiniota} 
Let $K$ be a thin knot. The complex $(\CFKi(K), \iota_K)$ admits a representative which decomposes as a direct sum of pairs of square complexes $C_s \otimes \F[U,U^{-1}]$ and $C_s'\otimes \F[U,U^{-1}]$ interchanged by the standard square map, and a complex $C\otimes \F[U,U^{-1}]$ preserved by $\iota_K$ such that $C$ consists of a staircase complex and at most one square complex on the main diagonal.
\end{proposition}

The key fact used in the proof of Proposition \ref{prop:thiniota} is that for a thin complex, grading considerations ensure that $\iota_K$ interchanges the planar gradings $(i,j)$ and $(j,i)$ and that ${\iota_K}^2 = \sigma$ on the nose (that is, the filtered chain homotopy $H$ relating them must be zero). Hendricks and Manolescu use Proposition \ref{prop:thiniota} to compute the involutive correction terms of thin knots \cite[Proposition 8.2]{HM:involutive}.

We conclude this subsection with a look at the involutive concordance invariants of the example complexes introduced so far, all of which fall into one of the special cases above.

\begin{example} Let $K$ be the right-handed trefoil, with a representative of $\CFKi(K)$ as in Figure \ref{fig:rightTrefoil}. The automorphism $\iota_K$ is given by
\[
\iota_K(b)=U^{-1}c \qquad \iota_K(c) = Ub \qquad \iota_K(a)=a.
\]
The right-handed trefoil $K$ is an $L$-space knot with $n(K)=1$; therefore we have
\[V_0(K)=\Vl_0(K)=\Vu_0(K) = 1. \]

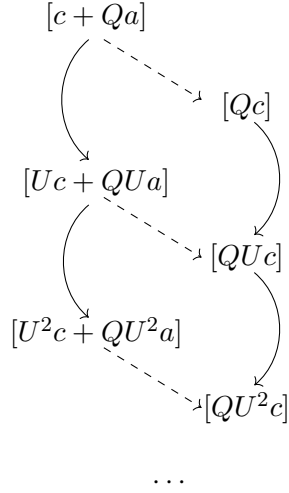
\begin{figure}
    \centering
    \begin{tikzpicture}
    \node(1)at(0,0.2){$[c+Qa]$};
    \path[->][bend right = 50](-.1,-.1)edge(-.1,-1.7);
    \path[->][dashed](.1,-.1)edge(1.4,-0.9);
    \node(2)at(0,-2){$[Uc+QUa]$};
    \path[->][bend right = 50](-.1,-2.3)edge(-.1,-3.8);
    \path[->][dashed](.1,-2.2)edge(1.4,-3);
    \node(3)at(0,-4){$[U^2c+QU^2a]$};
    \path[->][dashed](.1,-4.2)edge(1.4,-5);
    \node(5)at(2,-1){$[Qc]$};
    \path[->][bend left = 50](2.1,-1.2)edge(2.1,-2.7);
    \node(6)at(2,-3){$[QUc]$};
    \path[->][bend left = 50](2.1,-3.2)edge(2.1,-4.7);
    \node(7)at(2,-5){$[QU^2c]$};
    \node(8)at(1,-6){$\bf \cdots$};
    \end{tikzpicture}
    \caption{The homology $H_*(AI_0^-)$ for the right-handed trefoil, in terms of the representative for the filtered chain homotopy equivalence class for $\CFKi(K)$ of Figure \ref{fig:rightTrefoil}. Curved lines denote the action of the variable $U$, and dashed lines denote the action of the variable $Q$. The element $[c+Qa]$ lies in homological grading $-1$ and the element $[Qc]$ lies in grading $-2$.} \label{fig:rhtAI0}
\end{figure}

\noindent In more detail, the homology of $H_*(A_0^-)$ is isomorphic to $\mathbb F_{(-2)}[U]$ generated over $\mathbb F[U]$ by the maximally-graded element $[b]$ in the tower, with the consequence that
\[ V_0(K) = -\frac{1}{2}\gr[b] = -\frac{1}{2}(-2)=1. \]
The homology of the mapping cone 
\[ AI_0^- = \Cone(A_0^- \xrightarrow{Q(1+\iota_K)} QA_0^-[-1])\] 
is shown in Figure~\ref{fig:rhtAI0}; we see that as an $\mathbb F[U]$-module,
\[H_*(AI_0^-) \simeq \F_{(-1)}[U] \oplus \F_{(-2)}[U]\]
where the tower summand $\F_{(-1)}[U]$ is generated by the element $[c+Qa]$ and the tower summand $\F_{(-2)}[U]$ is generated by $[Qc]=Q[c+Qa]$. Here we have that $[c+Qa]$  is a maximally-graded element which is not $U$-torsion and such that $U^n[c+Qa]$ is never in the image of $Q$, so that
\[
\Vl_0(K) = -\frac{1}{2}(\gr[c+Qa]-1) = -\frac{1}{2}(-1-1) =1\]
and likewise $[Qc]$ is a maximally-graded element which is not $U$-torsion and lies in the image of $Q$, such that
\[
\Vu_0(K) = -\frac{1}{2}(\gr[Qc])=-\frac{1}{2}(-2)=1.\]
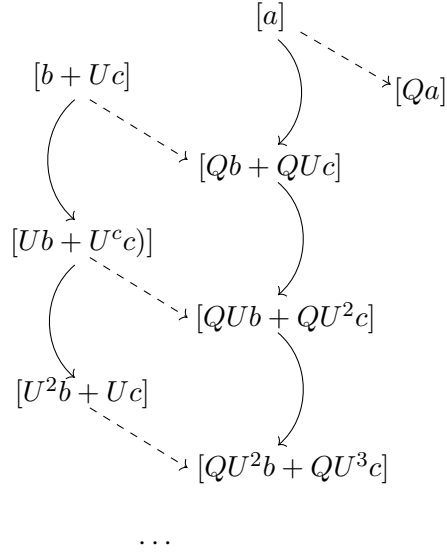
\begin{figure}
     \centering
    \begin{tikzpicture}
    \node(1)at(0,0.2){$[b+Uc]$};
    \path[->][bend right = 50](-.1,-.1)edge(-.1,-1.7);
    \path[->][dashed](.1,-.1)edge(1.4,-0.9);
    \node(2)at(0,-2){$[Ub+U^cc)]$};
    \path[->][bend right = 50](-.1,-2.3)edge(-.1,-3.8);
    \path[->][dashed](.1,-2.2)edge(1.4,-3);
    \node(3)at(0,-4){$[U^2b+Uc]$};
    \path[->][dashed](.1,-4.2)edge(1.4,-5);
    \node(4)at(2.5,1){$[a]$};
    \path[->][bend left = 50](2.6,0.7)edge(2.6,-0.7);
    \path[->][dashed](2.9,.8)edge(4.1,.1);
    \node(5)at(2.5,-1){$[Qb+QUc]$};
    \path[->][bend left = 50](2.6,-1.2)edge(2.6,-2.7);
    \node(6)at(2.7,-3){$[QUb+QU^2c]$};
    \path[->][bend left = 50](2.6,-3.2)edge(2.6,-4.7);
    \node(7)at(2.8,-5){$[QU^2b+QU^3c]$};
    \node(9)at(4.5,0){$[Qa]$};
    \node(8)at(1,-6){$\bf \cdots$};
    \end{tikzpicture}
     \caption{The homology $H_*(AI_0^-)$ for the left-handed trefoil, in terms of the representative for the filtered chain homotopy equivalence class for $\CFKi(K)$ of Figure \ref{fig:leftTrefoil}. Curved lines denote the action of the variable $U$, and dashed lines denote the action of the variable $Q$. The element $[b+Uc]$ lies in homological grading $1$ and the element $[a]$ lies in homological grading $2$; the element $[Qa]$ lies in homological grading $1$.}
     \label{fig:lhtAI0}
 \end{figure}

For the left-handed trefoil $\overline{K}$ with $\CFKi(\overline{K})$ as in Figure~\ref{fig:leftTrefoil}, we have
\[ \iota_K(b)=Uc \qquad \iota_K(c)=U^{-1}b \qquad \iota_K(a)=a.\]
Since $\overline{K}$ is the mirror of an $L$-space knot with $n(K)=1$ we have
\[
V_0(\overline{K})=\Vl_0(\overline{K})=0, \qquad \Vu_0(\overline{K})=-1.
\]
More precisely, we see that $H_*(A_0^-) \simeq \F_{(0)}[U]\oplus \F_{(1)}$, where the summand $\F_{(0)}[U]$ is generated over $\F[U]$ by $[b+Uc]$ and the summand $\F_{(1)}$ has an $\F$-basis $[a]$. We have that the homology class $[b+Uc]$ is a maximally-graded element which is not $U$-torsion, and $\gr([b+Uc])=0$, so $V_0(K)=0$. 

The homology of the mapping cone complex $AI_0^-$ is shown in Figure~\ref{fig:lhtAI0}. We see that as an $\mathbb F[U]$-module we have
\[ H_*(AI_0^-) \simeq \F_{(1)}[U] \oplus \F_{(2)}[U] \oplus \F_{(1)}\]
where the summand $\F_{(1)}[U]$ is generated over $\F[U]$ by $[b+Uc]$, the summand $\F_{(2)}[U]$ is generated over $\F[U]$ by $[a]$, and the summand $\F_{(1)}$ has an $\F$-basis $[Qa]$. Note, in particular, that $U[a]=[Ua] = [Qb+QUc]$, since \[\partial^{\iota}(b) = \partial(b) + Q(1+\iota_K)(b) = Ua + Q(b+Uc).\]
Therefore, in the homology $H_*(AI_0^-)$, the element $[b+Uc]$ is a maximally-graded element which is not $U$-torsion and for which no $U$-power lies in the image of $Q$, so
\[
\Vl_0(\overline{K}) = -\frac{1}{2}(\gr[b+Uc])=-\frac{1}{2}(1-1)=0\]
\noindent whereas $[a]$ is a maximally-graded element which is not $U$-torsion and such that $U[a]=[Q(b+Uc)]$, so
\[
\Vu_0(\overline{K})= -\frac{1}{2}(\gr[a])=-\frac{1}{2}(2)=-1.\]\end{example}

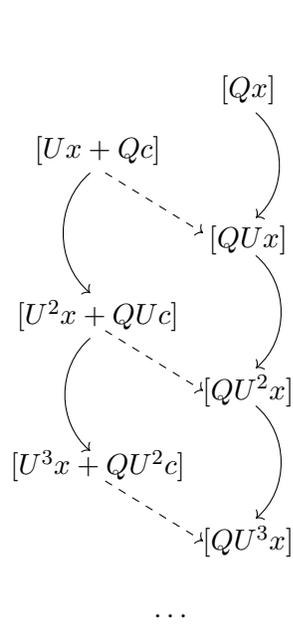
\begin{figure}
    \centering
    \begin{tikzpicture}
    \node(0)at(4,2){$[e]$};
    \node(1)at(0,0.2){$[Ux + Qc]$};
    \path[->][bend right = 50](-.1,-.1)edge(-.1,-1.7);
    \path[->][dashed](.1,-.1)edge(1.4,-0.9);
    \node(2)at(0,-2){$[U^2x + QUc]$};
    \path[->][bend right = 50](-.1,-2.3)edge(-.1,-3.8);
    \path[->][dashed](.1,-2.2)edge(1.4,-3);
    \node(3)at(0,-4){$[U^3x + QU^2c]$};
    \path[->][dashed](.1,-4.2)edge(1.4,-5);
    \node(4)at(2,1){$[Qx]$};
    \path[->][bend left = 50](2.1,0.7)edge(2.1,-0.7);
    \node(5)at(2,-1){$[QUx]$};
    \path[->][bend left = 50](2.1,-1.2)edge(2.1,-2.7);
    \node(6)at(2,-3){$[QU^2x]$};
    \path[->][bend left = 50](2.1,-3.2)edge(2.1,-4.7);
    \node(7)at(2,-5){$[QU^3x]$};
    \node(8)at(1,-6){$\bf \cdots$};
    \end{tikzpicture}
\caption{The homology $H_*(AI_0^-)$ for the figure-eight knot, in terms of the representative for the filtered chain homotopy equivalence class for $\CFKi(K)$ of Figure \ref{fig:figureEight}. Curved lines denote the action of the variable $U$, and dashed lines denote the action of the variable $Q$. The element $[Ux+Qc]$ lies in homological grading $-1$ and the element $[Qx]$ lies in homological grading $0$; the element $[e]$ lies in homological grading $1$.}
\label{fig:fig8AI0}
\end{figure}

\begin{example}Let $K$ be the figure-eight knot, with $\CFKi(K)$ as in Figure~\ref{fig:figureEight}. The automorphism $\iota_K$ is given by 
\[
\iota_K(c)= b \qquad \iota_K(b)=c \qquad \iota_K(e) = e \qquad \iota_K(a)= a + x \qquad \iota_K(x) = x + e.
\]
We see that the homology $H_*(A_0^-) \simeq \F_{(0)}[U]$ generated by the element $[x]$. Thus $[x]$ is a maximally graded element which is not $U$-torsion, such that $V_0(K) = -\frac{1}{2}\gr[x] = 0.$ The homology $H_*(AI_0^-)$ appears in Figure~\ref{fig:fig8AI0}. We see that as an $\F[U]$-module,
\[H_*(AI_0^-) \simeq \F_{(-1)}[U] \oplus \F_{(0)}[U] \oplus \F_{(1)}\]
where the summand $\F_{(-1)}[U]$ is generated over $\F[U]$ by $[Ux+Qc]$, the summand $\F_{(0)}[U]$ is generated over $\F[U]$ by $[Qx]$, and the summand $\F_{(1)}$ is generated by $[e]$. Note in particular that $Q[Ux+Qc] = [QUx] = U[Qx]$. Then the element $[Ux+Qc]$ is a maximally-graded element which is not $U$-torsion and for which no $U$-power lies in the image of $Q$, so
\[
\Vl_0(\overline{K}) = -\frac{1}{2}(\gr[Ux+Qc])=-\frac{1}{2}(-1-1)=1\]
\noindent whereas $[Qx]$ is a maximally-graded element which is not $U$-torsion, so
\[
\Vu_0(\overline{K})= -\frac{1}{2}(\gr[Qx])=-\frac{1}{2}(0)=0.\] \end{example}

\section{Involutive invariants of $(-2,m,n)$ pretzel knots}

\subsection{The complex $\CFK^{\infty}(P(-2,m,n))$}

In this subsection we review Goda, Matsuda, and Morifuji's  computation of the knot Floer homology of $P(-2,m,n)$ \cite[Section 5]{GMM} and prove that there is a change of basis which allows us to simplify the complex. Let $m \geq n \geq 3$. For our proof, we will require a description of the generators of the complex and their filtration levels, some information about the differentials of particular generators, and the final computation of knot Floer homology.

Following \cite{GMM}, we fix the following notation. Let $g = \frac{m + n}{2}$ be the genus of the pretzel knot, and set $m' = \frac{m - 3}{2},$ $ n' = \frac{n - 3}{2}$,
 \[ \gamma(g) =\begin{cases} 
      1 - \frac{g-1}{2} & g \; \;\mathrm{ odd} \\
      1 - \frac{g}{2} & g \; \; \mathrm{even}
   \end{cases} 
  \qquad \text{and} \qquad
   \delta(g) = \begin{cases} 
     \frac{g-1}{2} & g \; \;\mathrm{ odd} \\
      \frac{g}{2} - 1 & g \; \; \mathrm{even}
   \end{cases}.
\]

\noindent With respect to these choices, there is a representative for the complex $\CFK^{\infty}(K)$ (arising, indeed, from a genus one Heegaard diagram) which has a basis as an $\mathbb F$-vector space as follows, where $t \in \mathbb Z$ is arbitrary. Setting $t=0$ returns a set of generators of $\CFKi(K)$ over $\mathbb F[U,U^{-1}]$.
\begin{align*}
&[y_1,;t + \gamma(g) - 1, t + \delta(g) + 1] \\ &[y_2;t + \gamma(g) - 1, t + \delta(g)]\\ 
&[y_3;t + \delta(g), t + \gamma(g) - 1] \\ &[y_4; t + \delta(g) + 1, t + \gamma(g) -1]\\
&[x_{2p+1,2q+1};t + \gamma(g) + p + q + 1, t + \delta(g) - p - q] \;\; \text{ for } 0 \leq p \leq n', 0 \leq q \leq m' \\
&[x_{2p+1,2q};t + \gamma(g) + p + q, t + \delta(g) - p - q] \;\; \text{ for } 0 \leq p \leq n', 1 \leq q \leq m' \\
&[x_{2p,2q+1};t + \gamma(g) + m' + p - q, t + \delta(g) - m' - p - q] \;\; \text{ for } 1 \leq p \leq n', 0 \leq q \leq m' \\
&[x_{2p,2q};t + \gamma(g) + m' + p - q, t + \delta(g) - m' - p + q -1] \;\; \text{ for } 1 \leq p \leq n', 1 \leq q \leq m'
\end{align*}

We refer to the generators $y_i$ as the \emph{exceptional generators} and the generators $x_{k, \ell}$ as the \emph{ordinary generators}. One may straightforwardly check that the ordinary generators lie strictly between the lines $j - i = g-1$ and $j - i = 1-g$ and that $[y_2; t+ \gamma(g)-1, t+ \delta(g)]$ and $[y_3; t+ \delta(g), t+ \gamma(g)-1]$ lie on these lines respectively. Furthermore, $[y_1;t + \gamma(g) - 1, t + \delta(g) + 1]$ lies on the line $j - i = g$ and $[y_4;t + \delta(g) + 1, t + \gamma(g) - 1]$ lies on the line $j-i = -g.$  

We summarize the important aspects of the differential below.

\begin{lemma}\cite[Section 5]{GMM} With respect to the basis for $\CFK^{\infty}(P(-2,m,n))$ above, the differentials of the exceptional generators are
\begin{itemize}
    \item $\partial [y_1;i,j] = [y_2;i,j-1]$
    \item $\partial [y_2;i,j] = 0$
    \item $\partial [y_3;i,j] = 0$
    \item $\partial [y_4;i,j] = [y_3;i-1,j].$
\end{itemize}
Furthermore, the elements whose differential, written as a sum of the generators above, includes an exceptional generator are
\begin{itemize}
    \item $\partial [x_{1,1};i,j] = [x_{2,m-2};i,j-1] + [x_{1,2};i,j-1] + [y_2;i-2,j]$
    \item $\partial [x_{n-2,m-2};i,j] = [x_{n-3,1};i-1,j] + [x_{n-2,m-3};i-1,j] + [y_3;i,j-2].$
     
\end{itemize}
For all other ordinary generators $[x;i,j]$, the differential is a sum of elements $[w; i-1,j]$ and $[z;i,j-1]$; that is, any arrow appearing in the differential is either horizontal or vertical of length one.
\end{lemma}

Note that the computation of \cite{GMM} is carried out for $\mathbb Z$ coefficients, so there are some signs in the original that do not appear above. The computation of the knot Floer homology of the pretzel knots carried out from this definition is as follows.

\begin{lemma} \label{lemma:hfk} \cite[Section 5]{GMM} For $m\geq n \geq 3$, the knot Floer homology of $K=P(-2,m,n)$ in positive Alexander gradings is equal to
\[\HFKhat{(K,w)} = 
\begin{cases}
0 & w > g\\
\F_{g+w} & w = g, g - 1\\
0 & w = g -2\\
\F^{g-2-w}_{g-1+w} & g - n \leq w < g - 2\\
\F^{n-2}_{g-1+w} & 0 \leq w < g - n.
\end{cases}
\]
\end{lemma}

We may now state our first splitting lemma:
\begin{lemma} \label{lemma:split1} Let $K$ be a pretzel knot $P(-2,m,n)$ with $m\geq n>3$.  Then $CFK^\infty(K)$ is chain homotopy equivalent to the tensor product of $\F[U,U^{-1}]$ with a direct sum of a single staircase complex together with some number of copies of the square complex $C_s$.
\end{lemma}

\begin{figure}
\begin{center}
\scalebox{.75}{
\begin{tikzpicture}\tikzstyle{every node}=[font=\tiny] 

\fill(-1,2)circle [radius=2pt];
\fill(-3,2)circle [radius=2pt];
\fill(-3,3)circle [radius=2pt];
\path[->](-1.1,2)edge(-2.9,2);
\path[->](-3,2.9)edge(-3,2.1);

\fill(2,-1)circle [radius=2pt];
\fill(2,-3)circle [radius=2pt];
\fill(3,-3)circle [radius=2pt];
\path[->](2,-1.1)edge(2,-2.9);
\path[->](2.9,-3)edge(2.1,-3);

\node(y1)at (-1,2.3){$x_{1,1}$};
\node(y1)at (-2.8,2.2){$y_2$};
\node(y1)at (-2.8,3.2){$y_1$};

\node(y1)at (2.8,-.9){$x_{n-2,m-2}$};
\node(y1)at (2.2,-2.8){$y_3$};
\node(y1)at (3.2,-2.8){$y_4$};

\path[dotted](-4,-1)edge(1,4);
\path[dotted](-1,-4)edge(4,1);

\fill[fill=gray!20] (-3.9,-1.1) -- (1.1,3.9) -- (3.9,1.1) -- (-1.1,-3.9) -- cycle;

\node(y1)at(0,0){\large{$C'$}};
\node(y1)at(1.5,4.2){$j-i=\frac{m+n-6}2$};
\node(y1)at(4.8,1.2){$j-i=\frac{6-m-n}2$};

\end{tikzpicture}
}
\end{center}
\caption{A schematic diagram of the situation in Lemma \ref{lemma:split1}. The four exceptional generators and all of the differentials involving them are shown. The remaining generators in $C'$ lie strictly between the two dashed lines.}\label{fig:Lemma33idea}
\end{figure}
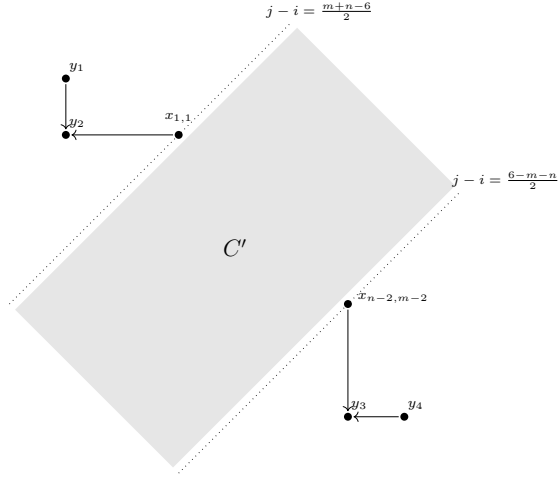

\begin{proof} As earlier in this section, we let $g = \frac{m+n}{2}$ denote the genus of the pretzel knot. In this proof we let $y_1 = [y_1; -g,0], y_2 = [y_2; -g,-1], y_3 = [y_3;-1,-g], y_4 = [y_4; 0, -g], x_{1,1} = [x_{1,1};2-g, -1],$ and $x_{n-2,m-2} = [x_{n-2,m-2};-1,2-g].$ These gradings are chosen so that $y_1, y_4, x_{1,1}$, and $x_{n-2,m-2}$ lie in homological grading $0$ and $y_2$ and $y_3$ lie in homological grading $-1$. Moreover, we have
\[\del(y_1)=y_2 \qquad \qquad \del(y_4) = y_3.\]
\noindent Furthermore, 
\[ \del_{\mathrm{horz}}(x_{1,1})=y_2 \qquad \qquad \del_{\mathrm{vert}}(x_{n-2,m-2})=y_3.\]

\noindent A schematic of the situation appears in Figure \ref{fig:Lemma33idea}. We first observe that $[y_1]$ is a nontrivial element of the horizontal homology of $\CFKi(K)$, since the only differential involving $y_1$ in the complex is a vertical arrow to $y_2$. Since the horizontal homology has rank one, it in fact must be generated by $[y_1]$. Similarly, the vertical homology is generated by $[y_4].$ Now, let us consider the complex $C'$ which is a quotient of $C =\CFK^{\infty}(K)$ by the subcomplex generated by the elements $y_1,y_2,y_3,$ and $y_4.$ In doing so, we have deleted the horizontal arrow from $x_{1,1}$ to $y_2$ and the vertical arrow from $x_{n-2,m-2}$ to $y_3.$ If we now consider the horizontal homology of $C',$ we see that it is generated by $[x_{1,1}]$ since now $x_{1,1}$ has trivial horizontal differential and the horizontal differential is otherwise unchanged. Similarly, the vertical homology of $C'$ is generated by $[x_{n-2,m-2}].$ In particular, both the vertical and horizontal homologies of $C'$ have rank one. Moreover, the remaining differentials in $C'$ each lower one of the horizontal or vertical planar gradings by exactly one; indeed, $C'$ is a thin complex. Therefore, Petkova's work \cite[Lemma 7]{Petkovacables}, reviewed in Section~\ref{sec:background} as Lemma \ref{lemma:Petkova}, tells us that up to change of basis $C'$ may be decomposed as the tensor product of $\F[U,U^{-1}]$ with a direct sum of a staircase complex and some number of one-by-one boxes.

However, the case analysis of \cite[Section 3]{Petkovacables} also recovers the sign and length of the staircase complex in question. We start by considering an element lying on the bottom-most occupied diagonal line $j-i=k$ in $C'$, which in this case is $x_{n-2,m-2}$ (or any of its $U$-translates). This element does not lie in the image of a vertical differential, which corresponds to Case 2 of Lemma 7 of Petkova \cite[Section 3]{Petkovacables}; in particular, $x_{n-2,m-2}$ must form the lower right corner of a negative staircase summand. Hence we can split off a staircase complex whose lower-right corner is $x_{n-2,m-2}$. By symmetry of the staircase across the main diagonal, the upper left corner must lie in planar grading $\left(2-g,-1\right)$ and have homological grading $0$; the only element in $C'$ satisfying this is $x_{1,1}$.

Now we include $C'$ back into $C.$ We see that $x_{1,1}$ has a horizontal length two arrow to $y_2$ and $x_{n-2,m-2}$ has a vertical length two arrow to $y_3.$ Hence the negative staircase in $C'$ extends to a negative staircase in $C$ which includes the four exceptional generators (and now has two steps of length two). The conclusion of the lemma follows. 
\end{proof}

See Figure \ref{fig:Pn299} for the example of the pretzel knot $P(-2,9,9)$ after this change of basis. Note that since the change of basis of the preceding lemma affected only to the ordinary generators, we may continue to distinguish between the four exceptional generators and the remaining ordinary generators in the subsection that follows.

\begin{figure}
\begin{center}
\scalebox{.75}{
\begin{tikzpicture}\tikzstyle{every node}=[font=\tiny] 
\path[->][dotted](0,-7)edge(0,7);
\path[->][dotted](-7,0)edge(7,0);

\node at (-.25,6.75){$j$};
\node at (6.75,-.25){$i$};

\node(y1)at (6,6){$\iddots$};
\node(y1)at (-5.5,-5.5){$\iddots$};

\fill(-4,5)circle [radius=2pt];
\fill(-4,4)circle [radius=2pt];
\fill(-2,4)circle [radius=2pt];
\fill(-2,3)circle [radius=2pt];
\fill(-1,3)circle [radius=2pt];
\fill(-1,2)circle [radius=2pt];
\fill(0,2)circle [radius=2pt];
\fill(0,1)circle [radius=2pt];
\fill(1,1)circle [radius=2pt];
\fill(1,0)circle [radius=2pt];
\fill(2,0)circle [radius=2pt];
\fill(2,-1)circle [radius=2pt];
\fill(3,-1)circle [radius=2pt];
\fill(3,-2)circle [radius=2pt];
\fill(4,-2)circle [radius=2pt];
\fill(4,-4)circle [radius=2pt];
\fill(5,-4)circle [radius=2pt];

\fill(-1.85,2.15)circle [radius=2pt];
\fill(-1.85,2.85)circle [radius=2pt];
\fill(-1.15,2.15)circle [radius=2pt];
\fill(-1.15,2.85)circle [radius=2pt];

\fill(-.85,1.15)circle [radius=2pt];
\fill(-.85,1.85)circle [radius=2pt];
\fill(-.15,1.15)circle [radius=2pt];
\fill(-.15,1.85)circle [radius=2pt];

\fill(.15,.15)circle [radius=2pt];
\fill(.15,.85)circle [radius=2pt];
\fill(.85,.15)circle [radius=2pt];
\fill(.85,.85)circle [radius=2pt];

\fill(1.15,-.85)circle [radius=2pt];
\fill(1.15,-.15)circle [radius=2pt];
\fill(1.85,-.85)circle [radius=2pt];
\fill(1.85,-.15)circle [radius=2pt];

\fill(2.15,-1.85)circle [radius=2pt];
\fill(2.15,-1.15)circle [radius=2pt];
\fill(2.85,-1.85)circle [radius=2pt];
\fill(2.85,-1.15)circle [radius=2pt];

\node(y1)at (-4.8,4.2){$z_8^1$};
\node(y1)at (-4.75,3.2){$z_7^1$};
\node(y1)at (-2.8,3.2){$z_6^1$};
\node(y1)at (-2.75,2.2){$z_5^1$};
\node(y1)at (-1.8,1.8){$z_4^1$};
\node(y1)at (-1.75,1.2){$z_3^1$};
\node(y1)at (-.8,.8){$z_2^1$};
\node(y1)at (-.75,.2){$z_1^1$};
\node(y1)at (.2,-.2){$z_0$};
\node(y1)at (.25,-.8){$z_1^2$};
\node(y1)at (1.2,-1.2){$z_2^2$};
\node(y1)at (1.25,-1.8){$z_3^2$};
\node(y1)at (2.2,-2.2){$z_4^2$};
\node(y1)at (2.25,-2.8){$z_5^2$};
\node(y1)at (3.2,-2.8){$z_6^2$};
\node(y1)at (3.25,-4.8){$z_7^2$};
\node(y1)at (4.2,-4.8){$z_8^2$};

\path[->](-4,4.9)edge(-4,4.1);
\path[->](-2.1,4)edge(-3.9,4);
\path[->](-2,3.9)edge(-2,3.1);
\path[->](-1.1,3)edge(-1.9,3);
\path[->](-1,2.9)edge(-1,2.1);
\path[->](-.1,2)edge(-.9,2);
\path[->](0,1.9)edge(0,1.1);
\path[->](.9,1)edge(.1,1);
\path[->](1,.9)edge(1,0.1);
\path[->](1.9,0)edge(1.1,0);
\path[->](2,-.1)edge(2,-.9);
\path[->](2.9,-1)edge(2.1,-1);
\path[->](3,-1.1)edge(3,-1.9);
\path[->](3.9,-2)edge(3.1,-2);
\path[->](4.9,-4)edge(4.1,-4);
\path[->](4,-2.1)edge(4,-3.9);

\path[->](-1.25,2.85)edge(-1.75,2.85);
\path[->](-1.25,2.15)edge(-1.75,2.15);
\path[->](-1.15,2.75)edge(-1.15,2.25);
\path[->](-1.85,2.75)edge(-1.85,2.25);

\path[->](-.25,1.85)edge(-.75,1.85);
\path[->](-.25,1.15)edge(-.75,1.15);
\path[->](-.15,1.75)edge(-.15,1.25);
\path[->](-.85,1.75)edge(-.85,1.25);

\path[->](.75,.85)edge(.25,.85);
\path[->](.75,.15)edge(.25,.15);
\path[->](.85,.75)edge(.85,.25);
\path[->](.15,.75)edge(.15,.25);

\path[->](1.75,-.15)edge(1.25,-.15);
\path[->](1.75,-.85)edge(1.25,-.85);
\path[->](1.85,-.25)edge(1.85,-.75);
\path[->](1.15,-.25)edge(1.15,-.75);

\path[->](2.75,-1.15)edge(2.25,-1.15);
\path[->](2.75,-1.85)edge(2.25,-1.85);
\path[->](2.85,-1.25)edge(2.85,-1.75);
\path[->](2.15,-1.25)edge(2.15,-1.75);

\node(y1)at(-.5,1.5){\large{2}};
\node(y1)at(.5,.5){\large{3}};
\node(y1)at(1.5,-.5){\large{2}};


\fill(-3,6)circle [radius=2pt];
\fill(-3,5)circle [radius=2pt];
\fill(-1,5)circle [radius=2pt];
\fill(-1,4)circle [radius=2pt];
\fill(0,4)circle [radius=2pt];
\fill(0,3)circle [radius=2pt];
\fill(1,3)circle [radius=2pt];
\fill(1,2)circle [radius=2pt];
\fill(2,2)circle [radius=2pt];
\fill(2,1)circle [radius=2pt];
\fill(3,1)circle [radius=2pt];
\fill(3,0)circle [radius=2pt];
\fill(4,0)circle [radius=2pt];
\fill(4,-1)circle [radius=2pt];
\fill(5,-1)circle [radius=2pt];
\fill(5,-3)circle [radius=2pt];
\fill(6,-3)circle [radius=2pt];

\fill(-0.85,3.15)circle [radius=2pt];
\fill(-0.85,3.85)circle [radius=2pt];
\fill(-0.15,3.15)circle [radius=2pt];
\fill(-0.15,3.85)circle [radius=2pt];

\fill(0.15,2.15)circle [radius=2pt];
\fill(0.15,2.85)circle [radius=2pt];
\fill(0.85,2.15)circle [radius=2pt];
\fill(0.85,2.85)circle [radius=2pt];

\fill(1.15,1.15)circle [radius=2pt];
\fill(1.15,1.85)circle [radius=2pt];
\fill(1.85,1.15)circle [radius=2pt];
\fill(1.85,1.85)circle [radius=2pt];

\fill(2.15,0.15)circle [radius=2pt];
\fill(2.15,0.85)circle [radius=2pt];
\fill(2.85,0.15)circle [radius=2pt];
\fill(2.85,0.85)circle [radius=2pt];

\fill(3.15,-0.85)circle [radius=2pt];
\fill(3.15,-0.15)circle [radius=2pt];
\fill(3.85,-0.85)circle [radius=2pt];
\fill(3.85,-0.15)circle [radius=2pt];


\path[->](-3,5.9)edge(-3,5.1);
\path[->](-1.1,5)edge(-2.9,5);
\path[->](-1,4.9)edge(-1,4.1);
\path[->](-0.1,4)edge(-0.9,4);
\path[->](0,3.9)edge(0,3.1);
\path[->](0.9,3)edge(0.1,3);
\path[->](1,2.9)edge(1,2.1);
\path[->](1.9,2)edge(1.1,2);
\path[->](2,1.9)edge(2,1.1);
\path[->](2.9,1)edge(2.1,1);
\path[->](3,0.9)edge(3,0.1);
\path[->](3.9,0)edge(3.1,0);
\path[->](4,-0.1)edge(4,-0.9);
\path[->](4.9,-1)edge(4.1,-1);
\path[->](5.9,-3)edge(5.1,-3);
\path[->](5,-1.1)edge(5,-2.9);

\path[->](-0.25,3.85)edge(-0.75,3.85);
\path[->](-0.25,3.15)edge(-0.75,3.15);
\path[->](-0.15,3.75)edge(-0.15,3.25);
\path[->](-0.85,3.75)edge(-0.85,3.25);

\path[->](0.75,2.85)edge(0.25,2.85);
\path[->](0.75,2.15)edge(0.25,2.15);
\path[->](0.85,2.75)edge(0.85,2.25);
\path[->](0.15,2.75)edge(0.15,2.25);

\path[->](1.75,1.85)edge(1.25,1.85);
\path[->](1.75,1.15)edge(1.25,1.15);
\path[->](1.85,1.75)edge(1.85,1.25);
\path[->](1.15,1.75)edge(1.15,1.25);

\path[->](2.75,0.85)edge(2.25,0.85);
\path[->](2.75,0.15)edge(2.25,0.15);
\path[->](2.85,0.75)edge(2.85,0.25);
\path[->](2.15,0.75)edge(2.15,0.25);

\path[->](3.75,-0.15)edge(3.25,-0.15);
\path[->](3.75,-0.85)edge(3.25,-0.85);
\path[->](3.85,-0.25)edge(3.85,-0.75);
\path[->](3.15,-0.25)edge(3.15,-0.75);

\node(y1)at(0.5,2.5){\large{2}};
\node(y1)at(1.5,1.5){\large{3}};
\node(y1)at(2.5,0.5){\large{2}};


\fill(-5,4)circle [radius=2pt];
\fill(-5,3)circle [radius=2pt];
\fill(-3,3)circle [radius=2pt];
\fill(-3,2)circle [radius=2pt];
\fill(-2,2)circle [radius=2pt];
\fill(-2,1)circle [radius=2pt];
\fill(-1,1)circle [radius=2pt];
\fill(-1,0)circle [radius=2pt];
\fill(0,0)circle [radius=2pt];
\fill(0,-1)circle [radius=2pt];
\fill(1,-1)circle [radius=2pt];
\fill(1,-2)circle [radius=2pt];
\fill(2,-2)circle [radius=2pt];
\fill(2,-3)circle [radius=2pt];
\fill(3,-3)circle [radius=2pt];
\fill(3,-5)circle [radius=2pt];
\fill(4,-5)circle [radius=2pt];

\fill(-2.85,1.15)circle [radius=2pt];
\fill(-2.85,1.85)circle [radius=2pt];
\fill(-2.15,1.15)circle [radius=2pt];
\fill(-2.15,1.85)circle [radius=2pt];

\fill(-1.85,0.15)circle [radius=2pt];
\fill(-1.85,0.85)circle [radius=2pt];
\fill(-1.15,0.15)circle [radius=2pt];
\fill(-1.15,0.85)circle [radius=2pt];

\fill(-0.85,-0.85)circle [radius=2pt];
\fill(-0.85,-0.15)circle [radius=2pt];
\fill(-0.15,-0.85)circle [radius=2pt];
\fill(-0.15,-0.15)circle [radius=2pt];

\fill(0.15,-1.85)circle [radius=2pt];
\fill(0.15,-1.15)circle [radius=2pt];
\fill(0.85,-1.85)circle [radius=2pt];
\fill(0.85,-1.15)circle [radius=2pt];

\fill(1.15,-2.85)circle [radius=2pt];
\fill(1.15,-2.15)circle [radius=2pt];
\fill(1.85,-2.85)circle [radius=2pt];
\fill(1.85,-2.15)circle [radius=2pt];


\path[->](-5,3.9)edge(-5,3.1);
\path[->](-3.1,3)edge(-4.9,3);
\path[->](-3,2.9)edge(-3,2.1);
\path[->](-2.1,2)edge(-2.9,2);
\path[->](-2,1.9)edge(-2,1.1);
\path[->](-1.1,1)edge(-1.9,1);
\path[->](-1,0.9)edge(-1,0.1);
\path[->](-0.1,0)edge(-0.9,0);
\path[->](0,-0.1)edge(0,-0.9);
\path[->](0.9,-1)edge(0.1,-1);
\path[->](1,-1.1)edge(1,-1.9);
\path[->](1.9,-2)edge(1.1,-2);
\path[->](2,-2.1)edge(2,-2.9);
\path[->](2.9,-3)edge(2.1,-3);
\path[->](3.9,-5)edge(3.1,-5);
\path[->](3,-3.1)edge(3,-4.9);

\path[->](-2.25,1.85)edge(-2.75,1.85);
\path[->](-2.25,1.15)edge(-2.75,1.15);
\path[->](-2.15,1.75)edge(-2.15,1.25);
\path[->](-2.85,1.75)edge(-2.85,1.25);

\path[->](-1.25,0.85)edge(-1.75,0.85);
\path[->](-1.25,0.15)edge(-1.75,0.15);
\path[->](-1.15,0.75)edge(-1.15,0.25);
\path[->](-1.85,0.75)edge(-1.85,0.25);

\path[->](-0.25,-0.15)edge(-0.75,-0.15);
\path[->](-0.25,-0.85)edge(-0.75,-0.85);
\path[->](-0.15,-0.25)edge(-0.15,-0.75);
\path[->](-0.85,-0.25)edge(-0.85,-0.75);

\path[->](0.75,-1.15)edge(0.25,-1.15);
\path[->](0.75,-1.85)edge(0.25,-1.85);
\path[->](0.85,-1.25)edge(0.85,-1.75);
\path[->](0.15,-1.25)edge(0.15,-1.75);

\path[->](1.75,-2.15)edge(1.25,-2.15);
\path[->](1.75,-2.85)edge(1.25,-2.85);
\path[->](1.85,-2.25)edge(1.85,-2.75);
\path[->](1.15,-2.25)edge(1.15,-2.75);

\node(y1)at(-1.5,0.5){\large{2}};
\node(y1)at(-0.5,-0.5){\large{3}};
\node(y1)at(0.5,-1.5){\large{2}};

\end{tikzpicture}
}
\end{center}
\caption{The complex $\CFKi(P(-2,9,9))$ after a change of basis, with the generators on the staircase labeled using the staircase notation of Section \ref{sec:background}. Here an integer $k$ in the center of a box indicates the presence of $k$ boxes on the corresponding diagonal. In the notation of Lemma \ref{lemma:split1}, we have $y_0=U^4z_8^1$, $y_1=U^4z^1_7$, $x_{1,1}=U^4z^1_6$, and similarly for the bottom half of the complex.}\label{fig:Pn299}
\end{figure}
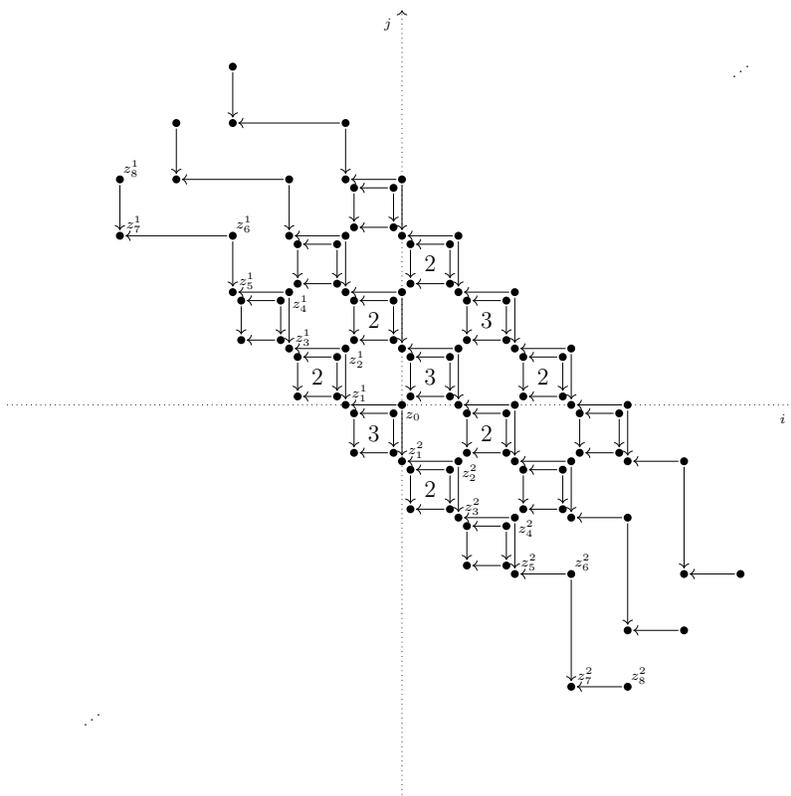

\subsection{Computations for model complexes} \label{subsec:model}

In this subsection we carry out computations for several model complexes. Our computations are modelled on \cite[Section 8]{HM:involutive}. We begin by using the description of the complexes $\CFKi(P(-2,m,n))$ given in the previous section to prove a simplifying lemma.

\begin{lemma} \label{lemma:simplify} Let $K$ be a pretzel knot $P(-2,m,n), m \geq n > 3$, and $C'\otimes \F[U,U^{-1}]$ be the representative for $\CFKi(K)$ of Lemma~\ref{lemma:split1}, consisting of the tensor product of a negative staircase and some number of square complexes with $\F[U,U^{-1}]$. Then
\begin{enumerate}
\item The skew-filtered chain map $\iota_K$ exchanges the planar gradings $(i,j)$ and $(j,i)$. \label{simplify1}
\item If $H$ is a filtered chain homotopy which raises the homological grading by one, $H=0$.  Hence ${\iota_K}^2=\sigma$. \label{simplify2} \end{enumerate} \end{lemma}

\begin{proof}
As previously in this section, let $g = \frac{m+n}{2}$. Let $[x;i,j]$ be a generator of $\CFK^{\infty}(K)$ with homological grading $s$. By the proof of Lemma \ref{lemma:split1}, we see that for any nonexceptional generator, there is some constant $k$ such that $s=i+j-k$. By recalling that the homological grading of $x_{1,1}=[x_{1,1};2-g, -1]$ is zero, we may solve for $k$, determining that for any  non-exceptional generator,
\begin{equation} \label{eq:grading}
s=i+j+g-1.\end{equation}
Furthermore, the non-exceptional generators of even homological grading lie between or on the lines $j-i = g-3$ and $j-i = 3-g$. In particular, we see that
\[ i + g-3 \geq j \geq i + 3-g.\]
Combining this with \eqref{eq:grading}, which we may rearrange to be $j=s-i-g+1$, we obtain
which may be rearranged to
\[ \frac{s}{2}-g+2 \leq i \leq \frac{s}{2}-1.\]
A symmetric computation shows the same is true of $j$, so when $s$ is even we have that
\[ \frac{s}{2}-g+2 \leq i,j \leq \frac{s}{2}-1.\]
On the other hand, the generators of odd homological gradings lie between or on the diagonal lines $j-i=g-4$ and $i-j=4-g$, so the analogous computation shows that
\[ \frac{s+1}{2}-g +2 \leq i, j \leq \frac{s-1}{2} -1.\]

\noindent If $[y;i,j]$ is an exceptional generator, then if $s$ is even either $(i,j)=\left(\frac{s}{2}-g,\frac{s}{2}\right)$ or $(i,j)=\left(\frac{s}{2}, \frac{s}{2}-g\right)$. These cases are $U$-translates of $y_1$ and $y_4$ respectively. If $s$ is odd then we have either $(i,j)=\left(\frac{s+1}{2}-g,\frac{s-1}{2}\right)$ or $(i,j)=\left(\frac{s-1}{2}, \frac{s+1}{2}-g\right)$. These cases are $U$-translates of $y_2$ and $y_3$ respectively. In either case, $s =i+j+g$.

We now prove part (\ref{simplify1}). Suppose we have an ordinary generator $[z; i,j]$ of homological grading $s$. Then $\iota_K([z;i,j])$ is a sum of generators $[x;i',j']$ such that $i'\leq j, j'\leq i$, and $\gr([x,i',j']) = s$. These restrictions imply that for any $[x;i',j']$ an ordinary generator which appears in this sum with nonzero coefficient, we must then have $i'=j$ and $j'=i$. Now suppose that $[y;i',j']$ an exceptional generator which appears in this sum with nonzero coefficient. Then if $s$ is even, one of $i'$ or $j'$ is $\frac{s}{2}$. However, $i,j\leq \frac{s}{2}-1$, so this contradicts $i'\leq j$ and $j'\leq i$. Similarly if $s$ is odd, then one of $i'$ or $j'$ is $\frac{s-1}{2}$, but $i,j \leq \frac{s-1}{2}-1$. Ergo the sum $\iota_K([z; i,j])$ cannot contain any exceptional generators with nonzero coefficient, and indeed $\iota_K([z;i,j])$ must be a sum of ordinary generators $[x;i',j']$ with $i'=j$ and $j'=i$. An analogous but simpler argument  shows that for the four exceptional generators, $\iota_K([y_{\ell};i,j])$ is either $[y_{5-\ell}; j,i]$ or zero. We conclude that $\iota_K$ interchanges the planar gradings $(i,j)$ and $(j,i)$. This proves (\ref{simplify1}). 

The reader is at this point invited to examine Figure \ref{fig:Pn299} as an example and compare it to the proof above, for clarity. The point is that for any generator $[x;i,j]$ in the basis shown, all elements in filtration level $\mathcal F_{j,i}$ with the same homological grading as $[x;i,j]$ have the same type (exceptional or ordinary) as $[x;i,j]$, and therefore grading considerations imply that $\iota_K$ must simply interchange the planar gradings.

We now turn our attention to part (\ref{simplify2}). Let $H$ be a filtered map which raises the homological grading by one. Suppose the map $H$ sends a generator $[z;i,j]$ in homological grading $s$ to a nonzero sum of generators $[x;i',j']$ in homological grading $t=s+1$ with $i'\leq i$ and $j'\leq j$. For any $[x;i',j']$ appearing with nonzero coefficient in this sum, if $[x;i',j']$ is an ordinary generator then  $t=s+1 =i'+j'+g-1$ and if $[x; i',j']$ is an exceptional generator then $t=s+1=i'+j'+g$. We deal with each of these two possibilities separately.

In the first case, if $s+1=i'+j'+g-1$, then since $i'\leq i$ and $j'\leq j$, we have $s+1\leq i+j+g-1$, which implies that $s\leq i+j+ g-2$. This is a contradiction since for any generator $[z;i,j]$, we either have that $s=i+j+g-1$ (for an ordinary generator) or $s=i+j+g$ (for an exceptional generator) is true.  

In the second case, if $s+1=i'+j'+g$, then we see that the generator $[x;i',j']$ is an exceptional generator. As before it follows that $s\leq i+j+g-1$, so we see that the original generator $[z;i,j]$ must be ordinary. Moreover, the inequality must actually be an equality, implying that $i'=i$ and $j'=j$. Therefore, $H$ must send the ordinary generator $[z; i,j]$ to a sum of exceptional generators $[x;i,j]$ in homological grading $s+1$ lying in the original planar grading $(i,j)$. However, no exceptional generators lie in the same planar gradings as ordinary generators anywhere in the chain complex, so this is impossible. This implies that $H=0$, proving (\ref{simplify2}).
\end{proof}

Lemma \ref{lemma:simplify} is the important step in proving that the strong equivalence class of the pair $(\CFK^{\infty}(P(-2,m,n)), \iota_K)$ admits a representative which decomposes into a direct sum of simple complexes preserved by $\iota_K$. Since it is no longer necessary to distinguish the ordinary and exceptional generators, from here on out we will use the notation for staircase complexes introduced in Section \ref{sec:background}.

\begin{lemma} \label{lemma:split2} Let $K$ be a pretzel knot $P(-2,m,n)$ with $m\geq n>3$.  Then the strong equivalence class of $(\CFK^\infty(K),\iota_K)$ admits a representative which decomposes equivariantly as the tensor product of $\F[U,U^{-1}]$ with a direct sum of pairs of square complexes interchanged by the standard square map, and one of four model complexes, according to the values of $m$ and $n$, as follows:
\begin{itemize}
\item When $m\equiv n\equiv 1\pmod 4$, we obtain the model complex $C_1$ consisting of a negative staircase with $n(K)=\frac{m+n-2}{4}$ and a single square complex on the main diagonal, as in Figure~\ref{fig:C1}.
 
\item When $m\equiv 3\pmod 4$ and $n\equiv 1\pmod 4$, we obtain the model complex $C_2$ consisting of a negative staircase with $n(K)=\frac{m+n}{4}$, as in Figure~\ref{fig:C2}.

\item When $m\equiv n\equiv 3\pmod 4$, we obtain the model complex $C_3$ consisting of a negative staircase with $n(K) = \frac{m+n-2}{4}$, as in Figure ~\ref{fig:C3}.

\item When $m\equiv 1\pmod 4$ and $n\equiv 3\pmod 4$, we obtain the model complex $C_4$, consisting of a negative staircase with $n(K)= \frac{m+n}{4}$, as in Figure~\ref{fig:C4}.
\end{itemize}
\end{lemma}
\begin{proof}
The argument that we can split off pairs of square complexes until we are left with a staircase complex and at most one square complex on the main diagonal proceeds as in \cite[Proposition 8.1]{HM:involutive}, using Lemma \ref{lemma:simplify}. It remains to analyze the staircases involved and determine whether we have a square complex remaining on the main diagonal.

We first note that in all cases by Lemma~\ref{lemma:split1} we have a negative staircase summand which begins at the top of the staircase with a vertical step of length one followed by a horizontal step of length two, subsequent to which all steps in the top half of the complex have length one. Let $v$ be the total number of steps in the top half of the complex and $u=v-2$ be the number of steps in the top half of the complex after the first two. Then since the total lengths of the steps in the complex sum to $g=\frac{m+n}{2}$ by projecting the elements of the staircase to their $U$-translates in the $\mathbb F$-subcomplex $C\{i=0\}$ of our representative for $\CFKi(K)$, we have that
\[
u = g-3.
\]
\noindent We see that if $m\equiv n \pmod 4$  then $m+n\equiv 2\pmod 4$, implying that $g=\frac{m+n}{2}$ is odd. Hence $u=g-3$ is an even number, and therefore so is $v$, so the number of steps above the main diagonal is even. However, if it is the case that $m\not\equiv n \pmod 4$, then $4\mid(m+n)$ so there is an odd number of steps above the main diagonal.

Before moving on, we also calculate the number $n(K)$ associated to the negative staircase summand of the complex. If $m \equiv n \parmod{4}$, then $g$ is odd and $u$ is even, and the total lengths of the vertical arrows in the top half of the complex is equal to $\frac{u}{2}+1$, so we have
\[
n(K) = \frac{g-3}{2} + 1 = \frac{g-1}{2} = \frac{m+n-2}{4}.
\]
Conversely if $m \not\equiv n \parmod{4}$, then $u$ is odd, and the total lengths of the vertical arrows in the top half of the complex is $\frac{u+1}{2}+1$, so we have
\[
n(K) = \frac{g-3+1}{2}+ 1 = \frac{g}{2} = \frac{m+n}{4}.
\]

We now turn to the matter of square complexes, which we analyze by considering the Alexander polynomial of the knot, which we recall is equal to the Euler characteristic of the knot Floer homology $\HFKhat(K)$, which is itself the homology of the associated graded of the $\mathbb F$-complex $C\{i=0\}$ inside any representative of $\CFKi(K)$. In general the tensor product of a square complex with its upper right corner on the diagonal $j=i+s$ with $\F[U,U^{-1}]$ contributes a factor of $\pm t^s(-t+2-t^{-1})$ to the Alexander polynomial of the knot. In our case, because the knot Floer complex is thin apart from the four exceptional generators, the sign is always $(-1)^s$. Taking the Euler characteristic of the knot Floer homology in Lemma~\ref{lemma:hfk}, we see the total Alexander polynomial of $K$ is
\begin{align*}
&t^g-t^{g-1}+\sum_{k=1}^{n-3}(-1)^{k-1}k t^{g-k-2}+\sum_{k=n-g}^{g-n}(-1)^{g-k-1}(n-2)t^{k} + \sum_{k=1}^{n-3}(-1)^{k-1}kt^{k+2-g}-t^{1-g}+t^{-g}.
\end{align*}

\noindent The staircase summand contributes $t^g-t^{g-1}+\left(\sum_{k=3-g}^{g-3}(-1)^{k-1}t^k\right)-t^{1-g}+t^{-g}$ to the polynomial.  Therefore, the boxes must contribute the remaining terms, which are:

\[
\sum_{k=1}^{n-4}(-1)^{k}k t^{g-k-3}+\sum_{k=n-g}^{g-n}(-1)^{g-k-1}(n-3)t^{k} + \sum_{k=1}^{n-4}(-1)^{k}kt^{k+3-g}.
\]

\noindent This polynomial factors as
\[\left(\sum_{k=1}^{\frac{n-5}2} kt^{g-2k-3} +\frac{n-3}{2}\sum_{0}^{g-n} t^{g-n-2k}+ \sum_{k=1}^{\frac{n-5}2}(-1)^kkt^{2k+3-g} \right)(-t+2-t^{-1}).
\]
In particular, note that the center sum $\frac{n-3}{2}\sum_{0}^{g-n} t^{g-n-2k}$ can be rewritten as $\frac{n-3}{2}(t^{g-n} + t^{g-n-2} + \dots + t^{n+2-g} + t^{n-g})$. If $m \not\equiv n \parmod{4}$, then the constant term in this polynomial is zero, and we conclude that there are no boxes on the main diagonal in our representative for $\CFKi(K)$. If $m \equiv n \parmod{4}$, then the constant term in this polynomial is $\frac{n-3}{2}$, and we conclude that there are $\frac{n-3}{2}$ boxes on the main diagonal. If $n \equiv 3 \parmod{4}$, this is an even number of boxes, which split off equivariantly in pairs. If $n \equiv 1 \parmod{4}$, this is an odd number of boxes, all but one of which split off equivariantly in pairs, leaving a single box on the main diagonal in the model complex. \end{proof}

\begin{example} In light of the high algebraic complexity of the previous proof, we include a few early cases. First we consider $K_1 = P(-2,5,5)$, an example of the case $m\equiv n \equiv 1 \parmod{4}$. The Alexander polynomial of $K_1$ is
\begin{align*}
\Delta_{K_1}(t) &= t^5-t^4+t^2-2t+3-2t^{-1}+t^{-2}-t^{-4}+ t^{-5} \\
				&= (t^5-t^4+t^2-t+1-t^{-1}+t^{-2}-t^{-4}+t^{-5}) + (-t+2-t^{-1})
\end{align*}
where the two parenthesized terms in the second row correspond in order to the contribution of the staircase and the contribution of a single box on the main diagonal. This gives us the complex of Figure \ref{fig:C1}; in this special case we do not need to split off any pairs of boxes to obtain the model complex shown. There are as promised $\frac{n-3}{2}=1$ boxes on the main diagonal. For a more complicated example of the case $m\equiv n \equiv 1 \parmod{4}$, consider $K_1' = P(-2,9,5)$, which has Alexander polynomial
\begin{align*}
\Delta_{K_1'}(t) &= t^7-t^6+t^4-2t^3+3t^2 - 3t + 3 -3t^{-1} + 3t^{-2} - 2t^{-3} + t^{-4} - t^{-6} + t^{-7}
\end{align*}
After subtracting the staircase contribution
\[t^7-t^6+t^4-t^3+t^2 - t + 1 -t^{-1} + t^{-2} - t^{-3} + t^{-4} - t^{-6} +t^{-7}\]
we are left with			
\begin{align*}			
-t^3 + 2t^2 - 2t + 2 -2t^{-1} + 2t^{-2} - t^{-3} = (-t+2-t^{-1})(t^2+1+t^{-2})
\end{align*}
We see we have a single box on the main diagonal, as promised by $\frac{n-3}{2}=1$, along with a pair of boxes on the diagonals $j-i=1$ and $j-i=-1$ which are interchanged by the standard square map; we split off this pair to obtain a version of the model complex $C_1$ of Figure~\ref{fig:C1}.

Next, let $K_2=P(-2,7,5)$ be an example of the case $m\equiv 3 \parmod{4}$ and $n\equiv 1 \parmod{4}$. Then we have
\[
\Delta_{K_2} = t^{6}-t^{5} + t^3 - 2t^2 + 3t - 3 + 3t^{-1} -2t^{-2} + t^{-3} -t^{-5} + t^{-6}\]
After subtracting the staircase contribution \[t^{6}-t^{5} + t^3 - t^2 + t - 1 + t^{-1} -t^{-2} + t^{-3} -t^{-5} + t^{-6}\] we are left with
\[-t^2+2t-2+2t^{-1}+t^2\]
which factors as
\[
(-t+2-t^{-1})(t+t^{-1})
\]
corresponding to the appearance of a pair of boxes on the diagonals $j-i=1$ and $j-i=-1$, interchanged by the standard square map. After splitting off this pair we obtain the staircase model complex $C_2$ of Figure \ref{fig:C2}.

Next, consider $K_3=P(-2,7,7)$ as an example of the case $n\equiv m \equiv 3 \parmod{4}$. The Alexander polynomial of $K_3$ is
\[ \Delta_{K_3}(t) = t^7 - t^{6}+t^4-2t^3+3t^2-4t+5-4t^{-1}+3t^{-2}-2t^{-3}+t^{-4}-t^{-6}+t^{-7}.\]
After subtracting the staircase contribution \[t^7 - t^{6}+t^4-t^3+t^2-t+1-t^{-1}+t^{-2}-t^{-3}+t^{-4}-t^{-6}+t^{-7}\] we are left with
\[ -t^3+2t^2-3t+4-3t^{-1}+2t^{-2}-t^{-3} = (-t+2-t^{-1})(t^2+2+t^{-2}).\]
We see there are as promised $\frac{n-3}{2}=2$ boxes on the main diagonal, which split off as a pair interchanged by the standard square map. Furthermore, there is a pair of boxes on the diagonals $j-i=1$ and $j-i=-1$ which are interchanged by the standard square map. After splitting off these pairs we are left with the staircase model complex $C_3$ of Figure~\ref{fig:C3}.

Finally, let $K_4 = P(-2, 9,7)$ be an example of the case $m \equiv 1 \parmod{4}$ and $n \equiv 3 \parmod{4}$. The Alexander polynomial of $K_4$ is
\[ \Delta_{K_4} = t^8 - t^7 +t^5 -2t^4 +3t^3 -4t^2 +5t-5+5t^{-1} -4t^{-2} + 3t^{-3} -2t^{-4} + t^5 - t^7 + t^8.\]
After subtracting the staircase contribution \[t^8 - t^7 +t^5 - t^4 + t^3 - t^2 + t - 1 + t^{-1} - t^{-2} +t^{-3} - t^{-4} +t^{-5} - t^{-7} +t^{-8}\] we are left with
\[-t^4 +2t^3 -3t^2 +4t-4+4t^{-1} -3t^{-2} + 2t^{-3} -t^{-4} = (-t + 2 - t^{-1})(t^3+2t+2t^{-1}+t^{-3}).\]
There is a pair of boxes on the diagonals $j-i=3$ and $j-i=-3$ interchanged by the standard square map; two boxes on each of the diagonals $j-i=1$ and $j-i=-1$ interchanged in pairs by standard square maps. After splitting off these summands we are left with the staircase model complex $C_4$ of Figure~\ref{fig:C4}.
\end{example}

\begin{figure}[htb] 
\begin{center}
\scalebox{.75}{
\begin{tikzpicture}\tikzstyle{every node}=[font=\tiny] 
\path[->][dotted](0,-5)edge(0,5);
\path[->][dotted](-5,0)edge(5,0);

\node at (-.25,4.75){$j$};
\node at (4.75,-.25){$i$};

\fill(3,-2)circle [radius=2pt];
\fill(2,-2)circle [radius=2pt];
\fill(0,2)circle [radius=2pt];
\fill(0,1)circle [radius=2pt];
\fill(1,1)circle [radius=2pt];
\fill(1,0)circle [radius=2pt];
\fill(2,0)circle [radius=2pt];
\fill(-2,2)circle [radius=2pt];
\fill(-2,3)circle [radius=2pt];

\fill(.15,.15)circle [radius=2pt];
\fill(.15,.85)circle [radius=2pt];
\fill(.85,.15)circle [radius=2pt];
\fill(.85,.85)circle [radius=2pt];

\node(y1)at (-2.8,2.2){$z_v^1$};
\node(y2)at (-3.2,.8){$z_{v-1}^1$};
\node(y3)at (.8,-3.2){$z_{v-1}^2$};
\node(y4)at (2.2,-2.8){$z_v^2$};
\node(x11)at (-0.8,1.2){$z_2^1$};
\node(x33)at (1.2,-0.8){$z_2^2$};
\node(x1331)at (-.2,0.1){$z_0$};
\node(x1223)at (-1.2,0.1){$z_1^1$};
\node(x3221)at (0,-1.3){$z_1^2$};
\node(x12)at (-.7,-.3){$b$};
\node(x13)at (-.3,-.3){$a$};
\node(x22)at (-.6,-.7){$Ue$};
\node(x22)at (.3,.3){$e$};
\node(x21)at (-.3,-.7){$c$};

\node(y1)at (-1.8,3.2){$U^{-1}z_v^1$};
\node(y4)at (3.2,-1.8){$U^{-1}z_v^2$};

\path[->](-2,2.9)edge(-2,2.1);
\path[->](2.9,-2)edge(2.1,-2);
\path[->](-.1,2)edge(-1.9,2);
\path[->](2,-.1)edge(2,-1.9);
\path[->](0,1.9)edge(0,1.1);
\path[->](1.9,0)edge(1.1,0);
\path[->](.9,1)edge(.1,1);
\path[->](1,.9)edge(1,.1);
\path[->](.75,.85)edge(.25,.85);
\path[->](.75,.15)edge(.25,.15);
\path[->](.85,.75)edge(.85,.25);
\path[->](.15,.75)edge(.15,.25);


\fill(2.0,-3.0)circle [radius=2pt];
\fill(1.0,-3.0)circle [radius=2pt];
\fill(-1.0,1.0)circle [radius=2pt];
\fill(-1.0,0.0)circle [radius=2pt];
\fill(0.0,0.0)circle [radius=2pt];
\fill(0.0,-1.0)circle [radius=2pt];
\fill(1.0,-1.0)circle [radius=2pt];
\fill(-3.0,1.0)circle [radius=2pt];
\fill(-3.0,2.0)circle [radius=2pt];

\fill(-0.85,-0.85)circle [radius=2pt];
\fill(-0.85,-0.15)circle [radius=2pt];
\fill(-0.15,-0.85)circle [radius=2pt];
\fill(-0.15,-0.15)circle [radius=2pt];

\path[->](-3.0,1.9)edge(-3.0,1.1);
\path[->](1.9,-3.0)edge(1.1,-3.0);
\path[->](-1.1,1.0)edge(-2.9,1.0);
\path[->](1.0,-1.1)edge(1.0,-2.9);
\path[->](-1.0,0.9)edge(-1.0,0.1);
\path[->](0.9,-1.0)edge(0.1,-1.0);
\path[->](-0.1,0.0)edge(-0.9,0.0);
\path[->](0.0,-0.1)edge(0.0,-0.9);
\path[->](-0.25,-0.15)edge(-0.75,-0.15);
\path[->](-0.25,-0.85)edge(-0.75,-0.85);
\path[->](-0.15,-0.25)edge(-0.15,-0.75);
\path[->](-0.85,-0.25)edge(-0.85,-0.75);


\fill(4.0,-1.0)circle [radius=2pt];
\fill(3.0,-1.0)circle [radius=2pt];
\fill(1.0,3.0)circle [radius=2pt];
\fill(1.0,2.0)circle [radius=2pt];
\fill(2.0,2.0)circle [radius=2pt];
\fill(2.0,1.0)circle [radius=2pt];
\fill(3.0,1.0)circle [radius=2pt];
\fill(-1.0,3.0)circle [radius=2pt];
\fill(-1.0,4.0)circle [radius=2pt];

\fill(1.15,1.15)circle [radius=2pt];
\fill(1.15,1.85)circle [radius=2pt];
\fill(1.85,1.15)circle [radius=2pt];
\fill(1.85,1.85)circle [radius=2pt];

\path[->](-1.0,3.9)edge(-1.0,3.1);
\path[->](3.9,-1.0)edge(3.1,-1.0);
\path[->](0.9,3.0)edge(-0.9,3.0);
\path[->](3.0,0.9)edge(3.0,-0.9);
\path[->](1.0,2.9)edge(1.0,2.1);
\path[->](2.9,1.0)edge(2.1,1.0);
\path[->](1.9,2.0)edge(1.1,2.0);
\path[->](2.0,1.9)edge(2.0,1.1);
\path[->](1.75,1.85)edge(1.25,1.85);
\path[->](1.75,1.15)edge(1.25,1.15);
\path[->](1.85,1.75)edge(1.85,1.25);
\path[->](1.15,1.75)edge(1.15,1.25);
\end{tikzpicture}
}
\end{center}
\caption{The model complex $C_1$ tensored with $\mathbb F[U,U^{-1}]$, corresponding to knots $P(-2,m,n)$ such that $m\equiv n\equiv 1\pmod 4$. The case of $m=n=5$ is shown.}
\label{fig:C1}
\end{figure}
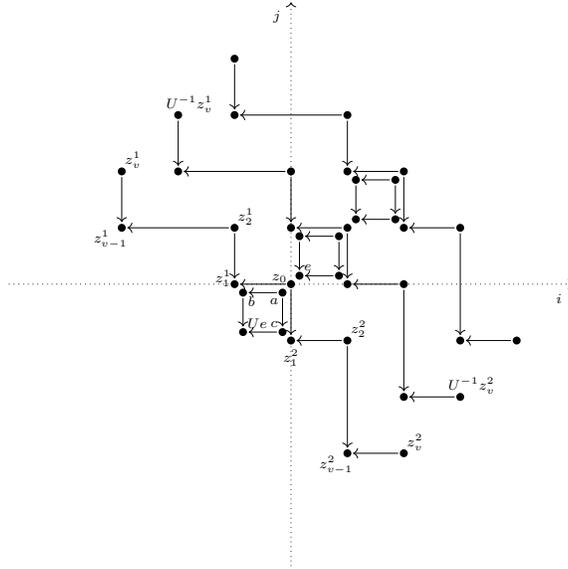

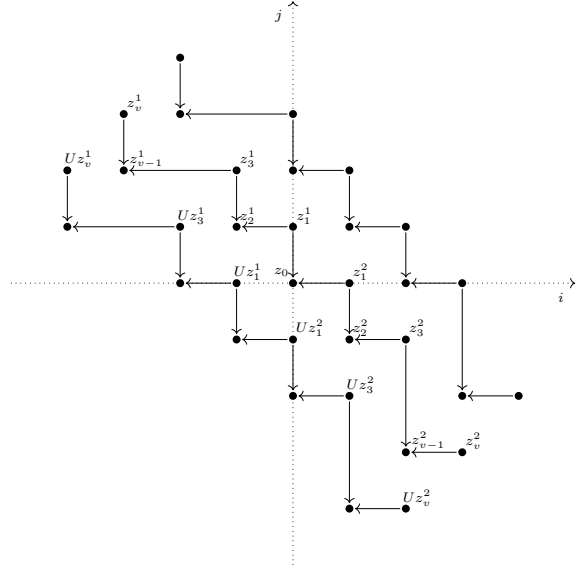
\begin{figure}[htb]
\begin{center}
\scalebox{.75}{
\begin{tikzpicture}\tikzstyle{every node}=[font=\tiny] 
\path[->][dotted](0,-5)edge(0,5);
\path[->][dotted](-5,0)edge(5,0);

\node at (-.25,4.75){$j$};
\node at (4.75,-.25){$i$};

\fill(-3,3)circle [radius=2pt];
\fill(-3,2)circle [radius=2pt];
\fill(-1,2)circle [radius=2pt];
\fill(-1,1)circle [radius=2pt];
\fill(0,1)circle [radius=2pt];
\fill(0,0)circle [radius=2pt];
\fill(1,0)circle [radius=2pt];
\fill(1,-1)circle [radius=2pt];
\fill(2,-1)circle [radius=2pt];
\fill(2,-3)circle [radius=2pt];
\fill(3,-3)circle [radius=2pt];

\node(y1)at (-2.8,3.2){$z_v^1$};
\node(y1)at (-2.6,2.2){$z_{v-1}^1$};
\node(y1)at (-0.8,2.2){$z_3^1$};
\node(y1)at (-0.8,1.2){$z_2^1$};
\node(y1)at (0.2,1.2){$z_1^1$};
\node(y1)at (-.2,.2){$z_0$};
\node(y1)at (1.2,0.2){$z_1^2$};
\node(y1)at (1.2,-0.8){$z_2^2$};
\node(y1)at (2.2,-0.8){$z_3^2$};
\node(y1)at (2.4,-2.8){$z_{v-1}^2$};
\node(y1)at (3.2,-2.8){$z_v^2$};

\path[->](-3,2.9)edge(-3,2.1);
\path[->](-1.1,2)edge(-2.9,2);
\path[->](-1,1.9)edge(-1,1.1);
\path[->](-.1,1)edge(-.9,1);
\path[->](0,.9)edge(0,.1);
\path[->](.9,0)edge(.1,0);
\path[->](1,-.1)edge(1,-.9);
\path[->](1.9,-1)edge(1.1,-1);
\path[->](2,-1.1)edge(2,-2.9);
\path[->](2.9,-3)edge(2.1,-3);


\fill(-2,4)circle [radius=2pt];
\fill(-2,3)circle [radius=2pt];
\fill(0,3)circle [radius=2pt];
\fill(0,2)circle [radius=2pt];
\fill(1,2)circle [radius=2pt];
\fill(1,1)circle [radius=2pt];
\fill(2,1)circle [radius=2pt];
\fill(2,0)circle [radius=2pt];
\fill(3,0)circle [radius=2pt];
\fill(3,-2)circle [radius=2pt];
\fill(4,-2)circle [radius=2pt];

\path[->](-2,3.9)edge(-2,3.1);
\path[->](-0.1,3)edge(-1.9,3);
\path[->](0,2.9)edge(0,2.1);
\path[->](0.9,2)edge(0.1,2);
\path[->](1,1.9)edge(1,1.1);
\path[->](1.9,1)edge(1.1,1);
\path[->](2,0.9)edge(2,0.1);
\path[->](2.9,0)edge(2.1,0);
\path[->](3,-0.1)edge(3,-1.9);
\path[->](3.9,-2)edge(3.1,-2);


\fill(-4,2)circle [radius=2pt];
\fill(-4,1)circle [radius=2pt];
\fill(-2,1)circle [radius=2pt];
\fill(-2,0)circle [radius=2pt];
\fill(-1,0)circle [radius=2pt];
\fill(-1,-1)circle [radius=2pt];
\fill(0,-1)circle [radius=2pt];
\fill(0,-2)circle [radius=2pt];
\fill(1,-2)circle [radius=2pt];
\fill(1,-4)circle [radius=2pt];
\fill(2,-4)circle [radius=2pt];

\node(y1)at (-3.8,2.2){$Uz_v^1$};
\node(y1)at (-1.8,1.2){$Uz_3^1$};
\node(y1)at (-0.8,0.2){$Uz_1^1$};
\node(y1)at (0.3,-0.8){$Uz_1^2$};
\node(y1)at (1.2,-1.8){$Uz_3^2$};
\node(y1)at (2.2,-3.8){$Uz_v^2$};

\path[->](-4,1.9)edge(-4,1.1);
\path[->](-2.1,1)edge(-3.9,1);
\path[->](-2,0.9)edge(-2,0.1);
\path[->](-1.1,0)edge(-1.9,0);
\path[->](-1,-0.1)edge(-1,-0.9);
\path[->](-0.1,-1)edge(-0.9,-1);
\path[->](0,-1.1)edge(0,-1.9);
\path[->](0.9,-2)edge(0.1,-2);
\path[->](1,-2.1)edge(1,-3.9);
\path[->](1.9,-4)edge(1.1,-4);

\end{tikzpicture}
}
\end{center}
\caption{The model complex $C_2$ tensored with $\mathbb F[U,U^{-1}]$, corresponding to knots $P(-2,m,n)$ such that $m\equiv 3\pmod 4$ and $n\equiv 1\pmod 4$. The case of $m=7$ and $n=5$ is shown.}
 \label{fig:C2}
\end{figure}

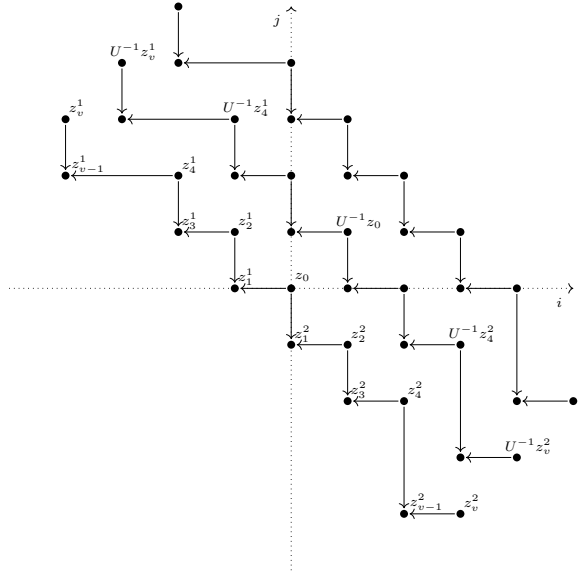
\begin{figure}[htb] 
\begin{center}
\scalebox{.75}{
\begin{tikzpicture}\tikzstyle{every node}=[font=\tiny] 
\path[->][dotted](0,-5)edge(0,5);
\path[->][dotted](-5,0)edge(5,0);

\node at (-.25,4.75){$j$};
\node at (4.75,-.25){$i$};

\fill(-3,4)circle [radius=2pt];
\fill(-3,3)circle [radius=2pt];
\fill(-1,3)circle [radius=2pt];
\fill(-1,2)circle [radius=2pt];
\fill(0,2)circle [radius=2pt];
\fill(0,1)circle [radius=2pt];
\fill(1,1)circle [radius=2pt];
\fill(1,0)circle [radius=2pt];
\fill(2,0)circle [radius=2pt];
\fill(2,-1)circle [radius=2pt];
\fill(3,-1)circle [radius=2pt];
\fill(3,-3)circle [radius=2pt];
\fill(4,-3)circle [radius=2pt];

\node(y1)at (-2.8,4.2){$U^{-1}z_v^1$};
\node(y1)at (-0.8,3.2){$U^{-1}z_4^1$};
\node(y1)at (1.2,1.2){$U^{-1}z_0$};
\node(y1)at (3.2,-0.8){$U^{-1}z_4^2$};
\node(y1)at (4.2,-2.8){$U^{-1}z_v^2$};

\path[->](-3,3.9)edge(-3,3.1);
\path[->](-1.1,3)edge(-2.9,3);
\path[->](-1,2.9)edge(-1,2.1);
\path[->](-.1,2)edge(-.9,2);
\path[->](0,1.9)edge(0,1.1);
\path[->](.9,1)edge(.1,1);
\path[->](1,.9)edge(1,0.1);
\path[->](1.9,0)edge(1.1,0);
\path[->](2,-.1)edge(2,-.9);
\path[->](2.9,-1)edge(2.1,-1);
\path[->](3,-1.1)edge(3,-2.9);
\path[->](3.9,-3)edge(3.1,-3);


\fill(-4,3)circle [radius=2pt];
\fill(-4,2)circle [radius=2pt];
\fill(-2,2)circle [radius=2pt];
\fill(-2,1)circle [radius=2pt];
\fill(-1,1)circle [radius=2pt];
\fill(-1,0)circle [radius=2pt];
\fill(0,0)circle [radius=2pt];
\fill(0,-1)circle [radius=2pt];
\fill(1,-1)circle [radius=2pt];
\fill(1,-2)circle [radius=2pt];
\fill(2,-2)circle [radius=2pt];
\fill(2,-4)circle [radius=2pt];
\fill(3,-4)circle [radius=2pt];

\node(y1)at (-3.8,3.2){$z_v^1$};
\node(y1)at (-3.6,2.2){$z_{v-1}^1$};
\node(y1)at (-1.8,2.2){$z_4^1$};
\node(y1)at (-1.8,1.2){$z_3^1$};
\node(y1)at (-0.8,1.2){$z_2^1$};
\node(y1)at (-0.8,0.2){$z_1^1$};
\node(y1)at (0.2,0.2){$z_0$};
\node(y1)at (0.2,-0.8){$z_1^2$};
\node(y1)at (1.2,-0.8){$z_2^2$};
\node(y1)at (1.2,-1.8){$z_3^2$};
\node(y1)at (2.2,-1.8){$z_4^2$};
\node(y1)at (2.4,-3.8){$z_{v-1}^2$};
\node(y1)at (3.2,-3.8){$z_v^2$};

\path[->](-4,2.9)edge(-4,2.1);
\path[->](-2.1,2)edge(-3.9,2);
\path[->](-2,1.9)edge(-2,1.1);
\path[->](-1.1,1)edge(-1.9,1);
\path[->](-1,0.9)edge(-1,0.1);
\path[->](-0.1,0)edge(-0.9,0);
\path[->](0,-0.1)edge(0,-0.9);
\path[->](0.9,-1)edge(0.1,-1);
\path[->](1,-1.1)edge(1,-1.9);
\path[->](1.9,-2)edge(1.1,-2);
\path[->](2,-2.1)edge(2,-3.9);
\path[->](2.9,-4)edge(2.1,-4);


\fill(-2,5)circle [radius=2pt];
\fill(-2,4)circle [radius=2pt];
\fill(0,4)circle [radius=2pt];
\fill(0,3)circle [radius=2pt];
\fill(1,3)circle [radius=2pt];
\fill(1,2)circle [radius=2pt];
\fill(2,2)circle [radius=2pt];
\fill(2,1)circle [radius=2pt];
\fill(3,1)circle [radius=2pt];
\fill(3,0)circle [radius=2pt];
\fill(4,0)circle [radius=2pt];
\fill(4,-2)circle [radius=2pt];
\fill(5,-2)circle [radius=2pt];

\path[->](-2,4.9)edge(-2,4.1);
\path[->](-0.1,4)edge(-1.9,4);
\path[->](0,3.9)edge(0,3.1);
\path[->](0.9,3)edge(0.1,3);
\path[->](1,2.9)edge(1,2.1);
\path[->](1.9,2)edge(1.1,2);
\path[->](2,1.9)edge(2,1.1);
\path[->](2.9,1)edge(2.1,1);
\path[->](3,0.9)edge(3,0.1);
\path[->](3.9,0)edge(3.1,0);
\path[->](4,-0.1)edge(4,-1.9);
\path[->](4.9,-2)edge(4.1,-2);
\end{tikzpicture}
}
\end{center}
\caption{The model complex $C_3$ tensored with $\mathbb F[U,U^{-1}]$, corresponding to knots $P(-2,m,n)$ such that $m\equiv n\equiv 3\pmod 4$. The case of $m=n=7$ is shown.}
\label{fig:C3}
\end{figure}

\begin{figure}[htb] 
\begin{center}
\scalebox{.75}{
\begin{tikzpicture}\tikzstyle{every node}=[font=\tiny] 
\path[->][dotted](0,-5)edge(0,5);
\path[->][dotted](-5,0)edge(5,0);

\node at (-.25,4.75){$j$};
\node at (4.75,-.25){$i$};

\fill(-4,4)circle [radius=2pt];
\fill(-4,3)circle [radius=2pt];
\fill(-2,3)circle [radius=2pt];
\fill(-2,2)circle [radius=2pt];
\fill(-1,2)circle [radius=2pt];
\fill(-1,1)circle [radius=2pt];
\fill(0,1)circle [radius=2pt];
\fill(0,0)circle [radius=2pt];
\fill(1,0)circle [radius=2pt];
\fill(1,-1)circle [radius=2pt];
\fill(2,-1)circle [radius=2pt];
\fill(2,-2)circle [radius=2pt];
\fill(3,-2)circle [radius=2pt];
\fill(3,-4)circle [radius=2pt];
\fill(4,-4)circle [radius=2pt];

\node(y1)at (-3.8,4.2){$z_v^1$};
\node(y1)at (-3.6,3.2){$z_{v-1}^1$};
\node(y1)at (-1.8,3.2){$z_5^1$};
\node(y1)at (-1.8,2.2){$z_4^1$};
\node(y1)at (-0.8,2.2){$z_3^1$};
\node(y1)at (-0.8,1.2){$z_2^1$};
\node(y1)at (0.2,1.2){$z_1^1$};
\node(y1)at (0.2,0.2){$z_0$};
\node(y1)at (1.2,0.2){$z_1^2$};
\node(y1)at (1.2,-0.8){$z_2^2$};
\node(y1)at (2.2,-0.8){$z_3^2$};
\node(y1)at (2.2,-1.8){$z_4^2$};
\node(y1)at (3.2,-1.8){$z_5^2$};
\node(y1)at (3.4,-3.8){$z_{v-1}^2$};
\node(y1)at (4.2,-3.8){$z_v^2$};

\path[->](-4,3.9)edge(-4,3.1);
\path[->](-2.1,3)edge(-3.9,3);
\path[->](-2,2.9)edge(-2,2.1);
\path[->](-1.1,2)edge(-1.9,2);
\path[->](-1,1.9)edge(-1,1.1);
\path[->](-.1,1)edge(-.9,1);
\path[->](0,.9)edge(0,.1);
\path[->](.9,0)edge(.1,0);
\path[->](1,-.1)edge(1,-.9);
\path[->](1.9,-1)edge(1.1,-1);
\path[->](2,-1.1)edge(2,-1.9);
\path[->](2.9,-2)edge(2.1,-2);
\path[->](3,-2.1)edge(3,-3.9);
\path[->](3.9,-4)edge(3.1,-4);


\fill(-3,5)circle [radius=2pt];
\fill(-3,4)circle [radius=2pt];
\fill(-1,4)circle [radius=2pt];
\fill(-1,3)circle [radius=2pt];
\fill(0,3)circle [radius=2pt];
\fill(0,2)circle [radius=2pt];
\fill(1,2)circle [radius=2pt];
\fill(1,1)circle [radius=2pt];
\fill(2,1)circle [radius=2pt];
\fill(2,0)circle [radius=2pt];
\fill(3,0)circle [radius=2pt];
\fill(3,-1)circle [radius=2pt];
\fill(4,-1)circle [radius=2pt];
\fill(4,-3)circle [radius=2pt];
\fill(5,-3)circle [radius=2pt];

\path[->](-3,4.9)edge(-3,4.1);
\path[->](-1.1,4)edge(-2.9,4);
\path[->](-1,3.9)edge(-1,3.1);
\path[->](-0.1,3)edge(-0.9,3);
\path[->](0,2.9)edge(0,2.1);
\path[->](0.9,2)edge(0.1,2);
\path[->](1,1.9)edge(1,1.1);
\path[->](1.9,1)edge(1.1,1);
\path[->](2,0.9)edge(2,0.1);
\path[->](2.9,0)edge(2.1,0);
\path[->](3,-0.1)edge(3,-0.9);
\path[->](3.9,-1)edge(3.1,-1);
\path[->](4,-1.1)edge(4,-2.9);
\path[->](4.9,-3)edge(4.1,-3);


\fill(-5,3)circle [radius=2pt];
\fill(-5,2)circle [radius=2pt];
\fill(-3,2)circle [radius=2pt];
\fill(-3,1)circle [radius=2pt];
\fill(-2,1)circle [radius=2pt];
\fill(-2,0)circle [radius=2pt];
\fill(-1,0)circle [radius=2pt];
\fill(-1,-1)circle [radius=2pt];
\fill(0,-1)circle [radius=2pt];
\fill(0,-2)circle [radius=2pt];
\fill(1,-2)circle [radius=2pt];
\fill(1,-3)circle [radius=2pt];
\fill(2,-3)circle [radius=2pt];
\fill(2,-5)circle [radius=2pt];
\fill(3,-5)circle [radius=2pt];

\node(y1)at (-4.8,3.2){$Uz_v^1$};
\node(y1)at (-2.8,2.2){$Uz_5^1$};
\node(y1)at (-1.8,1.2){$Uz_3^1$};
\node(y1)at (-0.8,0.2){$Uz_1^1$};
\node(y1)at (0.2,-0.8){$Uz_1^2$};
\node(y1)at (1.2,-1.8){$Uz_3^2$};
\node(y1)at (2.2,-2.8){$Uz_5^2$};
\node(y1)at (3.2,-4.8){$Uz_v^2$};

\path[->](-5,2.9)edge(-5,2.1);
\path[->](-3.1,2)edge(-4.9,2);
\path[->](-3,1.9)edge(-3,1.1);
\path[->](-2.1,1)edge(-2.9,1);
\path[->](-2,0.9)edge(-2,0.1);
\path[->](-1.1,0)edge(-1.9,0);
\path[->](-1,-0.1)edge(-1,-0.9);
\path[->](-0.1,-1)edge(-0.9,-1);
\path[->](0,-1.1)edge(0,-1.9);
\path[->](0.9,-2)edge(0.1,-2);
\path[->](1,-2.1)edge(1,-2.9);
\path[->](1.9,-3)edge(1.1,-3);
\path[->](2,-3.1)edge(2,-4.9);
\path[->](2.9,-5)edge(2.1,-5);

\end{tikzpicture}
}
\end{center}
\caption{The model complex $C_4$ tensored with $\mathbb F[U,U^{-1}]$, corresponding to knots $P(-2,m,n)$ such that $m\equiv 1\pmod 4$ and $n\equiv 3\pmod 4$. The case of $m=9$ and $n=7$ is shown.}
\label{fig:C4}
\end{figure}
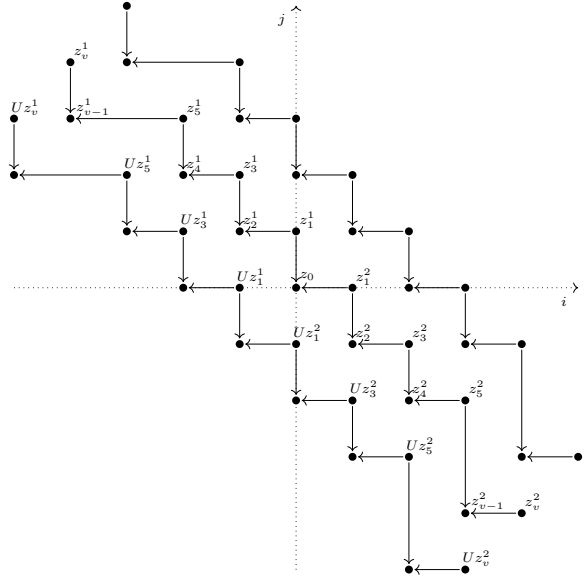

We now conclude with the proof of our main theorem. 

\begin{proof}[Proof of Theorem \ref{thm:main}] We present the argument for the model complex $C_1$ and its dual; the other cases, being staircase complexes, are substantially easier. Our computations are similar to Cases 2(e) and 2(c) of \cite[Proposition 8.2]{HM:involutive} respectively. The expert reader may find it helpful to note that, because of the step of length two in our staircase complex, we in particular get parallel results to the case of thin knots in the case that the Oszv{\'a}th-Szab{\'o} concordance invariant $\tau$ is an even number, even though in our case $\tau=g = \frac{m+n}{2}$ is odd \cite[Section 3]{GMM}. The complex is pictured in Figure ~\ref{fig:C1}, and has generators $a, b, c, e, z_0, z_r^s$ for $s \in \{1,2\}$ and $1 \leq r \leq v$, where $v=\frac{m+n}{2}-1$. The homological gradings are determined by the requirement that $H_*(C\{i=0\})= \mathbb F_{(0)}$, which fixes $\gr(U^{n(K)}z_v^2)=0.$ The complex has nonzero differentials
  \begin{alignat*}{3}
        &\partial a = b+c  & \qquad\qquad & \qquad\qquad &      &\partial b = \partial c = Ue \\
        &\partial(z_0)=z_1^1+z_1^2  & & &      &\partial(z^s_v) = z^s_{v-1}\\
        &\partial(z_r^s) = z_{r-1}^s + z_{r+1}^s, \ 0<r<v, r \text{ even } &&&& \
\end{alignat*}  

\noindent Up to change of basis, the unique skew-filtered map on $C_1$ squaring to the Sarkar involution is
  \begin{alignat*}{3}
        &\iota_K(a)=a+z_0  & \qquad\qquad & \qquad\qquad &      &\iota_K(b)=c+z_1^2 \\
        &\iota_K(z_0)=z_0+e  & & &      &\iota_K(c)=b+z_1^1  \\
        & \iota_K(z_r^1)=z_r^2, 1<r\leq v & & & & \iota_K(z_r^2)=z_r^1, 1 \leq r \leq v \\
        &\iota_K(e)=e & & & & \ 
    \end{alignat*}

We begin by considering the $\mathbb F[U]$-complex of $C_1\otimes \F[U, U^{-1}]$ with $\mathbb F$-basis consisting of elements lying in the third quadrant; via a small abuse of notation as in \cite[Section 8]{HM:involutive}, we refer to this subcomplex as $A_0^-$. We see that this complex is generated over $\F[U]$ by the set 
\[\{e, a, b, c, z_0, z^s_1, Uz^s_2, Uz^s_3, U^2z^s_4, \dots, U^{n(K)}z^s_{v}\}\] 
\noindent and differentials inherited from the differentials on $C_1$. The homology $H_*(A_0^-)$ is isomorphic to
\[H_*(A_0^-)\cong\F_{(0)}[U]\oplus \F_{(2n(K)-1)}[U]/U^{n(K)}   \oplus \F_{(2n(K))}\]
\noindent where the summand $\F_{(0)}[U]$ is generated by $[U^{n(K)}(z^1_v + \dots + z_2^1 + z_0 + z_2^2 + \dots z^2_v)]$, the summand $\F_{(2n(K)-1)}[U]/U^{n(K)}$ is generated by $[z_1^1]$ and the summand $\F_{(2n(K))}$ is generated as an $\F$-vector space by $[e]$. Here, the non-$U$-torsion part of $H_*(A_0^-)$ is generated by an element of grading $0$, so $V_0(K)=0$.

\begin{figure}[htb] 
\begin{center}
\scalebox{.75}{
\begin{tikzpicture}\tikzstyle{every node}=[font=\tiny] 
\path[->][dotted](2,-5)edge(2,5);
\path[->][dotted](-5,2)edge(5,2);

\node at (1.75,4.75){$j$};
\node at (4.75,1.75){$i$};

\fill(2,-2)circle [radius=2pt];
\fill(0,2)circle [radius=2pt];
\fill(0,1)circle [radius=2pt];
\fill(1,1)circle [radius=2pt];
\fill(1,0)circle [radius=2pt];
\fill(2,0)circle [radius=2pt];
\fill(-2,2)circle [radius=2pt];

\fill(.15,.15)circle [radius=2pt];
\fill(.15,.85)circle [radius=2pt];
\fill(.85,.15)circle [radius=2pt];
\fill(.85,.85)circle [radius=2pt];

\node(x1331)at (1.8,2.1){$z_0$};
\node(x1223)at (.8,2.1){$z_1^1$};
\node(x3221)at (2,.7){$z_1^2$};
\node(x12)at (1.3,1.7){$b$};
\node(x13)at (1.7,1.7){$a$};
\node(x22)at (1.4,1.3){$Ue$};
\node(x22)at (2.3,2.3){$e$};
\node(x21)at (1.7,1.3){$c$};

\path[->](-.1,2)edge(-1.9,2);
\path[->](2,-.1)edge(2,-1.9);
\path[->](0,1.9)edge(0,1.1);
\path[->](1.9,0)edge(1.1,0);
\path[->](.9,1)edge(.1,1);
\path[->](1,.9)edge(1,.1);
\path[->](.75,.85)edge(.25,.85);
\path[->](.75,.15)edge(.25,.15);
\path[->](.85,.75)edge(.85,.25);
\path[->](.15,.75)edge(.15,.25);


\fill(2.0,-3.0)circle [radius=2pt];
\fill(1.0,-3.0)circle [radius=2pt];
\fill(-1.0,1.0)circle [radius=2pt];
\fill(-1.0,0.0)circle [radius=2pt];
\fill(0.0,0.0)circle [radius=2pt];
\fill(0.0,-1.0)circle [radius=2pt];
\fill(1.0,-1.0)circle [radius=2pt];
\fill(-3.0,1.0)circle [radius=2pt];
\fill(-3.0,2.0)circle [radius=2pt];

\fill(-0.85,-0.85)circle [radius=2pt];
\fill(-0.85,-0.15)circle [radius=2pt];
\fill(-0.15,-0.85)circle [radius=2pt];
\fill(-0.15,-0.15)circle [radius=2pt];

\node(uy2)at (-1.7,2.2){$Uz_{v-1}^1$};
\node(uy3)at (2.5,-1.8){$Uz_{v-1}^2$};

\node(uy2)at (-3.1,2.2){$U^{n(K)}z_v^1$};
\node(uy2)at (2.4,-2.8){$U^{n(K)}z_v^2$};

\path[->](-3.0,1.9)edge(-3.0,1.1);
\path[->](1.9,-3.0)edge(1.1,-3.0);
\path[->](-1.1,1.0)edge(-2.9,1.0);
\path[->](1.0,-1.1)edge(1.0,-2.9);
\path[->](-1.0,0.9)edge(-1.0,0.1);
\path[->](0.9,-1.0)edge(0.1,-1.0);
\path[->](-0.1,0.0)edge(-0.9,0.0);
\path[->](0.0,-0.1)edge(0.0,-0.9);
\path[->](-0.25,-0.15)edge(-0.75,-0.15);
\path[->](-0.25,-0.85)edge(-0.75,-0.85);
\path[->](-0.15,-0.25)edge(-0.15,-0.75);
\path[->](-0.85,-0.25)edge(-0.85,-0.75);


\fill(1.0,2.0)circle [radius=2pt];
\fill(2.0,2.0)circle [radius=2pt];
\fill(2.0,1.0)circle [radius=2pt];

\fill(1.15,1.15)circle [radius=2pt];
\fill(1.15,1.85)circle [radius=2pt];
\fill(1.85,1.15)circle [radius=2pt];
\fill(1.85,1.85)circle [radius=2pt];

\path[->](1.9,2.0)edge(1.1,2.0);
\path[->](2.0,1.9)edge(2.0,1.1);
\path[->](1.75,1.85)edge(1.25,1.85);
\path[->](1.75,1.15)edge(1.25,1.15);
\path[->](1.85,1.75)edge(1.85,1.25);
\path[->](1.15,1.75)edge(1.15,1.25);

\fill(2.15,2.15)circle [radius=2pt];

\node(ellipsis1)at(-3,-1){$\iddots$};
\node(ellipsis2)at(-2,-2){$\iddots$};
\node(ellipsis3)at(-1,-3){$\iddots$};
\end{tikzpicture}
}
\end{center}
\caption{The $\F[U]$-submodule $A_0^-$ of the complex $C_1\otimes \mathbb F[U,U^{-1}]$. The specific example shown remains $m=n=5$; in this case $n(K)=2$.}
\label{fig:Aminus}
\end{figure}
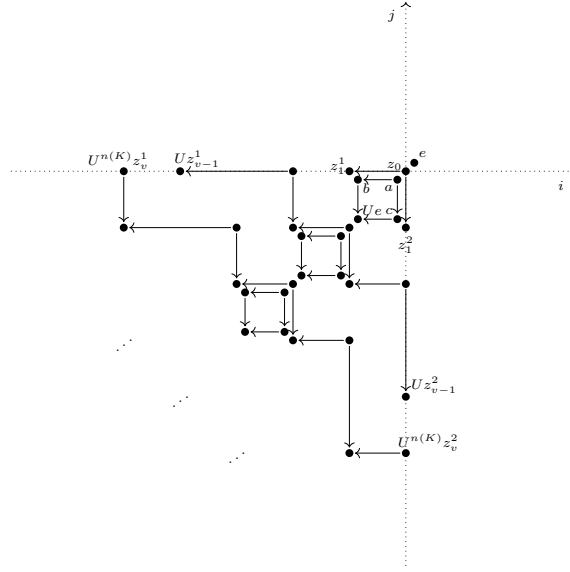

We now compute the involutive correction terms. We start by computing the homology of the grading-shifted mapping cone \[CI^{\infty} = \mathrm{Cone}(C_1\otimes \F[U,U^{-1}] \xrightarrow{Q(1+\iota_K)} Q(C_1\otimes \F[U,U^{-1}])[-1]).\] 
This mapping cone comes to us with a set of generators \[\{e, a, b, c, z_0, z^s_1,\dots, z^s_{v}, Qe, Qa, Qb, Qc, Qz_0, Qz^s_1, \dots, Qz^s_{v}\}\] for $s=1,2$ as usual. We do a change of basis to replace this by the set of generators 
\begin{align*}&\{e, a, b, b+c, z_0, z^1_1, z_1^1+z_1^2 ,\dots, z_1^v, z^1_{v}+z^2_v, \\ & \qquad Qe, Qa, Qb, Q(b+c), Qz_0, Qz^1_1, Qz_1^1+Qz_1^2, \dots, Qz^1_{v}, Q(z^1_v+z_v^2)\}.\end{align*}

\noindent After this change of basis the complex breaks up as a direct sum of the tensor products of three model $\mathbb F$-complexes $D_1$, $D_2$, and $D_3$ shown in Figure~\ref{fig:ConeComplexes} with $\mathbb F[U,U^{-1}]$. The top model complex $D_1$ has generators
\[ \{QU^{-1}c, Qe, z_0, z^1_1+z^2_1, \dots, z_v^1+z_v^2\}\]
and nonzero differentials $\partial^{\iota} = \partial + Q(1+\iota_K)$ given by
    \begin{alignat*}{3}
        &\partial^{\iota}(QU^{-1}c)= Qe \\
        &\partial^{\iota}(z_0)= (z_1^1+z_1^2)+Qe \\
        &\partial^{\iota}(z_v^1+z_v^2) = z_{v-1}^1 + z_{v-2}^2 \\
        &\partial^{\iota}(z_r^1+z_r^2)=(z_{r-1}^1+z_{r-1}^2)+(z_{r+1}^1+z_{r+1}^2), \ \ 0<r<v, \ r\text{ even.} 
    \end{alignat*}
We observe that tensor product of $D_1$ with $\mathbb F[U,U^{-1}]$ contributes a summand $\mathbb F[U,U^{-1}]$ to $H_*(CI^{\infty})$ generated by $[\theta] = [z_0+(z_2^1+z_2^2)+\cdots+(z_v^1+z_v^2)+QU^{-1}c]$. Our next model complex $D_2$, shown on the middle row of Figure~\ref{fig:ConeComplexes}, has generators
\[\{a, b+c, Qz_0, z_1^1, \dots, z_v^1, Q(z_1^1+z_1^2), \dots, Q(z_v^1+z_v^2)\}\]
and nonzero differentials
    \begin{alignat*}{3}
        &\partial^{\iota}(a) = (b+c)+Qz_0   \\
        &\partial^{\iota}(b+c)= Q(z_1^1+z_1^2)  \\
        &\partial^{\iota}(Qz_0)= Q(z_1^1+z_1^2) \\
        &\partial^{\iota}(z_v^1) = z_{v-1}^1+Q(z_v^1+z_v^2) \\
	    & \partial^{\iota}(Q(z_v^1+z_v^2)) = Q(z_{v-1}^1 + z_{v-1}^2)\\
        &\partial^{\iota}(z_r^1)=z_{r-1}^1+z_{r+1}^1+Q(z_r^1+z_r^2), \ \ 0<r< v,\ r\text{ even}\\
        &\partial^{\iota}(z_r^1)=Q(z_r^1+z_r^2), \ \ 0<r< v,\ r\text{ odd}\\
         &\partial^{\iota}(Q(z_r^1+z_r^2))=Q(z_{r-1}^1+z_{r-1}^2)+Q(z_{r+1}^1+z_{r+1}^2), \ \ 0<r<v, \ r\text{ even}
    \end{alignat*}
We observe that the tensor product of $D_2$ with $\mathbb F[U,U^{-1}]$ contributes an $\mathbb F[U,U^{-1}]$ summand to the homology $H_*(CI^{\infty})$ generated by $[Qz_0+z_1^1]=[Q(z_0+z_2^1+z_2^2+\cdots+z_v^1+z_v^2)]=Q[\theta]$. The final model complex $D_3$ has generators
\[\{Qa, Q(b+c), c, e, Qz_1^1, \dots, Qz_v^1\}\]
and nonzero differentials
    \begin{alignat*}{3}
        &\partial^{\iota}(Qa)= Q(b+c)  \\
        &\partial^{\iota}(c) = Ue+Q(b+c)+Qz_1^1 \\
        &\partial^{\iota}(Qz_{v}^1)=Qz_{v-1}^1 \\
         &\partial^{\iota}(Qz_r^1)=Qz_{r-1}^1+Qz_{r+1}^1, \ \ 0<r< v, \ r\text{ even}
    \end{alignat*}
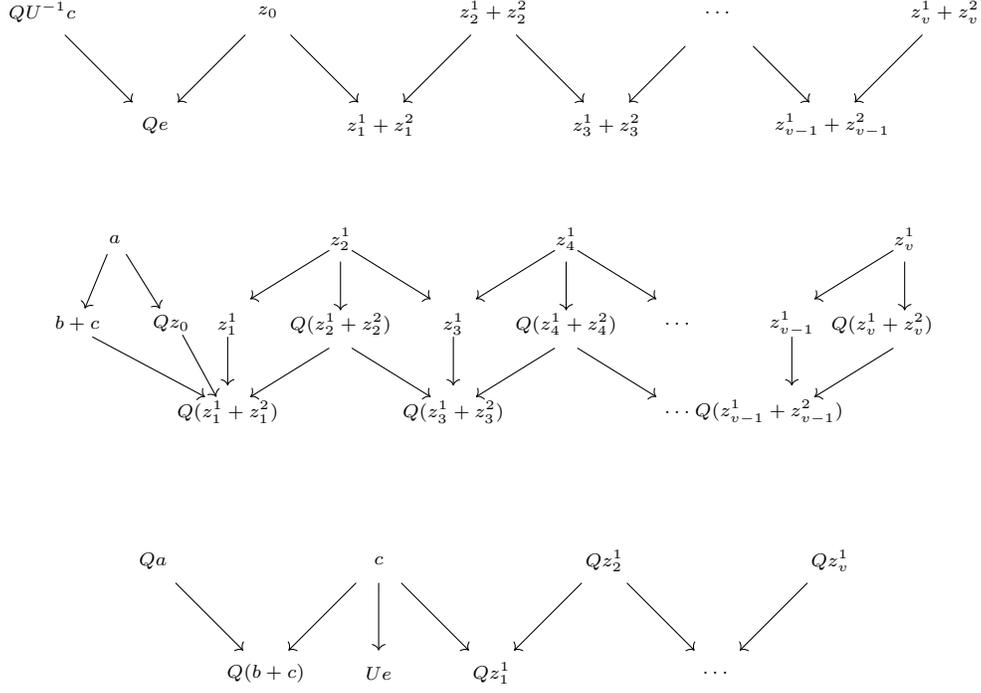
\begin{figure}[htb] 
\begin{center}
\begin{tikzpicture}\tikzstyle{every node}=[font=\tiny]

\node(Qb)at(-6.0,0.0){$QU^{-1}c$};
\node(QUe)at(-4.5,-1.5){$Qe$};
\node(z0)at(-3.0,0.0){$z_0$};
\node(zs)at(-1.5,-1.5){$z_1^1+z_1^2$};
\node(zs)at(0.0,0.0){$z_2^1+z_2^2$};
\node(zs)at(1.5,-1.5){$z_3^1+z_3^2$};
\node(dots)at(3.0,0.0){$\dots$};
\node(zs)at(4.5,-1.5){$z_{v-1}^1+z_{v-1}^2$};
\node(zs)at(6.0,0.0){$z_v^1+z_v^2$};

\path[->](-5.7,-0.3)edge(-4.8,-1.2);
\path[->](-3.3,-0.3)edge(-4.2,-1.2);
\path[->](-2.7,-0.3)edge(-1.8,-1.2);
\path[->](-0.3,-0.3)edge(-1.2,-1.2);
\path[->](0.3,-0.3)edge(1.2,-1.2);
\path[->](2.55,-0.45)edge(1.8,-1.2);
\path[->](3.45,-0.45)edge(4.2,-1.2);
\path[->](5.7,-0.3)edge(4.8,-1.2);
\end{tikzpicture}
\end{center}
\vspace{5mm}
\begin{center}
\begin{tikzpicture}\tikzstyle{every node}=[font=\tiny]

\node(nd)at(-4.5,1.5){$a$};
\node(nd)at(-5,0.375){$b+c$};
\node(nd)at(-3.75,0.375){$Qz_0$};
\node(nd)at(-3.0,0.375){$z_1^1$};
\node(nd)at(-3.0,-0.8){$Q(z_1^1+z_1^2)$};
\node(nd)at(-1.5,1.5){$z_2^1$};
\node(nd)at(-1.5,0.375){$Q(z_2^1+z_2^2)$};
\node(nd)at(0.0,0.375){$z_3^1$};
\node(nd)at(0.0,-0.8){$Q(z_3^1+z_3^2)$};
\node(nd)at(1.5,1.5){$z_4^1$};
\node(nd)at(1.5,0.375){$Q(z_4^1+z_4^2)$};
\node(nd)at(3.0,0.375){$\dots$};
\node(nd)at(4.5,0.375){$z_{v-1}^1$};
\node(nd)at(3.0,-0.8){$\dots$};
\node(nd)at(4.2,-0.8){$Q(z_{v-1}^1+z_{v-1}^2)$};
\node(nd)at(6.0,1.5){$z_v^1$};
\node(nd)at(5.7,0.375){$Q(z_v^1+z_v^2)$};

\path[->](-4.6,1.3)edge(-4.9,0.575);
\path[->](-3.0,0.2)edge(-3.0,-0.45);
\path[->](-1.5,1.2)edge(-1.5,0.55);
\path[->](0.0,0.2)edge(0.0,-0.45);
\path[->](1.5,1.2)edge(1.5,0.55);
\path[->](4.5,0.2)edge(4.5,-0.45);
\path[->](6.0,1.2)edge(6.0,0.55);
\path[->](-4.8,.2)edge(-3.3,-0.6);
\path[->](-1.65,0.05)edge(-2.7,-0.6);
\path[->](-1.35,0.05)edge(-0.3,-0.6);
\path[->](1.35,0.05)edge(0.3,-0.6);
\path[->](1.65,0.05)edge(2.7,-0.6);
\path[->](5.85,0.05)edge(4.8,-0.6);
\path[->](-1.65,1.35)edge(-2.7,0.7);
\path[->](-1.35,1.35)edge(-0.3,0.7);
\path[->](1.35,1.35)edge(0.3,0.7);
\path[->](1.65,1.35)edge(2.7,0.7);
\path[->](5.85,1.35)edge(4.8,0.7);
\path[->](-4.35,1.3)edge(-3.9,0.575);
\path[->](-3.6,0.225)edge(-3.15,-0.6);
\end{tikzpicture}
\begin{center}
\vspace{5mm}
\end{center}
\begin{tikzpicture}\tikzstyle{every node}=[font=\tiny]

\node(Qa)at(-4.5,0.0){$Qa$};
\node(Qbc)at(-3.0,-1.5){$Q(b+c)$};
\node(b)at(-1.5,0.0){$c$};
\node(Ue)at(-1.5,-1.5){$Ue$};
\node(Qz11)at(0.0,-1.5){$Qz_1^1$};
\node(Qz21)at(1.5,0.0){$Qz_2^1$};
\node(dots)at(3.0,-1.5){$\dots$};
\node(Qzv1)at(4.5,0.0){$Qz_v^1$};

\path[->](-4.2,-0.3)edge(-3.3,-1.2);
\path[->](-1.8,-0.3)edge(-2.7,-1.2);
\path[->](-1.5,-0.3)edge(-1.5,-1.2);
\path[->](-1.2,-0.3)edge(-0.3,-1.2);
\path[->](1.2,-0.3)edge(0.3,-1.2);
\path[->](1.8,-0.3)edge(2.7,-1.2);
\path[->](4.2,-0.3)edge(3.3,-1.2);
\end{tikzpicture}
\end{center}
\caption{Model complexes for direct summands appearing in the mapping cone $\Cone(C_1\otimes \F[U,U^{-1}] \xrightarrow{Q(1+\iota_K)} Q(C_1 \otimes \F[U,U^{-1}])[-1])$ . These complexes are labeled $D_1$, $D_1$, $D_3$ in order down the page.} \label{fig:summands}
\label{fig:ConeComplexes}
\end{figure}

This final model complex is acyclic and its tensor product with $\mathbb F[U,U^{-1}]$ gives no contribution to the homology of the mapping cone. Hence, the homology of the mapping cone is generated as a module over $\F[U,U^{-1}]$ by $[\theta]$ and $[Qz_0+z_1^1]=Q[\theta]$.

Now we compute the homology of the mapping cone $AI_0^-$. Recall that $A_0^-$ is generated by the elements \[\{e, a, b, c, z_0, z^s_1, Uz^s_2, Uz_3^s, U^2z^s_4, \dots, U^{n(K)}z^s_{v}\}\] for $s=1,2$ as usual. Therefore $AI_0^-$ contains the product of $U^{m}$ for $m\geq n(K)$ with all three of the model complexes $D_1$, $D_2$, and $D_3$ appearing in Figure \ref{fig:ConeComplexes}; for lower powers we must consider appropriate truncations of these complexes. 

We start with $D_1$. We see that $\left(U^mD_1\right) \cap AI_0^- = U^mD_1$ for $m \geq n(K)$, and has homology $U^m[\theta]$. To obtain $(U^{n(K)-1}D_1)\cap AI_0^-$, we delete the term $z^1_v+z^2_v$ and multiply all other elements by $U^{n(K)-1}$ after which we have an acyclic complex. To obtain $(U^{n(K)-2}D_2) \cap AI_0^-$, we additionally delete the terms $z^1_{v-1}+z^2_{v-1}$ and $z^1_{v-2}+z^2_{v-2}$, and multiply all other elements by $U^{n(K)-2}$, after which we still have an acyclic complex. This pattern continues until we reach $(UD_1)\cap AI_0^1$, which has generators \[\{Qc, QUe, Uz_0, Uz_1^1+Uz_1^2, Uz_2^1+Uz_2^2, Uz^1_3+Uz_3^2\}\] as an $\mathbb F$-vector space and is acyclic. If we then consider $D_1\cap AI_0^-$, which is generated by $\{Qe, z_0, z_1^1+z_1^2\}$, we see this truncation has homology $[Qe]$. So the model complex $D_1$ contributes a summand $\mathbb F_{(1)}[U]$ generated over $\F[U]$ by $[U^{n(K)}\theta]$ and a summand $\mathbb F_{(2n(K))}$ with $\F$-basis $[Qe]$ to the homology $H_*(AI_0^-)$ considered as an $\F[U]$-module.

We now consider $D_2$. As previously, $(U^mD_2) \cap AI_0^- = U^mD_2$ for $m \geq n(K)$, and has homology $U^m[Qz_0+z_1^1]=Q[\theta]$. One may easily check that this is unchanged by successive truncations; that is, that the homology of $(U^{\ell}D_2) \cap AI_0^-$ for $0\leq \ell \leq n(K)-1$ remains $U^{\ell}[Qz_0+z_1^1]$. The model complex $D_2$ therefore contributes a summand $\mathbb F_{(2n(K))}[U]$ generated over $\F[U]$ by $[Qz_0+z_1^1]$ to the homology $H_*(AI_0^-)$ considered as an $\F[U]$-module. 

Finally, we consider $D_3$. We see that $(U^{m}D_3) \cap AI_0^- = U^mD_3$ is acyclic for all $m \geq n(K)$. However, consider $(U^{n(K)-1}D_3) \cap AI_0^-$, which we may obtain by deleting $Qz_v^1$ and multiplying all the remaining elements by $U^{n(K)-1}$. We see this complex now has homology generated by \[[U^{n(K)}e]=[U^{n(K)-1}Qz_1^1] = \dots = [U^{n(K)-1}Qz_{v-1}^1].\] We see the same behavior under successive truncations; that is, the homology of $(U^{\ell}D_3) \cap AI_0^-$ for $-1\leq \ell \leq n(K)-1$ is generated by $U^{\ell+1}[e]$. Special mention is due of the case $(U^{-1}D_3) \cap AI_0^{-1}$, since this is the only case in the computation in which the product a negative power of $U$ with any of the three model complexes has nontrivial intersection with $AI_0^{-}$, namely $\{e\}$. In total the model complex $D_3$ contributes a summand $\mathbb F[U]_{(2n(K))}/(U^{n(K)+1})$ to the homology $H_*(AI_0^-)$ considered as an $\F[U]$-module, generated by the element $[e]$.

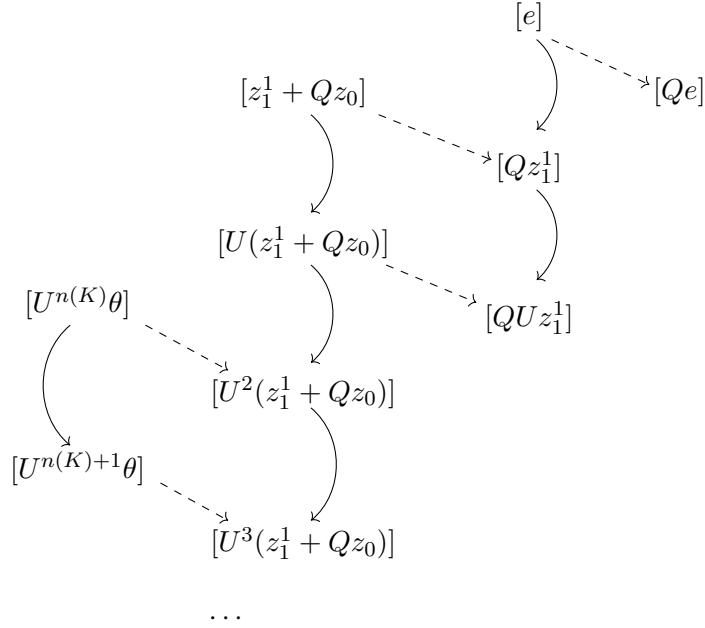
\begin{figure}
\centering
    \begin{tikzpicture}
    \node(0)at(7,5){$[Qe]$};
    \node(1)at(-1,2.2){$[U^{n(K)}\theta]$};
    \path[->][bend right = 50](-1.1,1.9)edge(-1.1,.3);
    \path[->][dashed](-.1,1.9)edge(1,1.3);
    \node(2)at(-1,0){$[U^{n(K)+1}\theta]$};
    \path[->][dashed](.1,-0.2)edge(1,-.7);
    \node(5)at(2,1){$[U^2(z_1^1+Qz_0)]$};
    \path[->][bend left = 50](2.1,.8)edge(2.1,-.7);
    \node(6)at(2,-1){$[U^3(z_1^1+Qz_0)]$};
    \node(8)at(1,-2){$\bf \cdots$};
    \node (9) at (2,3){$[U(z_1^1+Qz_0)]$ };
    \path[->][bend left = 50](2.1,2.7)edge(2.1,1.4);
    \node(10) at (2,5){$[z_1^1+Qz_0]$};
     \path[->][dashed](3,4.7)edge(4.5,4.1);
     \path[->][dashed](3.1,2.7)edge(4.3,2.2);
    \path[->][bend left = 50](2.1,4.7)edge(2.1,3.4);
    \node(4)at(5,6){$[e]$};
    \path[->][bend left = 50](5.1,5.7)edge(5.1,4.5);
    \path[->][dashed](5.3,5.7)edge(6.6,5.1);
    \node(5)at(5,4){$[Qz_1^1]$};
    \path[->][bend left = 50](5.1,3.7)edge(5.1,2.5);
    \node(6)at(5,2){$[QUz_1^1]$};
    \end{tikzpicture}
\caption{The homology of the complex $AI_0^-$ associated to $C_1\otimes F[U,U^{-1}]$. Curved lines denote the action of the variable $U$ and dashed lines denote the action of the variable $Q$. The example of $n(K)=2$, as in the case of $P(-2,5,5)$ is shown. The element $[U^{n(K)}\theta]$ lies in homological grading $1$, the elements $[z_1^1+Qz_0]$ and $Qe$ lie in homological grading $2n(K)$, and the element $[e]$ lies in homological grading $2n(K)+1$.}
\label{fig:C1final}
\end{figure}

To sum up, we conclude that as an $\F[U]$ module, the homology of the mapping cone is
\[
H_*(AI_0^-) \simeq \F_{(1)}[U] \oplus \F_{(2n(K))}[U] \oplus \F_{(2n(K)+1)}[U]/(U^{n(K)+1}) \oplus \F_{(2n(K))}.
\]
\noindent The summand $\F_{(1)}[U]$ is generated by $[U^{n(K)} \theta]$ and the summand $\F_{(2n(K))}[U]$ is generated by $[z_1^1+Qz_0]$, which notably has the property that $U^{n(K)}[z_1^1+Qz_0] = [QU^{n(K)}\theta]$. As for the torsion summands, the summand $\F_{(2n(K)+1)}[U]/(U^{n(K)+1})$ is generated by $[e]$, which notably has the property that $U[e]=Q[z_1^1]$, and finally the summand $\F_{(2n(K))}$ is generated by $[Qe]$. This module appears in Figure~\ref{fig:C1final}. This implies that
\begin{align*}
\Vl_0(K) &= -\frac{1}{2}(\gr([\theta])-1)\\
		&= -\frac{1}{2}(1-1) \\
		&=0
\end{align*}
\noindent and
\begin{align*}
\Vu_0(K) = -\frac{1}{2}(\gr([z_1^1+Qz_0])) = -\frac{1}{2}(2n(K)) =-n(K) =-\frac{m+n-2}{4}.
\end{align*}

\begin{figure}[htb] 
\begin{center}
\scalebox{.75}{
\begin{tikzpicture}\tikzstyle{every node}=[font=\tiny] 
\path[->][dotted](0,-5)edge(0,5);
\path[->][dotted](-5,0)edge(5,0);

\node (1) at (-.25,4.75){$i$};
\node(2) at (4.75,-.25){$j$};

\fill(-3,2)circle [radius=2pt];
\fill(-2,2)circle [radius=2pt];
\fill(0,-2)circle [radius=2pt];
\fill(0,-1)circle [radius=2pt];
\fill(-1,-1)circle [radius=2pt];
\fill(-1,0)circle [radius=2pt];
\fill(-2,0)circle [radius=2pt];
\fill(2,-2)circle [radius=2pt];
\fill(2,-3)circle [radius=2pt];

\fill(-0.15,-0.15)circle [radius=2pt];
\fill(-0.15,-0.85)circle [radius=2pt];
\fill(-0.85,-0.15)circle [radius=2pt];
\fill(-0.85,-0.85)circle [radius=2pt];

\node(y1)at (2.8,-2.2){$z_v^2$};
\node(y2)at (3.2,-0.8){$z_{v-1}^2$};
\node(y3)at (-0.8,3.2){$z_{v-1}^1$};
\node(y4)at (-2.2,2.8){$z_v^1$};
\node(x11)at (0.8,-1.2){$z_2^2$};
\node(x33)at (-1.2,0.8){$z_2^1$};
\node(x1331)at (0.2,-0.2){$z_0$};
\node(x1223)at (1.2,-0.1){$z_1^2$};
\node(x3221)at (0,1.3){$z_1^1$};
\node(x13)at (0.3,0.3){$e$};
\node(x22)at (-0.3,-0.3){$a$};

\node(y1)at (1.7,-3.2){$Uz_v^2$};
\node(y4)at (-3.2,1.7){$Uz_v^1$};

\path[->](2,-2.1)edge(2,-2.9);
\path[->](-2.1,2)edge(-2.9,2);
\path[->](1.9,-2)edge(0.1,-2);
\path[->](-2,1.9)edge(-2,0.1);
\path[->](0,-1.1)edge(0,-1.9);
\path[->](-1.1,0)edge(-1.9,0);
\path[->](-0.1,-1)edge(-0.9,-1);
\path[->](-1,-0.1)edge(-1,-0.9);
\path[->](-0.25,-0.85)edge(-0.75,-0.85);
\path[->](-0.25,-0.15)edge(-0.75,-0.15);
\path[->](-0.85,-0.25)edge(-0.85,-0.75);
\path[->](-0.15,-0.25)edge(-0.15,-0.75);


\fill(-2.0,3.0)circle [radius=2pt];
\fill(-1.0,3.0)circle [radius=2pt];
\fill(1.0,-1.0)circle [radius=2pt];
\fill(1.0,-0.0)circle [radius=2pt];
\fill(-0.0,-0.0)circle [radius=2pt];
\fill(-0.0,1.0)circle [radius=2pt];
\fill(-1.0,1.0)circle [radius=2pt];
\fill(3.0,-1.0)circle [radius=2pt];
\fill(3.0,-2.0)circle [radius=2pt];

\fill(0.85,0.85)circle [radius=2pt];
\fill(0.85,0.15)circle [radius=2pt];
\fill(0.15,0.85)circle [radius=2pt];
\fill(0.15,0.15)circle [radius=2pt];

\node(x12)at (0,-0.7){$b$};
\node(x13)at (-.6,-0.7){$Ue$};
\node(x21)at (-.7,-0.3){$c$};

\path[->](3.0,-1.1)edge(3.0,-1.9);
\path[->](-1.1,3.0)edge(-1.9,3.0);
\path[->](2.9,-1.0)edge(1.1,-1.0);
\path[->](-1.0,2.9)edge(-1.0,1.1);
\path[->](1.0,-0.1)edge(1.0,-0.9);
\path[->](-0.1,1.0)edge(-0.9,1.0);
\path[->](0.9,-0.0)edge(0.1,-0.0);
\path[->](-0.0,0.9)edge(-0.0,0.1);
\path[->](0.75,0.15)edge(0.25,0.15);
\path[->](0.75,0.85)edge(0.25,0.85);
\path[->](0.15,0.75)edge(0.15,0.25);
\path[->](0.85,0.75)edge(0.85,0.25);


\fill(-4.0,1.0)circle [radius=2pt];
\fill(-3.0,1.0)circle [radius=2pt];
\fill(-1.0,-3.0)circle [radius=2pt];
\fill(-1.0,-2.0)circle [radius=2pt];
\fill(-2.0,-2.0)circle [radius=2pt];
\fill(-2.0,-1.0)circle [radius=2pt];
\fill(-3.0,-1.0)circle [radius=2pt];
\fill(1.0,-3.0)circle [radius=2pt];
\fill(1.0,-4.0)circle [radius=2pt];

\fill(-1.15,-1.15)circle [radius=2pt];
\fill(-1.15,-1.85)circle [radius=2pt];
\fill(-1.85,-1.15)circle [radius=2pt];
\fill(-1.85,-1.85)circle [radius=2pt];

\path[->](1.0,-3.1)edge(1.0,-3.9);
\path[->](-3.1,1.0)edge(-3.9,1.0);
\path[->](0.9,-3.0)edge(-0.9,-3.0);
\path[->](-3.0,0.9)edge(-3.0,-0.9);
\path[->](-1.0,-2.1)edge(-1.0,-2.9);
\path[->](-2.1,-1.0)edge(-2.9,-1.0);
\path[->](-1.1,-2.0)edge(-1.9,-2.0);
\path[->](-2.0,-1.1)edge(-2.0,-1.9);
\path[->](-1.25,-1.85)edge(-1.75,-1.85);
\path[->](-1.25,-1.15)edge(-1.75,-1.15);
\path[->](-1.85,-1.25)edge(-1.85,-1.75);
\path[->](-1.15,-1.25)edge(-1.15,-1.75);
\end{tikzpicture}
}
\end{center}
\caption{The dual $\overline C_1$ of the model complex $C_1$.}
\label{fig:C1dual}
\end{figure}
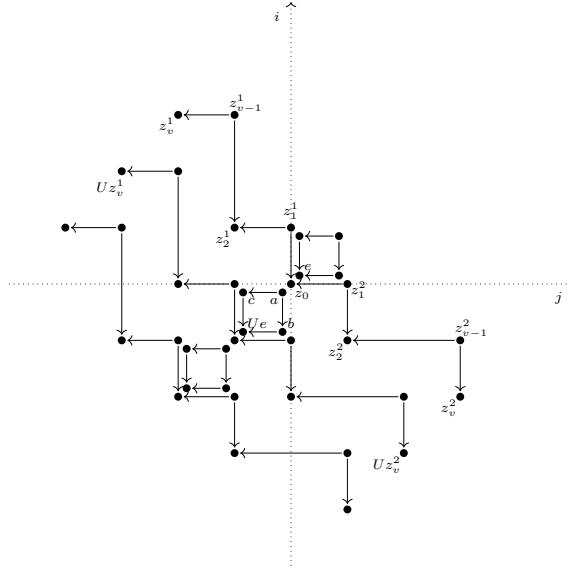

We now consider the dual $\overline{C}_1$ of the model complex $C_1$. This complex, which appears in Figure \ref{fig:C1dual}, has generators $a,b,c,e,z_0$ and $z_r^s$ for $1\leq r\leq v$ and $s \in \{1,2\}$. The homological gradings are determined by the requirement that $H_*(C\{i=0\})\simeq \mathbb F_{(0)}$, which specifies that $\gr(U^{-n(K)}z_v^1)=0$, which implies that for example $\gr(z_0)=-2n(K)$. The nonzero differentials are as follows:
 \begin{alignat*}{3}
        &\partial(a)=b+c  & \qquad\qquad & \qquad\qquad &      &\partial(b)=\partial(c) =Ue \\
        &\partial(z_r^s)=z_{r-1}^s+z_{r+1}^s, \ r\text{ is odd}  & & &      &\
\end{alignat*} 
Up to change of basis, the unique involution $\iota_K$ squaring to the Sarkar map is given by:
  \begin{alignat*}{3}
        &\iota_K(a)=a+z_0  & \qquad\qquad & \qquad\qquad &      &\iota_K(b)=c \\
        &\iota_K(z_0)=z_0+e  & & &      &\iota_K(c)=b  \\
        &\iota_K(z_1^1)=z_1^2+U^{-1}b & & & & \iota_K(z_1^2)=z_1^1+U^{-1}c \\
        & \iota_K(z_r^1)=z_r^2, \ 1<r\leq v & & & & \iota_K(z_r^2)=z_r^1, \ 1\leq r \leq v \\
        &\iota_K(e)=e & & & & \ 
    \end{alignat*}
We begin by considering the $\mathbb F[U]$-complex of $\overline{C}_1\otimes \mathbb F[U,U^{-1}]$ lying in the third quadrant, again called by convention $A_0^-$, and shown in Figure~\ref{fig:DualAMinus}. We see this complex is generated over $\mathbb F[U]$ by the set
\[\{a, b, c, e, z_0, Uz_1^s, Uz_2^s, U^2z_3^s, \dots, U^{n(K)-1}z_{v-3}^s, U^{n(K)-1}z_{v-2}^s, U^{n(K)+1}z_{v-1}^s, U^{n(K)+1}z_{v}^s\}\]
where $s=1,2$ as usual, and has differentials inherited from the differentials on $\overline{C}_1$. The homology of $A_0^-$ is isomorphic to
$$\F_{(-2n(K))}\oplus\F_{(-2n(K))}[U]$$
where the $\F_{(-2n(K))}$ summand is has an $\F$-basis $[e]$ and the $\F_{(-2n(K))}[U]$ summand is generated over $\F[U]$ by $[z_0]$. We therefore have $V_0(\overline{K})=n(K) = \frac{m+n-2}{4}$.

\begin{figure}[htb] 
\begin{center}
\scalebox{.75}{
\begin{tikzpicture}\tikzstyle{every node}=[font=\tiny] 
\path[->][dotted](0,-6)edge(0,2);
\path[->][dotted](-6,0)edge(2,0);

\fill(0,-2)circle [radius=2pt];
\fill(0,-1)circle [radius=2pt];
\fill(-1,-1)circle [radius=2pt];
\fill(-1,0)circle [radius=2pt];
\fill(-2,0)circle [radius=2pt];

\fill(-0.15,-0.15)circle [radius=2pt];
\fill(-0.15,-0.85)circle [radius=2pt];
\fill(-0.85,-0.15)circle [radius=2pt];
\fill(-0.85,-0.85)circle [radius=2pt];

\node(x1331)at (0.2,-0.2){$z_0$};
\node(x13)at (0.3,0.3){$e$};

\path[->](0,-1.1)edge(0,-1.9);
\path[->](-1.1,0)edge(-1.9,0);
\path[->](-0.1,-1)edge(-0.9,-1);
\path[->](-1,-0.1)edge(-1,-0.9);
\path[->](-0.25,-0.85)edge(-0.75,-0.85);
\path[->](-0.25,-0.15)edge(-0.75,-0.15);
\path[->](-0.85,-0.25)edge(-0.85,-0.75);
\path[->](-0.15,-0.25)edge(-0.15,-0.75);


\fill(-0.0,-0.0)circle [radius=2pt];

\fill(0.15,0.15)circle [radius=2pt];

\node(x22)at (-0.3,-0.3){$a$};
\node(x12)at (0,-0.7){$b$};
\node(x13)at (-.6,-0.7){$Ue$};
\node(x21)at (-.7,-0.3){$c$};
\node(x11)at (-0.3,-2.2){$Uz_2^2$};
\node(x33)at (-2.3,-0.2){$Uz_2^1$};
\node(x1223)at (0.3,-1.1){$Uz_1^2$};
\node(x3221)at (-1,0.2){$Uz_1^1$};


\fill(-1.0,-3.0)circle [radius=2pt];
\fill(-1.0,-2.0)circle [radius=2pt];
\fill(-2.0,-2.0)circle [radius=2pt];
\fill(-2.0,-1.0)circle [radius=2pt];
\fill(-3.0,-1.0)circle [radius=2pt];

\fill(-1.15,-1.15)circle [radius=2pt];
\fill(-1.15,-1.85)circle [radius=2pt];
\fill(-1.85,-1.15)circle [radius=2pt];
\fill(-1.85,-1.85)circle [radius=2pt];

\path[->](-1.0,-2.1)edge(-1.0,-2.9);
\path[->](-2.1,-1.0)edge(-2.9,-1.0);
\path[->](-1.1,-2.0)edge(-1.9,-2.0);
\path[->](-2.0,-1.1)edge(-2.0,-1.9);
\path[->](-1.25,-1.85)edge(-1.75,-1.85);
\path[->](-1.25,-1.15)edge(-1.75,-1.15);
\path[->](-1.85,-1.25)edge(-1.85,-1.75);
\path[->](-1.15,-1.25)edge(-1.15,-1.75);


\fill(-5.0,0.0)circle [radius=2pt];
\fill(-4.0,0.0)circle [radius=2pt];
\fill(-2.0,-4.0)circle [radius=2pt];
\fill(-2.0,-3.0)circle [radius=2pt];
\fill(-3.0,-3.0)circle [radius=2pt];
\fill(-3.0,-2.0)circle [radius=2pt];
\fill(-4.0,-2.0)circle [radius=2pt];
\fill(0.0,-4.0)circle [radius=2pt];
\fill(0.0,-5.0)circle [radius=2pt];

\fill(-2.15,-2.15)circle [radius=2pt];
\fill(-2.15,-2.85)circle [radius=2pt];
\fill(-2.85,-2.15)circle [radius=2pt];
\fill(-2.85,-2.85)circle [radius=2pt];

\node(y2)at (0.2,-3.8){$U^{n(K)+1}z_{v-1}^2$};
\node(y3)at (-3.8,0.2){$U^{n(K)+1}z_{v-1}^1$};
\node(y1)at (-0.6,-5.3){$U^{n(K)+1}z_v^2$};
\node(y4)at (-5.2,-0.3){$U^{n(K)+1}z_v^1$};

\path[->](0.0,-4.1)edge(0.0,-4.9);
\path[->](-4.1,0.0)edge(-4.9,0.0);
\path[->](-0.1,-4.0)edge(-1.9,-4.0);
\path[->](-4.0,-0.1)edge(-4.0,-1.9);
\path[->](-2.0,-3.1)edge(-2.0,-3.9);
\path[->](-3.1,-2.0)edge(-3.9,-2.0);
\path[->](-2.1,-3.0)edge(-2.9,-3.0);
\path[->](-3.0,-2.1)edge(-3.0,-2.9);
\path[->](-2.25,-2.85)edge(-2.75,-2.85);
\path[->](-2.25,-2.15)edge(-2.75,-2.15);
\path[->](-2.85,-2.25)edge(-2.85,-2.75);
\path[->](-2.15,-2.25)edge(-2.15,-2.75);

\node(ellipsis1)at(-5,-3){$\iddots$};
\node(ellipsis2)at(-4,-4){$\iddots$};
\node(ellipsis3)at(-3,-5){$\iddots$};
\end{tikzpicture}
}
\end{center}
\caption{The $\F[U]$-submodule $A_0^-$ of $\overline C_1 \otimes \mathbb F[U,U^{-1}]$.}
\label{fig:DualAMinus}
\end{figure}

\noindent Now consider the mapping cone $\Cone((\overline{C}_1\otimes \F[U,U^{-1}]) \xrightarrow{Q(1+\iota_K)} Q(\overline{C}_1\otimes \F[U,U^{-1}])[-1])$. This mapping cone comes to us with a set of generators \[\{e, a, b, c, z_0, z^s_1,\dots, z^s_{v}, Qe, Qa, Qb, Qc, Qz_0, Qz^s_1, \dots, Qz^s_{v}\}\] for $s=1,2$ as usual. We do a change of basis to replace this by the set of generators 
\begin{align*}&\{e, a, b, b+c, z_0, z^1_1, z_1^1+z_1^2 ,\dots, z_1^v, z^1_{v}+z^2_v, \\ &\qquad Qe, Qa, Qb, Q(b+c), Qz_0, Qz^1_1, Q(z_1^1+z_1^2), \dots, Qz^1_{v}, Q(z^1_v+z_v^2)\}.\end{align*}
As in the previous case, this breaks up into direct summands over $\F[U,U^{-1}]$ each of which is the tensor product of a model $\mathbb F$-complex over $\F[U,U^{-1}]$. These model complexes are shown in Figure~\ref{fig:DualConeComplexes}. The first, $E_1$, has generators
\[\{a,b+c,Qz_0, Qz_1^1, \dots, Qz_v^1\}\]
and nonzero differentials $\partial^{\iota} = \partial + Q(1+\iota_K)$
\begin{alignat*}{3}
        &\partial^{\iota}(a) = (b+c)+Qz_0   \\
         & \partial^{\iota}(Qz_r^1)=Qz_{r-1}^1+Qz_{r+1}^1, \ \ 1 \leq r < v, \ r\text{ odd.}
  \end{alignat*}
The homology of the complex $E_1$ is one-dimensional with basis $[b+c]=[Qz_0]$. The tensor product $E_1 \otimes \F[U,U^{-1}]$ therefore contributes an $\mathbb F[U,U^{-1}]$ summand to the homology of the mapping cone generated by $[b+c]=[Qz_0]$.

Next we consider the model complex $E_2$, which has generators
\[ \{ Qe, QU^{-1}c, z_0, z_1^1, \dots, z_v^1, Qz_1^1+Qz_1^2, \dots, Qz_v^1, Qz_v^2\}\]
and nonzero differentials
\begin{alignat*}{3}
        & \partial^{\iota}(QU^{-1}c)= Qe \\
        & \partial^{\iota}(z_0)= Qe \\
        & \partial^{\iota}(z_1^1)= z_0 + z^1_2 + Q(z_1^1 + z^2_1) + QU^{-1}c \\
         &\partial^{\iota}(z_r^1)=\begin{cases}(Q(z_r^1+z_r^2), \ \ 1<r \leq v, \ r\text{ even}\\ z^1_{r-1} + z^1_{r+1} + Q(z_r^1+z_r^2), \ \ 1<r < v, \ r\text{ odd}\end{cases} \\
         &\partial^{\iota}(Q(z_r^1+z_r^2))= Q(z_{r-1}^1+z_{r-1}^2)+Q(z_{r+1}^1+z_{r+1}^2), \ \ 1<r<v,  \ r\text{ odd}
  \end{alignat*}
The homology of this complex is one-dimensional with basis $[z_0+QU^{-1}c]$. So, the tensor product $E_1\otimes \F[U,U^{-1}]$ contributes an $\mathbb F[U,U^{-1}]$ summand to the homology of the mapping cone.

Finally we consider the third model complex $E_3$, which has generators
\[\{e, U^{-1}c, QU^{-1}a, QU^{-1}(b+c), z_1^1+z_1^2, \dots, z_v^1+z_v^2\}\]
and nonzero differentials 
\begin{alignat*}{3}
        &\partial^{\iota}(QU^{-1}a)= QU^{-1}(b+c)  \\
        & \partial^{\iota}(U^{-1}c) = e+QU^{-1}(b+c) \\
                 & \partial^{\iota}(z_1^1+z_1^2)=z_2^1 + z_2^2 + QU^{-1}(b+c)\\
          & \partial^{\iota}(z_r^1+z_r^2)= (z^1_{r-1} + z^2_{r-1}) + (z^1_{r+1} + z^2_{r+1}), \ \ 1<r <v, \ r\text{ odd}
  \end{alignat*}
\noindent This final model complex is acyclic and therefore $E_3 \otimes \F[U,U^{-1}]$ does not contribute anything to the total mapping cone.

\begin{figure}[htb] 
\begin{center}
\begin{tikzpicture}\tikzstyle{every node}=[font=\tiny]

\node(Qb)at(6.0,-0.0){$b+c$};
\node(QUe)at(4.5,1.5){$a$};
\node(z0)at(3.0,-0.0){$Qz_0$};
\node(zs)at(1.5,1.5){$Qz_1^1$};
\node(zs)at(-0.0,-0.0){$Qz_2^1$};
\node(zs)at(-1.5,1.5){$Qz_3^1$};
\node(dots)at(-3.0,-0.0){$\dots$};
\node(zs)at(-4.5,1.5){$Qz_{v-1}^1$};
\node(zs)at(-6.0,-0.0){$Qz_v^1$};

\path[->](4.8,1.2)edge(5.7,0.3);
\path[->](4.2,1.2)edge(3.3,0.3);
\path[->](1.8,1.2)edge(2.7,0.3);
\path[->](1.2,1.2)edge(0.3,0.3);
\path[->](-1.2,1.2)edge(-0.3,0.3);
\path[->](-1.8,1.2)edge(-2.55,0.45);
\path[->](-4.2,1.2)edge(-3.45,0.45);
\path[->](-4.8,1.2)edge(-5.7,0.3);
\end{tikzpicture}
\end{center}
\vspace{5mm}
\begin{center}
\begin{tikzpicture}\tikzstyle{every node}=[font=\tiny]

\node(nd)at(4.5,-1.5){$Qe$};
\node(nd)at(5,-0.375){$QU^{-1}c$};
\node(nd)at(3.75,-0.375){$z_0$};
\node(nd)at(2.8,-0.375){$Q(z_1^1+z_1^2)$};
\node(nd)at(3.0,0.75){$z_1^1$};
\node(nd)at(1.5,-1.6){$Q(z_2^1+z_2^2)$};
\node(nd)at(1.5,-0.375){$z_2^1$};
\node(nd)at(-0.0,-0.375){$Q(z_3^1+z_3^2)$};
\node(nd)at(-0.0,0.75){$z_3^1$};
\node(nd)at(-1.5,-1.6){$Q(z_4^1+z_4^2)$};
\node(nd)at(-1.5,-0.375){$z_4^1$};
\node(nd)at(-3.0,-0.375){$\dots$};
\node(nd)at(-4.2,-0.375){$Q(z_{v-1}^1+z_{v-1}^2)$};
\node(nd)at(-3.0,0.75){$\dots$};
\node(nd)at(-4.5,0.75){$z_{v-1}^1$};
\node(nd)at(-6.0,-1.5){$Q(z_v^1+z_v^2)$};
\node(nd)at(-6.0,-0.375){$z_v^1$};

\path[->](4.9,-0.575)edge(4.6,-1.35);
\path[->](3.0,0.45)edge(3.0,-0.2);
\path[->](1.5,-0.55)edge(1.5,-1.2);
\path[->](-0.0,0.45)edge(-0.0,-0.2);
\path[->](-1.5,-0.55)edge(-1.5,-1.2);
\path[->](-4.5,0.45)edge(-4.5,-0.2);
\path[->](-6.0,-0.55)edge(-6.0,-1.2);
\path[->](3.3,0.6)edge(4.6,-.1);
\path[->](2.7,0.6)edge(1.65,-0.05);
\path[->](0.3,0.6)edge(1.35,-0.05);
\path[->](-0.3,0.6)edge(-1.35,-0.05);
\path[->](-2.7,0.6)edge(-1.65,-0.05);
\path[->](-4.8,0.6)edge(-5.85,-0.05);
\path[->](2.7,-0.6)edge(1.65,-1.35);
\path[->](0.3,-0.6)edge(1.35,-1.35);
\path[->](-0.3,-0.6)edge(-1.35,-1.35);
\path[->](-2.7,-0.6)edge(-1.65,-1.35);
\path[->](-4.8,-0.6)edge(-5.85,-1.35);
\path[->](3.9,-0.575)edge(4.35,-1.35);
\path[->](3.15,0.6)edge(3.6,-0.15);
\end{tikzpicture}
\begin{center}
\vspace{5mm}
\end{center}
\begin{tikzpicture}\tikzstyle{every node}=[font=\tiny]

\node(Qa)at(4.5,-0.0){$e$};
\node(Qbc)at(3.0,1.5){$U^{-1}c$};
\node(b)at(1.5,-0.0){$QU^{-1}(b+c)$};
\node(Ue)at(1.5,1.5){$QU^{-1}a$};
\node(Qz11)at(-0.0,1.5){$z_1^1+z_1^2$};
\node(Qz21)at(-1.5,-0.0){$z_2^1+z_2^2$};
\node(dots)at(-3.0,1.5){$\dots$};
\node(Qzv1)at(-4.5,-0.0){$z_v^1+z_v^2$};

\path[->](3.3,1.2)edge(4.2,0.3);
\path[->](2.7,1.2)edge(1.8,0.3);
\path[->](1.5,1.2)edge(1.5,0.3);
\path[->](0.3,1.2)edge(1.2,0.3);
\path[->](-0.3,1.2)edge(-1.2,0.3);
\path[->](-2.7,1.2)edge(-1.8,0.3);
\path[->](-3.3,1.2)edge(-4.2,0.3);
\end{tikzpicture}
\end{center}
\caption{Model complexes for direct summands appearing in the mapping cone $\Cone(\overline C_1\otimes \F[U,U^{-1}] \xrightarrow{Q(1+\iota_K)} Q(\overline C_1 \otimes \F[U,U^{-1}])[-1])$. From top to bottom, the model complexes are denoted $E_1$, $E_2$, and $E_3$.}
\label{fig:DualConeComplexes}
\end{figure}
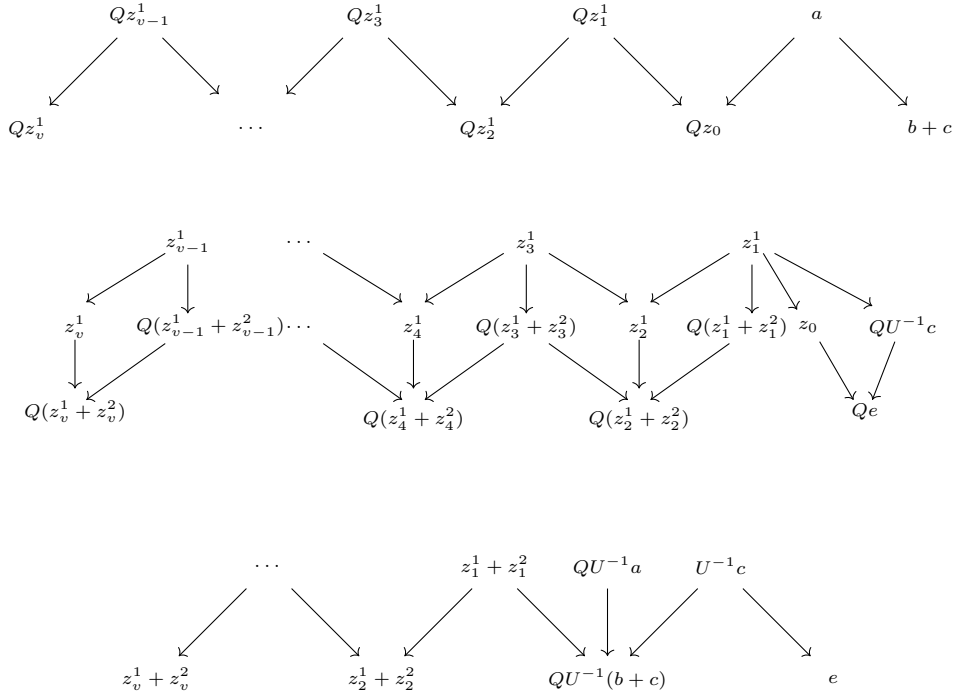

We now consider the homology of the mapping cone $AI_0^-$. We start by considering the intersection $(E_1 \otimes \F[U,U^{-1}])\cap AI_0^-$. We note that $(U^{m}E_1) \cap AI_0^- = U^mE_1$ for all $m\geq n(K)+1$, and has one-dimensional homology with basis $[U^m(b+c)]=[U^mQz_0]$. Subsequently lowering the power of $U$ by one has at each step the effect of either preserving the complex or of removing pairs of generators $z^1_k$ and $z^1_{k-1}$ for $k$ even (and changing the exponent on $U$); in particular, none of these things alters the fact that the homology of $(U^{\ell}D_1) \cap AI_0^-$ is one-dimensional with basis $[U^{\ell}Qz_0]=[U^{\ell}(b+c)]$ for $\ell \geq 0$. So, $(E_1 \otimes \F[U,U^{-1}])\cap AI_0^-)$ contributes a summand $\mathbb F_{-2n(K)}[U]$ generated by $[Qz_0]$ to $H_*(AI_0^-)$ considered as an $\mathbb F[U]$-module.

We now consider the intersection $(E_2 \otimes \F[U,U^{-1}])\cap AI_0^-$. We observe that the intersection $(U^mE_2)\cap AI_0^- = U^mE_2$ for all $m\geq n(K)+1$, and has one-dimensional homology with basis $U^m[z_0+QU^{-1}c]$. The intersection $(U^{n(K)}E_2) \cap AI_0^-$ is obtained by truncating the four elements $z_v^1,Q(z_v^1+z_v^2), z^1_{v-1}, Q(z_{v-1}^1+z_{v-1}^2)$ and multiplying everything by $U^{n(K)+1}$; this does not change the homology of the complex except for the power of $U$; it is still generated over $\F$ by $U^{n(K)}[z_0+QU^{-1}c]$. Successive truncations either preserve the complex from the previous step or delete four elements $z_k^1,Q(z_k^1+z_k^2), z_k^1, Q(z_k^1+z_k^2)$ for $k$ even (and change the power of $U$); none of this changes the homology of $(U^{\ell}E_2)\cap AI_0^-$ apart from the power of $U$, and it continues to be generated by $U^{\ell}[z_0+QU^{-1}c]$ for $\ell \geq 1$. For the final nontrivial intersection $E_2 \cap AI_0^-$, we are left with a complex generated by $z_0$ and $e$, which is acyclic. So the intersection $(E_2 \otimes \F[U,U^{-1}])\cap AI_0^-$ contributes a summand $\mathbb F_{-2n(K)-1}[U]$ generated by $[Uz_0+Qc]$ to the homology $H_*(AI_0^-)$ considered as an $\mathbb F[U]$-module.

Finally we consider the intersection $(E_3 \otimes \F[U,U^{-1}])\cap AI_0^-$. For $m \geq n(K)+1$, the intersection $(U^{m}E_3) \cap AI_0^- = U^mE_3$ and is acyclic. To obtain $(U^{n(K)}E_3) \cap AI_0^-$, we truncate $z_v^1+z^2_v$ and $z_{v-1}^1+z^2_{v-1}$ and multiply the remaining basis elements by $U^{n(K)}$, obtaining a complex which remains acyclic. Successive truncations either preserve the complex up to changing the power of $U$ or remove pairs of elements $z_k^1+z_k^2$ and $z_{k-1}^1+z^2_{k-1}$ for $k$ even, leaving the complex acyclic. This persists until we reach $(U^0E_3) \cap AI_0^- = E_3 \cap AI_0^-$, which consists solely of the element $e$, and has one dimensional homology over $\F$ generated by the element $[e]$. So $(E_3 \otimes \F[U,U^{-1}])\cap AI_0^-$ contributes a summand $\F_{(-2n(K)+1}$ with basis $[e]$ to the homology of $AI_0^-$ considered as an $\F[U]$-module.

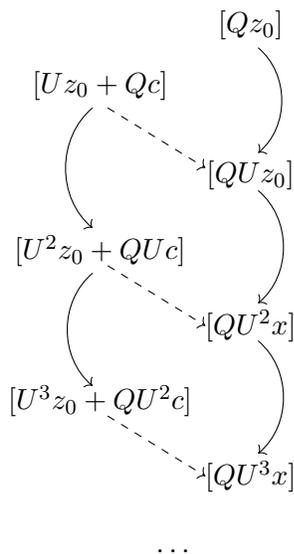
\begin{figure}
\centering
   \begin{tikzpicture}
    \node(0)at(4,2){$[e]$};
    \node(1)at(0,0.2){$[Uz_0 + Qc]$};
    \path[->][bend right = 50](-.1,-.1)edge(-.1,-1.7);
    \path[->][dashed](.1,-.1)edge(1.4,-0.9);
    \node(2)at(0,-2){$[U^2z_0 + QUc]$};
    \path[->][bend right = 50](-.1,-2.3)edge(-.1,-3.8);
    \path[->][dashed](.1,-2.2)edge(1.4,-3);
    \node(3)at(0,-4){$[U^3z_0 + QU^2c]$};
    \path[->][dashed](.1,-4.2)edge(1.4,-5);
    \node(4)at(2,1){$[Qz_0]$};
    \path[->][bend left = 50](2.1,0.7)edge(2.1,-0.7);
    \node(5)at(2,-1){$[QUz_0]$};
    \path[->][bend left = 50](2.1,-1.2)edge(2.1,-2.7);
    \node(6)at(2,-3){$[QU^2x]$};
    \path[->][bend left = 50](2.1,-3.2)edge(2.1,-4.7);
    \node(7)at(2,-5){$[QU^3x]$};
    \node(8)at(1,-6){$\bf \cdots$};
    \end{tikzpicture}
\caption{The homology of the complex $AI_0^-$ associated to $\overline{C}_1\otimes\F[U,U^{-1}]$. Curved lines denote the action of the variable $U$ and dashed lines denote the action of the variable $Q$. The element $[Uz_0+Qc]$ lies in homological grading $-2n(K)-1$, the element $[Qz_0]$ lies in homological grading $-2n(K)$, and the element $[e]$ lies in homological grading $-2n(K)+1$. The case of $n(K)=2$ as for the mirror of $P(-2,5,5)$ is shown.}
\label{fig:C1dualfinal}
\end{figure}

Summing up, we see that
\[H_*(AI_0^-) \simeq\F_{(-2n(K)-1)}[U]\oplus\F_{(-2n(K))}[U] \oplus \F_{(-2n(K)+1)}\]
where the $\F_{(-2n(K)-1)}[U]$ summand is generated by $[U(z_2^1+Q(z_1^1+z_1^2))]=U[z_0+QU^{-1}c]$, the $\F_{(-2n(K))}[U]$ summand is generated by $[b+c]=[Qz_0]$, and the $\F_{(-2n(K)+1)}$ summand has basis $[e]$. This module is shown in Figure~\ref{fig:C1dualfinal}. We therefore see that 
\begin{align*}
\Vl_0(\overline{K}) &= -\frac{1}{2}\left(\gr(U[z_0+QU^{-1}c]) - 1\right)\\
					&= -\frac{1}{2}(-2n(K)-1-1)\\
					&=n(K)+1\\
					&=\frac{m+n+2}{4}
\end{align*}
and
\[
\Vu_0(\overline{K}) = -\frac{1}{2}\left(\gr([Qz_0])\right)= -\frac{1}{2}(-2n(K))=n(K)=\frac{m+n-2}{4}
\]
\noindent as promised.\end{proof}

\bibliographystyle{amsalpha}
\bibliography{bib}

\end{document}